\documentclass[12pt,a4paper]{amsart}

\usepackage{amsmath,amssymb,amsthm,amscd}
\usepackage{url}

\numberwithin{equation}{section}

\theoremstyle{plain}
\newtheorem{theorem}{Theorem}[section]
\newtheorem{lemma}[theorem]{Lemma}
\newtheorem{proposition}[theorem]{Proposition}

\newtheorem{corollary}[theorem]{Corollary}

\newtheorem{maintheorem}{Theorem}

\newtheoremstyle{neosn}
  {0.5\topsep}
  {0.5\topsep}
  {\normalfont}
  {}
  {\bfseries}
  {.}
  { }
  {\thmname{#1}\thmnumber{ #2}\thmnote{ {\mdseries #3}}}

\theoremstyle{neosn}

\newtheorem{definition}[theorem]{Definition}

\theoremstyle{plain}

\usepackage[hidelinks]{hyperref}
\hypersetup{
  pdftitle={Automorphisms of twisted Chevalley groups over commutative rings},
  pdfauthor={Elena I. Bunina and Deep H. Makadiya},
  pdfsubject={Automorphisms of twisted Chevalley groups},
  pdfkeywords={Twisted Chevalley groups, automorphisms, commutative rings,
    local-global methods, exceptional automorphisms, unitary groups, triality}
}

\begin{document}

\title{Automorphisms of twisted Chevalley groups over commutative rings}

\author{Elena I. Bunina}
\address{Department of Mathematics,
        Bar--Ilan University, Ramat Gan, Israel}
\email{helenbunina@gmail.com}

\author{Deep H. Makadiya}
\address{Department of Mathematics,
        Bar--Ilan University, Ramat Gan, Israel}
\email{deepmakadia25.dm@gmail.com}


\begin{abstract}
We classify all abstract automorphisms of the elementary and full
twisted Chevalley groups of types ${}^2A_\ell$ ($\ell\geq3$),
${}^2D_\ell$ ($\ell\geq4$), ${}^2E_6$, and ${}^3D_4$ over
arbitrary commutative rings, for every invariant isogeny type.
The coefficient automorphism $\theta$ has order dividing that of
the diagram automorphism. We assume that $2$ is invertible and,
for triality, that $3$ is invertible as well.

Three features distinguish the answer from the usual inner--diagonal
and coefficient decomposition. The elementary groups of types
${}^2A_3$ and ${}^2A_4$ admit an additive family of exceptional
automorphisms supported at $\theta$-fixed maximal ideals with
residue field $\mathbb F_3$.
In type ${}^2D_4$, diagram automorphisms may vary on idempotent
components, subject to compatibility with the involution and the
isogeny lattice. In triality, the coefficient automorphisms preserve
the trace operator $r\mapsto r+\theta(r)+\theta^2(r)$ but need not
commute with $\theta$.
We determine which elementary automorphisms extend to the full
groups and describe the remaining central factors. A nonzero
exceptional unitary factor extends precisely in the simply connected
type ${}^2A_3$, giving a classification for the hyperbolic group
$\operatorname{SU}_4(R,\theta;Q_4)$. The proof glues uniquely
determined normalization data for a finite root frame and then
recovers the coefficient action from root-subgroup relations.
\end{abstract}

\maketitle

\section*{Introduction}

An abstract automorphism of a group defined over a ring need not
preserve its root subgroups or act coefficientwise in a chosen
representation. The automorphism problem asks how much of the
defining algebraic structure is determined by the group law.
For twisted Chevalley groups, this includes a coefficient
automorphism as well as the ring.

We classify automorphisms of elementary and full twisted Chevalley
groups over arbitrary commutative rings. In most types, every
elementary automorphism is induced by an adjoint group element
and a ring automorphism commuting with the coefficient operator.
Three cases modify this answer: nilpotent exceptional parameters
in ${}^2A_3$ and ${}^2A_4$, componentwise diagram symmetries in
${}^2D_4$, and trace-preserving coefficient actions in ${}^3D_4$.
For full groups, central automorphisms account for the ambiguity
of extension, while the low-rank unitary cases also have an
extension obstruction.

\subsection*{Main results}

Let $R$ be a commutative ring with identity, let $\rho$ be the
chosen diagram automorphism, and let $\theta\in\operatorname{Aut}(R)$
satisfy
$$
        \theta^m=\operatorname{id}_R,
        \qquad m=\operatorname{ord}(\rho)\in\{2,3\}.
$$
We assume $2\in R^\times$ and, when $m=3$, also $3\in R^\times$.
The twisted types under consideration are
$$
        {}^2A_\ell\ (\ell\geq3),\qquad
        {}^2D_\ell\ (\ell\geq4),\qquad
        {}^2E_6,\qquad {}^3D_4.
$$
No locality, Noetherianity, reducedness, or connectedness assumption
is imposed on $R$; the order of $\theta$ may be smaller than $m$.

Let $\Lambda_r\subseteq\Lambda\subseteq\Lambda_w$ be a
$\rho$-invariant lattice between the root and weight lattices.
Put $\sigma=\rho\theta$. For the statements in this introduction,
write
$$
\begin{aligned}
        G_\Lambda(R)&:=G_{\Lambda,\sigma}(\Phi,R)
                     =G_\Lambda(\Phi,R)^{\rho\theta},\\
        E_\Lambda(R)&:=E_{\Lambda,\sigma}(\Phi,R),
\end{aligned}
$$
where $E_\Lambda(R)$ is generated by the relative root subgroups.
Thus $\Lambda=\Lambda_r$ gives the adjoint form and
$\Lambda=\Lambda_w$ the simply connected form.

An element $g\in G_{\mathrm{ad},\sigma}(\Phi,R)$ acts canonically
on both groups; denote this action by $i_g^\Lambda$.
It combines the usual inner and diagonal factors and is conjugation
by a group element when $g$ lifts to $G_\Lambda(R)$.
For $\mu\in\operatorname{Aut}(R,\theta)$, where
$$
        \operatorname{Aut}(R,\theta)
        :=\{\mu\in\operatorname{Aut}(R)\mid\mu\theta=\theta\mu\},
$$
write $\mu_*$ for coefficientwise action.
We describe the three modifications before stating the theorem.

\smallskip
\noindent\emph{Exceptional unitary factors.}
For an involution $\theta$, put
$R^-_\theta:=\{r\in R\mid\theta(r)=-r\}$ and
$$
\mathcal E_{\mathrm{glob}}(R,\theta)
 :=R^-_\theta\cap\operatorname{Ann}_R
   \langle 3,\ r-\theta(r),\ r^3-r\mid r\in R\rangle_R.
$$
For $n=3,4$, every $\delta$ in this module defines an automorphism
$\Theta_{\delta,\Lambda}^{(n)}$ of $E_\Lambda(R)$ in type ${}^2A_n$.
These automorphisms form an additive family; for $\delta\ne0$
they are not compositions of adjoint-induced and coefficient actions.
Every parameter satisfies $\delta^2=3\delta=0$ and vanishes after
localization outside the fixed maximal ideals with residue field
$\mathbb F_3$. The construction is given in
Section~\ref{sec:exceptional-unitary-types} and lifted through
central isogenies in
Section~\ref{sec:arbitrary-isogeny-types-full-groups}.

\smallskip
\noindent\emph{Graph factors in ${}^2D_4$.}
Here $\mathcal D=\operatorname{Aut}(\operatorname{Dyn}(D_4))\simeq S_3$
and $\rho$ is a transposition. Choose a three-cycle $\tau$, so that
$$
        \mathcal D=\langle\tau\rangle\rtimes\langle\rho\rangle,
        \qquad \rho\tau\rho^{-1}=\tau^{-1},
$$
and use $\mathcal R_\rho=\{1,\tau,\tau^2\}$ as representatives
for the left cosets of $\langle\rho\rangle$.
A graph factor is specified by pairwise orthogonal idempotents
$\mathbf e=(e_0,e_1,e_2)$ satisfying
$$
        e_0+e_1+e_2=1,\qquad
        \theta(e_j)=e_{-j},\qquad
        e_j=0\ \text{if }\tau^j(\Lambda)\ne\Lambda,
$$
with indices modulo $3$ and $e_j=e_{\tau^j}$.
Applying the pinned action $\tau^j$ on $e_jR$ defines
$\Gamma_{\mathbf e,\Lambda}$ on the full and elementary twisted
groups. These \emph{reduced compatible graph factors} form an
abelian group of exponent dividing three; see
Section~\ref{sec:graph-automorphisms-2D4}.
Powers of $\rho$ are already coefficient actions and give no
additional factor.

\smallskip
\noindent\emph{Triality coefficient factors.}
When $m=3$, put
$$
\begin{aligned}
        \operatorname{Tr}_\theta(r)&:=r+\theta(r)+\theta^2(r),\\
        \operatorname{Aut}_{\mathrm{tr}}(R,\theta)
        &:=\{\mu\in\operatorname{Aut}(R)\mid
              \mu\operatorname{Tr}_\theta
              =\operatorname{Tr}_\theta\mu\}.
\end{aligned}
$$
Every such $\mu$ induces an automorphism $\Phi_{\mu,\Lambda}$
of $E_\Lambda(R)$ by applying $\mu$ to the relative root
parameters in the coordinates fixed in Section~\ref{sec:triality}.
This need not be coefficientwise action on the ambient split group:
$\mu$ need not commute with $\theta$ or conjugate it to
$\theta^{-1}$. For triality, the invariant isogeny lattices are
$\Lambda_r$ and $\Lambda_w$.

\begin{maintheorem}[Elementary groups]
\label{thm:main-elementary}
\label{thm:automorphisms-arbitrary-elementary-twisted-groups}
\label{thm:complete-adjoint-elementary-automorphisms}
Under the preceding assumptions, every automorphism of
$E_\Lambda(R)$ has the following form:
$$
\def\arraystretch{1.25}
\begin{array}{c|l}
 \text{twisted type}&\text{automorphism}\\ \hline
 {}^2A_\ell\ (\ell\geq5),\ {}^2D_\ell\ (\ell\geq5),\ {}^2E_6
       &i_g^\Lambda\circ\mu_*\\
 {}^2A_n\ (n=3,4)
       &i_g^\Lambda\circ\Theta_{\delta,\Lambda}^{(n)}\circ\mu_*\\
 {}^2D_4&i_g^\Lambda\circ\Gamma_{\mathbf e,\Lambda}\circ\mu_*\\
 {}^3D_4&i_g^\Lambda\circ\Phi_{\mu,\Lambda}.
\end{array}
$$
Here $g\in G_{\mathrm{ad},\sigma}(\Phi,R)$,
$\mu\in\operatorname{Aut}(R,\theta)$ outside triality and
$\mu\in\operatorname{Aut}_{\mathrm{tr}}(R,\theta)$ in triality;
$\delta\in\mathcal E_{\mathrm{glob}}(R,\theta)$ and the
idempotents $\mathbf e$ satisfy the conditions above.
With the root coordinates and graph representatives fixed, all
displayed data are unique. Conversely, every such composition is
an automorphism.
\end{maintheorem}

The adjoint elementary classification is the case $\Lambda=\Lambda_r$.
No nontrivial central factor occurs for elementary groups, which
are perfect. Separating $i_g^\Lambda$ into strictly inner and
diagonal factors is possible but noncanonical; keeping the
adjoint-induced action makes the uniqueness statement precise.

For full groups, a \emph{central automorphism} is a bijective map
$$
        c_\zeta(x)=\zeta(x)x,
        \qquad
        \zeta:G_\Lambda(R)\longrightarrow Z(G_\Lambda(R))
$$
with $\zeta$ a group homomorphism.
Such an automorphism fixes $E_\Lambda(R)$ pointwise.
The next theorem describes both the image and the kernel of
restriction to the elementary subgroup.

\begin{maintheorem}[Full groups]
\label{thm:main-full}
\label{thm:automorphisms-full-groups-nonexceptional}
\label{thm:automorphisms-full-groups-exceptional-unitary}
\label{thm:final-automorphism-theorem}
Under the same assumptions, $E_\Lambda(R)$ is characteristic in
$G_\Lambda(R)$.
An automorphism of $E_\Lambda(R)$ extends to an automorphism of
$G_\Lambda(R)$ if and only if its decomposition in
Theorem~\ref{thm:main-elementary} has no nonzero exceptional factor,
or the type is ${}^2A_3$ and $\Lambda=\Lambda_w$.

For each extendible elementary automorphism $\alpha$, let
$\widehat\alpha$ be the product, in the same order, of the
full-group extensions constructed in
Section~\ref{sec:arbitrary-isogeny-types-full-groups}.
Every automorphism of $G_\Lambda(R)$ has a unique expression
$$
        \varphi=\widehat\alpha\circ c_\zeta,
$$
where $\alpha=\varphi|_{E_\Lambda(R)}$ and $c_\zeta$ is a central
automorphism. Conversely, every such composition is an automorphism.
The central factor is trivial for the adjoint isogeny type.
\end{maintheorem}

The extensions in Theorem~\ref{thm:main-full} use the canonical
adjoint and graph actions, ordinary coefficientwise action, the
triality action $\widehat\Phi_{\mu,\Lambda}$, and, for the simply
connected type ${}^2A_3$, the specified extension
$\widehat\Theta_{\delta,\mathrm{sc}}^{(3)}$.
Uniqueness is relative to these fixed extensions: any two extensions
of the same elementary automorphism differ by a central automorphism.
Both main theorems are proved in
Section~\ref{sec:arbitrary-isogeny-types-full-groups}.

In particular, the hyperbolic special unitary group
$$
 \operatorname{SU}_4(R,\theta;Q_4)
   =\{x\in\operatorname{SL}_4(R)\mid
                   xQ_4\overline x^{\,t}=Q_4\},
        \qquad \overline x=\theta(x),
$$
with the standard form matrix $Q_4$, has the classification
$$
        \varphi=i_g^{\Lambda_w}
          \circ\widehat\Theta_{\delta,\mathrm{sc}}^{(3)}
          \circ\mu_*\circ c_\zeta.
$$
This holds over every commutative ring with $2$ invertible.
In contrast, no nonzero exceptional elementary parameter extends
to either of the other isogeny forms of ${}^2A_3$ or to a full
group of type ${}^2A_4$.

\subsection*{Relation to earlier work}

The study of automorphisms of linear groups began with the
field-level results of Schreier and van der Waerden and the
subsequent work on classical groups by Rickart and Dieudonn\'e
\cite{SchreierVanDerWaerden,RickartLinearTransformations,
DieudonneClassicalGroups}.
For Chevalley groups over fields, the standard descriptions of
automorphisms and abstract isomorphisms were developed by
Steinberg, Humphreys, and Borel--Tits
\cite{SteinbergAutomorphisms,SteinbergLectures,
HumphreysAutomorphisms,BorelTitsAbstractHomomorphisms}.

Over rings, early results include those of Hua and Reiner on the
integral unimodular group and of O'Meara on linear groups over
integral domains
\cite{HuaReinerUnimodular,OMearaIntegralDomain}.
Automorphisms over local and general commutative rings were studied
by Pomfret and McDonald, Waterhouse, Petechuk, and others
\cite{PomfretMcDonaldLocal,WaterhouseGLn,PetechukMatrixGroups}.
The work of Golubchik and Mikhalev and of Zelmanov established
broad isomorphism theorems over associative rings containing $1/2$,
under the appropriate rank hypotheses
\cite{GolubchikMikhalevGL,ZelmanovLinearGroups}.

For unitary groups, Golubchik and Mikhalev worked under
hyperbolic-rank and central-unit assumptions
\cite{GolubchikMikhalevUnitary}; Zelmanov's method used graded
Lie algebras and Jordan systems \cite{ZelmanovLinearGroups}.
Ismagilova obtained an isomorphism theorem under suitable
idempotent hypotheses \cite{IsmagilovaUnitary}.
Their hypotheses do not cover the arbitrary-ring hyperbolic-rank-two
situation considered here. The exceptional family in ${}^2A_3$
and ${}^2A_4$ explains why the lower-rank answer differs from
the stable one.

For Chevalley groups over commutative rings, the structural work of
Abe, Suzuki, Vaserstein, Vavilov, and others
\cite{AbeLocal,AbeCommutative,AbeSuzukiNormal,Vaserstein,PlotkinVavilov}
provided the tools for local automorphism theorems and their
extension to arbitrary commutative rings
\cite{BuninaDifferentTypes,BuninaGeneral}.
Twisted groups and their central coverings were constructed over
rings by Abe \cite{AbeCoverings}; their normal subgroups,
elementary subgroups, and centers were studied by Suzuki
\cite{SuzukiLocalNormal,SuzukiNormality,SuzukiCenters}.
More recent normal-structure, normalizer, and factorization results
are due to Garge and Makadiya
\cite{GargeMakadiyaNormal,GargeMakadiyaNormalizer,
GargeMakadiyaFactorization}.

The local inputs for this paper are provided by
\cite{BuninaMakadiya2ALarge,BuninaMakadiya2ASmall,
BuninaMakadiya2D,BuninaMakadiya2E6,BuninaMakadiyaTriality}.
Passing to arbitrary rings requires more than applying a local
automorphism theorem at each maximal ideal: a localized global
automorphism has a possibly nonlocal source and need not be
surjective onto its local target. We use the available relative
results and prove the additional normalization and recognition
statements needed for these maps. We then glue the resulting data
and handle central isogenies and passage to the full groups.

\subsection*{Idea of the proof}

We work first with the adjoint elementary group. Its
\emph{standard frame} is the finite indexed family
$$
        \mathcal F_\Phi(R)
        =\bigl(x_\lambda(1)\bigr)_{\lambda\in\Phi_\rho},
$$
where $1$ means $(1,1/2)$ for a nonadditive relative root group.
Its centralizer in the ambient adjoint group is trivial over every
base ring under consideration. The frame need not generate the
elementary group: coefficient automorphisms fix it. We therefore
separate normalization of this finite family from recovery of
the arbitrary root parameters.

\smallskip
\noindent\emph{Residual quotients and local models.}
The maximal proper normal subgroups are the kernels of reduction
along $\theta$-orbits of maximal ideals. For order-two twistings,
this follows from the normal-structure results of
\cite{GargeMakadiyaNormal}, with the corresponding split results
when $\theta=\operatorname{id}_R$.
An ideal-level classification fails in general for triality, so we
prove the required maximal-normal-subgroup statement directly in
Theorem~\ref{thm:triality-maximal-normal-subgroups}.
An automorphism consequently permutes these residual quotients.
Localizing along a whole orbit retains the coefficient operator;
residual isomorphism theorems and relative local normalization
then supply the local data.

\smallskip
\noindent\emph{Finite normalization and gluing.}
The data to be glued are
$$
\begin{array}{c|c}
 \text{case}&\text{normalization data}\\ \hline
 \text{ordinary types and }{}^3D_4&g\\
 {}^2A_3,\ {}^2A_4&(g,\delta)\\
 {}^2D_4&(g,\mathbf e).
\end{array}
$$
In each case, finitely many equations define an affine scheme of
finite presentation, with at most one point over every test
algebra, not merely over local rings. This uniqueness gives both
faithfully flat descent and agreement on open intersections;
Proposition~\ref{prop:unique-affine-gluing} yields the global data.
For the exceptional parameter, the finite system imposes
$\delta^2=3\delta=0$; its full annihilator conditions are checked
after gluing. Twisting compatibility is also detected locally.

\smallskip
\noindent\emph{Coefficient recovery.}
Removing the normalization factors fixes the ordinary frame.
Recognition of root subgroups and rank-two commutator identities
recover addition, multiplication, and the coefficient involution.
In triality they recover the ring and $\operatorname{Tr}_\theta$,
but not necessarily $\theta$ up to inversion. This step is
independent of how the frame was normalized. For exceptional and
triality actions, local detection proves that the prescribed root
images respect all group relations.

\smallskip
\noindent\emph{Central isogenies and full groups.}
The elementary group in every invariant isogeny form is perfect,
with adjoint quotient obtained by dividing out its center.
Perfectness makes elementary lifts unique; graph factors lift
precisely when they preserve the lattice.
For full groups, restriction to the characteristic elementary
subgroup has central automorphisms as its kernel.
We construct extensions from local elementary and torus actions.
In the excluded unitary forms, conjugation by a torus element
changes the sign of the exceptional parameter, contradicting the
compatibility that any full-group extension must satisfy.

Section~1 fixes the groups, root coordinates, and local models.
Sections~2--3 establish the general normalization and recovery
argument under explicit local hypotheses, verified in
Sections~4--6. Sections~\ref{sec:exceptional-unitary-types},
\ref{sec:graph-automorphisms-2D4}, and~\ref{sec:triality}
treat the exceptional unitary, graph, and triality cases,
respectively. Section~\ref{sec:arbitrary-isogeny-types-full-groups}
passes to all invariant isogeny types and full groups and completes
the proofs of the main theorems.

\section{Preliminaries and local models}
\label{sec:preliminaries-local-models}

We fix the coefficient pairs, root coordinates, and adjoint groups
used in the proof. Localization along an orbit of maximal ideals
provides the local models; an affine gluing principle then assembles
uniquely determined local normalization data. Other isogeny forms
are treated in
Section~\ref{sec:arbitrary-isogeny-types-full-groups}.

\subsection{Rings and coefficient maps}
\label{subsec:rings-with-operators}

Fix a root system $\Phi$ and a diagram automorphism $\rho$ from
the list in the introduction, and put $m:=\operatorname{ord}(\rho)$.
All rings are commutative with identity, and all ring homomorphisms
are unital. We work with pairs $(R,\theta)$ satisfying
$$
        \theta\in\operatorname{Aut}(R),\qquad
        \theta^m=\operatorname{id}_R,\qquad m\in\{2,3\}.
$$
Throughout, $2\in R^\times$ and, if $m=3$, also $3\in R^\times$.
The order of $\theta$ may be smaller than $m$; this convention
is essential when passing to residue fields. Write
$$
        \overline r:=\theta(r),\qquad
        R_0:=R^\theta.
$$
For $m=2$, also put
$R^-_\theta:=\{r\in R\mid\overline r=-r\}$, so that
$R=R_0\oplus R^-_\theta$ as $R_0$-modules.
For $m=3$, $\overline{\overline r}$ denotes $\theta^2(r)$.

A ring homomorphism $f:(R,\theta_R)\to(S,\theta_S)$ is
\emph{equivariant} if $f\theta_R=\theta_Sf$.
For a $\theta$-invariant ideal $I$ or multiplicative set $\Sigma$,
the quotient $R/I$ or localization $\Sigma^{-1}R$ carries the
induced operator
$$
        \theta(r+I)=\theta(r)+I,\qquad
        \theta(r/s)=\theta(r)/\theta(s).
$$

The pinned graph action and the coefficient action commute on the
split adjoint group. We write
$$
        \sigma_R:=\rho\circ\theta
$$
for their composite, and use $\sigma$ when the ring is understood.
Every ring homomorphism induces the usual map on split group
points. For an equivariant map, it also induces maps on the twisted
and elementary groups defined below. All these maps are denoted
by $f_*$.

\subsection{Relative roots and root coordinates}
\label{subsec:relative-root-subgroups}

We group the absolute roots by their reduced relative directions
and choose coordinates compatible with equivariant coefficient maps.
For $\alpha\in\Phi$, put
$$
        \alpha_\rho:=\sum_{j=0}^{m-1}\rho^j(\alpha).
$$
Two absolute roots determine the same reduced relative root if
their orbit sums lie on the same positive ray. Denote the resulting
class by $[\alpha]$ and put
$$
        \Phi_\rho:=\{[\alpha]\mid\alpha\in\Phi\},\qquad
        -[\alpha]=[-\alpha].
$$
Write $\overline\alpha:=\rho(\alpha)$.
A class is either an orbit of $d$ pairwise orthogonal roots,
with $d\in\{1,m\}$, or, only when $m=2$, a set
$\{\alpha,\overline\alpha,\alpha+\overline\alpha\}$.
We call these classes of type $A_1^d$ and $A_2$, respectively.
The possibilities are
$$
\begin{array}{c|c|c|c}
\text{twisted type}&\Phi_\rho&
\text{long roots}&\text{short roots}\\ \hline
{}^2A_{2n-1}\ (n\geq2)&C_n&A_1&A_1^2\\
{}^2A_{2n}\ (n\geq2)&B_n&A_1^2&A_2\\
{}^2D_\ell\ (\ell\geq4)&B_{\ell-1}&A_1&A_1^2\\
{}^2E_6&F_4&A_1&A_1^2\\
{}^3D_4&G_2&A_1&A_1^3
\end{array}
$$
For ${}^2A_{2n}$ the non-reduced system is $BC_n$.
We index by the reduced system $B_n$, keeping the three absolute
roots on each short ray in one relative root subgroup.

Fix a $\rho$-adapted Chevalley basis, with signs as in
\cite[Lemma~2.4]{BuninaMakadiya2ALarge}, and write $x_\alpha(t)$
for the absolute root maps.
For each $\lambda\in\Phi_\rho$, choose a representative
$\alpha_\lambda$ with $\alpha_{-\lambda}=-\alpha_\lambda$.
For an $A_2$-class, choose one of its two non-fixed roots;
the order $(\alpha_\lambda,\rho(\alpha_\lambda))$ is then fixed.
These choices affect the coordinates, not the subgroups.
When $\lambda$ is understood, write $\alpha:=\alpha_\lambda$.

\smallskip
\noindent\emph{Additive root groups.}
For $\lambda\sim A_1^d$, put $R_\lambda=R_0$ if $d=1$ and
$R_\lambda=R$ if $d=m$, and define
$$
        x_\lambda(t):=
        \prod_{j=0}^{d-1}x_{\rho^j(\alpha)}(\theta^j(t)),
        \qquad t\in R_\lambda.
$$
The factors commute, so
$$
        x_\lambda(t)x_\lambda(u)=x_\lambda(t+u).
$$
In triality we subsequently choose the Weyl-coherent short-root
coordinates of
Subsection~\ref{subsec:triality-coefficient-automorphisms}.
They differ only by cyclic changes $t\mapsto\theta^j(t)$ and
leave the subgroups and their parameter-one elements unchanged.

\smallskip
\noindent\emph{Root groups of type $A_2$.}
Here $m=2$. Put $\gamma=\alpha+\overline\alpha$ and fix
$N=N_{\alpha,\overline\alpha}\in\{\pm1\}$ by
$$
 x_\alpha(a)x_{\overline\alpha}(b)
 =x_{\overline\alpha}(b)x_\alpha(a)x_\gamma(Nab).
$$
The graph action interchanges the first two root maps and sends
$x_\gamma(c)$ to $x_\gamma(-c)$.
Set
$$
        \mathcal A(R):=
        \{(t,u)\in R^2\mid t\overline t=u+\overline u\},
        \qquad R_\lambda:=\mathcal A(R),
$$
and define
$$
\begin{aligned}
 x_\lambda(t,u)
 &:=x_{\overline\alpha}(\overline t)x_\alpha(t)x_\gamma(Nu)\\
 &=x_\alpha(t)x_{\overline\alpha}(\overline t)
                         x_\gamma(-N\overline u).
\end{aligned}
$$
The parameter condition gives $u-t\overline t=-\overline u$,
so the two expressions agree and are fixed by $\sigma_R$.
Multiplication gives the parameter-group law
$$
        (t,u)\oplus(s,v)=(t+s,u+v+t\overline s),\qquad
        (t,u)^{-1}=(-t,\overline u).
$$
The identity is $(0,0)$, and the parameters $(0,z)$ with
$z\in R^-_\theta$ form a central subgroup.
Since $2$ is invertible, the defining condition is equivalent to
$$
        u=t\overline t/2+z,\qquad z\in R^-_\theta.
$$
Thus the first coordinate is arbitrary and the second has one
free anti-invariant part. Correspondingly,
$$
        x_\lambda(t,u)
        =x_\lambda(t,t\overline t/2)x_\lambda(0,z),
        \qquad z=u-t\overline t/2.
$$
We also write $x_\lambda(a)$ for $a=(t,u)\in\mathcal A(R)$.

In the positive $3\times3$ chart with $N=1$, this is the matrix
$$
        x_\lambda(t,u)=
        \begin{pmatrix}
        1&t&u\\0&1&\overline t\\0&0&1
        \end{pmatrix}.
$$
Here $u$ is the central matrix coordinate, as in
\cite[Subsection~2.2]{BuninaMakadiya2ASmall}.
For opposite roots, using
$N_{-\alpha,-\overline\alpha}=-N_{\alpha,\overline\alpha}$
in the unitary realizations with transposed absolute root maps gives
$$
        x_{-\lambda}(t,u)
        =x_\lambda(t,\overline u)^{\mathsf t}.
$$
Thus $(t,u)\mapsto(t,\overline u)$ converts the transposed
negative charts of \cite[Subsection~2.2]{BuninaMakadiya2ASmall}
to ours. It also converts the law with cross term $\overline t\,s$
in \cite[Subsection~2.4]{BuninaMakadiya2ALarge} to the present law.
It fixes $(1,1/2)$ and $(t,t\overline t/2)$ and sends
$(0,z)$ to $(0,-z)$. Imported Weyl and commutator formulas are
read with this conversion.

In all cases, ordered absolute unipotent coordinates show that
$x_\lambda:R_\lambda\to G_{\mathrm{ad}}(\Phi,R)$ is injective.
We denote its image by $X_\lambda(R)$. Thus
$$
R_\lambda=
\begin{cases}
R_0,&\lambda\sim A_1,\\
R,&\lambda\sim A_1^2\text{ or }A_1^3,\\
\mathcal A(R),&\lambda\sim A_2.
\end{cases}
$$
An equivariant map $f:(R,\theta_R)\to(S,\theta_S)$ induces
$f_\lambda(t)=f(t)$ in the additive cases and
$f_\lambda(t,u)=(f(t),f(u))$ in type $A_2$.
These are homomorphisms of parameter groups and satisfy
$$
        f_*(x_\lambda(a))=x_\lambda(f_\lambda(a)).
$$

\subsection{The adjoint twisted groups}
\label{subsec:adjoint-twisted-elementary-group}

Let $G_{\mathrm{ad}}(\Phi,R)$ denote the $R$-points of the split
adjoint group scheme. Define the full twisted group and its
elementary subgroup by
$$
\begin{aligned}
 G_\sigma(\Phi,R)
   &:=G_{\mathrm{ad},\sigma}(\Phi,R)
     =G_{\mathrm{ad}}(\Phi,R)^{\sigma_R},\\
 E_\sigma(\Phi,R)
   &:=E'_{\mathrm{ad},\sigma}(\Phi,R)
     =\langle X_\lambda(R)\mid\lambda\in\Phi_\rho\rangle.
\end{aligned}
$$
The root formulas give $E_\sigma(\Phi,R)\leq G_\sigma(\Phi,R)$.
When the type and operator are understood, write
$E(R)=E_\sigma(\Phi,R)$ and $G(R)=G_\sigma(\Phi,R)$;
$G_{\mathrm{ad}}(\Phi,R)$ always denotes the ambient split group.

Under the standing assumptions, $E(R)$ is normal in $G(R)$.
Hence $g\in G(R)$ induces the conjugation automorphism
$i_g(x)=gxg^{-1}$ of $E(R)$, called \emph{strictly inner}.
Every
$$
        \mu\in\operatorname{Aut}(R,\theta)
        :=\{\mu\in\operatorname{Aut}(R)\mid\mu\theta=\theta\mu\}
$$
acts on both groups by coefficient change, with
$\mu_*(x_\lambda(a))=x_\lambda(\mu_\lambda(a))$.
We call this a \emph{ring automorphism} of the group.
An automorphism of the adjoint elementary group is \emph{standard}
if it is $i_g\circ\mu_*$. The factors satisfy
$$
        \mu_*\circ i_g=i_{\mu_*(g)}\circ\mu_*.
$$
The additional graph factors in ${}^2D_4$ and the broader
coefficient actions in triality are treated in
Sections~\ref{sec:graph-automorphisms-2D4} and~\ref{sec:triality}.
They are not included in this use of the word standard.

\subsection{The standard frame and its rigidity}
\label{subsec:standard-frame}

For an additive relative root group, $x_\lambda(1)$ has its
literal meaning. For type $A_2$, put
$x_\lambda(1):=x_\lambda(1,1/2)$.
The \emph{standard frame} is the finite indexed family
$$
        \mathcal F_\Phi(R):=
        (x_\lambda(1))_{\lambda\in\Phi_\rho}.
$$
Its elements are defined over $\mathbb Z[1/2]$, independently of
the coefficient operator. Thus $\mathcal F_\Phi(T)$ is defined
in the ambient adjoint group over every ring $T$ with $2$ invertible,
without an operator on $T$, and every unital coefficient map
carries the frame to the frame.

The \emph{frame rigidity} property, denoted by $\mathrm{(LR3)}$, is
$$
        C_{G_{\mathrm{ad}}(\Phi,T)}(\mathcal F_\Phi(T))=1
$$
for every commutative ring $T$ satisfying the standing
invertibility assumptions. It is proved in
Section~\ref{sec:rigidity-standard-frame}.
In particular, an element $g\in G_{\mathrm{ad}}(\Phi,T)$ satisfying
$$
        g x_\lambda(1)g^{-1}=v_\lambda
        \qquad(\lambda\in\Phi_\rho)
$$
for a prescribed indexed family $(v_\lambda)$ is unique whenever
it exists. We call it the \emph{transporter} of the two frames.

The frame determines a conjugating element, not an automorphism:
it need not generate $E(R)$, and coefficient automorphisms fix it.
The coefficient map is recovered from the action on the whole root
subgroups. Local calculations also use Weyl and toral words, but
only their resulting frame identities enter the gluing.

\subsection{Orbit localizations and affine gluing}
\label{subsec:local-models}
\label{subsec:local-global-principles}
\label{subsec:local-global-language}

The operator carries $R_{\mathfrak m}$ to
$R_{\theta(\mathfrak m)}$, so a single maximal localization is
a coefficient pair only when $\theta(\mathfrak m)=\mathfrak m$.
We therefore localize along whole orbits. Put
$$
        \Omega_\theta(R):=\operatorname{Max}R/\langle\theta\rangle
          =\{\mathcal O_i\mid i\in\mathcal I\}.
$$
Each orbit has one or $m$ elements. For $i\in\mathcal I$, define
$$
\begin{aligned}
 I_i&:=\bigcap_{\mathfrak m\in\mathcal O_i}\mathfrak m,
 &K_i&:=R/I_i,\\
 \Sigma_i&:=R\setminus\bigcup_{\mathfrak m\in\mathcal O_i}\mathfrak m,
 &B_i&:=\Sigma_i^{-1}R,\\
 &&A_i&:=\prod_{\mathfrak m\in\mathcal O_i}R_{\mathfrak m}.
\end{aligned}
$$
These rings serve different purposes: $K_i$ is the residue algebra,
$B_i$ the semilocal orbit neighborhood, and $A_i$ its faithfully
flat cover by a product of local rings. The local calculations take
place over $A_i$; descent first returns their data to $B_i$.
All three rings carry induced operators, with $\theta$ permuting
the factors of $A_i$. Write
$\iota_i:R\to A_i$ and $q_i:R\to K_i$ for the natural equivariant
maps.

\begin{definition}
A pair $(A,\theta)$ satisfying the standing assumptions is a
\emph{local model} if either $A$ is local, or
$$
        A=\prod_{j=0}^{m-1}A^{(j)}
$$
with each $A^{(j)}$ local and $\theta$ cyclically permuting the
factors. These models are called \emph{fixed} and \emph{split},
respectively. In the fixed case it is the maximal ideal that is
fixed; $\theta$ need not be the identity.
For either model put
$$
        J_A:=\operatorname{Jac}(A),\qquad
        A_{\mathrm{res}}:=A/J_A,
$$
and denote reduction by $q_A:A\to A_{\mathrm{res}}$.
\end{definition}

Thus a fixed residue algebra is a field with an operator of order
dividing $m$; a split residue algebra is a cyclic product of $m$ fields.

\begin{lemma}
\label{lem:orbit-neighborhoods-general}
The ring $B_i$ is semilocal, with
$$
 \operatorname{Max}B_i
    =\{\mathfrak mB_i\mid\mathfrak m\in\mathcal O_i\},
 \qquad \operatorname{Jac}(B_i)=I_iB_i.
$$
The ring $A_i$ is a local model, and $B_i\to A_i$ is equivariant
and faithfully flat. There are canonical equivariant identifications
$$
 B_i/\operatorname{Jac}(B_i)\simeq K_i
 \simeq A_i/\operatorname{Jac}(A_i)
 \simeq\prod_{\mathfrak m\in\mathcal O_i}R/\mathfrak m,
$$
and
$$
        \operatorname{Jac}(A_i)
        =\prod_{\mathfrak m\in\mathcal O_i}
                    \mathfrak mR_{\mathfrak m}.
$$
\end{lemma}

\begin{proof}
By finite prime avoidance, every prime disjoint from $\Sigma_i$
is contained in some $\mathfrak m\in\mathcal O_i$.
Thus the maximal ideals of $B_i$ are precisely the
$\mathfrak mB_i$, and their intersection is $I_iB_i$.
The Chinese remainder theorem identifies
$K_i$ with $\prod_{\mathfrak m\in\mathcal O_i}R/\mathfrak m$.
Every element of $\Sigma_i$ becomes a unit in this product, so
localization preserves the quotient.

Each factor of $A_i$ is
$R_{\mathfrak m}=(B_i)_{\mathfrak mB_i}$.
Their finite product is flat, and their spectra jointly cover
$\operatorname{Spec}B_i$, since every prime lies in a maximal ideal.
Hence $B_i\to A_i$ is faithfully flat.
The radical and residue algebra of $A_i$ are computed factorwise.
All maps and identifications commute with the induced operators.
\end{proof}

Thus $A_i=B_i=R_{\mathfrak m}$ for a fixed orbit
$\mathcal O_i=\{\mathfrak m\}$, while $A_i$ is split for an orbit
of size $m$. We use the residue identifications throughout, writing
$$
        \kappa_i:K_i\xrightarrow{\ \sim\ }A_{i,\mathrm{res}},
        \qquad q_{A_i}\iota_i=\kappa_iq_i.
$$
Every maximal localization of $R$ occurs as a factor of some $A_i$.

\smallskip
\noindent\emph{Groups over a cyclic product.}
After identifying its factors by powers of $\theta$, a split model
has the form
$$
        A=B^m,\qquad
        \theta_A(b_0,\ldots,b_{m-1})
        =(b_{m-1},b_0,\ldots,b_{m-2}).
$$
The following identification does not require $B$ to be local.

\begin{lemma}
\label{lem:split-product-identification}
For any commutative ring $B$ satisfying the standing assumptions,
first projection and the map
$$
        \Delta_\rho(g):=(g,\rho(g),\ldots,\rho^{m-1}(g))
$$
are inverse isomorphisms
$$
\begin{aligned}
 G_{\rho\theta_A}(\Phi,B^m)&\simeq G_{\mathrm{ad}}(\Phi,B),\\
 E_{\rho\theta_A}(\Phi,B^m)&\simeq E_{\mathrm{ad}}(\Phi,B).
\end{aligned}
$$
Here $E_{\mathrm{ad}}(\Phi,B)$ is the untwisted elementary group.
Projection maps each relative root subgroup onto the product of
all absolute root subgroups in its class, in the fixed root order.
\end{lemma}

\begin{proof}
The fixed-point equation for $(g_0,\ldots,g_{m-1})$ is
$g_j=\rho(g_{j-1})$, with indices modulo $m$.
It gives the full-group identification.
For a class of type $A_1$, the relative parameter is diagonal
and projects to an arbitrary parameter in $B$.
For type $A_1^m$, its $m$ components give independent parameters
in the $m$ commuting absolute root groups.
For type $A_2$, where $m=2$, take
$$
        t=(a,b),\qquad u=(c,ab-c),\qquad a,b,c\in B.
$$
Then $t\overline t=u+\overline u$, and the projected element is
$$
        x_{\rho(\alpha)}(b)x_\alpha(a)
        x_{\alpha+\rho(\alpha)}(N_{\alpha,\rho(\alpha)}c).
$$
Its three absolute coordinates are independent. Thus the projected
relative root subgroups contain every absolute root subgroup,
and their images generate $E_{\mathrm{ad}}(\Phi,B)$.
Injectivity follows from the full-group identification.
\end{proof}

This identifies the groups and root subgroups, but does not
normalize a homomorphism into them. The fixed-to-split normalization
argument is given in Subsection~\ref{subsec:common-local-reductions}.

\smallskip
\noindent\emph{Affine descent and gluing.}
Write $R_s=R[1/s]$ and $D(s)=\operatorname{Spec}R_s$.
We may always use fixed denominators in an invariant localization:
for $s\in\Sigma$, the product
$$
        \prod_{j=0}^{m-1}\theta^j(s)\in\Sigma\cap R_0
$$
is fixed and divisible by $s$. Replacing $s$ by this product
restricts $D(s)$ to a smaller invariant principal open.

For an affine $R$-scheme $X$, write $X(T)$ for its points over
an $R$-algebra $T$. Our normalization data will be such points.
We use the following standard facts; see \cite{StacksProject}.

\begin{lemma}[Affine descent and local detection]
\label{lem:affine-descent-and-local-detection}
\begin{enumerate}
\item
If $z$ is an element of an $R$-module and $z/1=0$ in every
maximal localization, then $z=0$. Consequently, equality of affine
$R$-points and membership of an $R$-point in a closed subscheme
can be checked over all $R_{\mathfrak m}$, or over all $A_i$.
\item
For a faithfully flat map $B\to A$ and an affine $B$-scheme $X$,
$$
        X(B)\longrightarrow X(A)
             \rightrightarrows X(A\otimes_B A)
$$
is an equalizer.
\item
If $X$ is affine of finite presentation over $R$ and $\Sigma$
is multiplicative, every point of $X(\Sigma^{-1}R)$ comes from
$X(R_s)$ for some $s\in\Sigma$.
\item
Points of an affine $R$-scheme over a finite principal-open
covering, equal on its pairwise intersections, glue uniquely
over $R$.
\end{enumerate}
\end{lemma}

\begin{proof}
For (1), a nonzero $z$ has annihilator contained in a maximal
ideal $\mathfrak m$. Vanishing in that localization would give
$sz=0$ for some $s\notin\mathfrak m$, a contradiction.
Apply this to coordinate equalities and defining equations for
the affine assertions.
Part (2) is affine faithfully flat descent; equivalently, apply
$\operatorname{Hom}_{B\text{-alg}}(C,-)$ to the equalizer for
$B\to A$, where $X=\operatorname{Spec}C$.
In (3), finitely many denominators suffice for the coordinates
and the finitely many relations. Part (4) is the affine Zariski
sheaf property.
\end{proof}

The uniqueness required below is over every base algebra, not
only over local models. It makes both descent and agreement on
intersections automatic.

\begin{proposition}[Unique affine gluing]
\label{prop:unique-affine-gluing}
Let $X$ be an affine $R$-scheme of finite presentation.
Suppose that $X(T)$ has at most one point for every commutative
$R$-algebra $T$, and that $X(A_i)$ is nonempty for every
$i\in\mathcal I$. Then $X(R)$ consists of one point, whose
image over each $A_i$ is its prescribed point.
\end{proposition}

\begin{proof}
Choose $a_i\in X(A_i)$. Its two pullbacks to
$A_i\otimes_{B_i}A_i$ agree by uniqueness, so faithful flatness
descends it to $b_i\in X(B_i)$.
Finite presentation spreads $b_i$ to some $x_i\in X(R_{s_i})$,
$s_i\in\Sigma_i$. Replacing the denominator as above and restricting
$x_i$, we may take $s_i\in R_0$.

Since $s_i$ avoids every maximal ideal in $\mathcal O_i$,
the opens $D(s_i)$ cover all maximal ideals, hence all of
$\operatorname{Spec}R$. Choose a finite subcover.
On each intersection $D(s_is_j)$, the points agree by the same
uniqueness hypothesis. They therefore glue to a point of $X(R)$.
Its uniqueness and its prescribed images over all $A_i$ follow
again from uniqueness over the corresponding algebras.
\end{proof}

The construction follows the sequence
$$
        X(A_i)\leadsto X(B_i)
        \leadsto X(R_{s_i})\leadsto X(R).
$$
Finite presentation allows passage to principal neighborhoods;
uniqueness over all algebras gives both descent and agreement
on intersections. No operator is needed on these test algebras.
Twisting identities such as $\sigma(g)=g$ are checked on the
resulting global data by local detection.

The hypotheses on residual isomorphisms and relative local
normalization are stated where they are used, at the beginning
of Section~\ref{sec:global-frame-normalization}.
\medskip


\section{Global normalization of the standard frame}
\label{sec:global-frame-normalization}

We first identify the maximal congruence quotients and the residual
isomorphisms induced by an automorphism. This construction applies
to every type under consideration. Under the ordinary local
hypotheses below, we then obtain a unique global frame transporter
by the affine gluing principle of Section~\ref{sec:preliminaries-local-models}.

Retain the orbit notation $I_i,K_i,A_i$ of that section and write
$E(U),G(U)$ for the adjoint elementary and full twisted groups
over a coefficient pair $U$. Fix $\varphi\in\operatorname{Aut}E(R)$.

\subsection{Local hypotheses}
\label{subsec:local-rigidity-package}

Frame rigidity $\mathrm{(LR3)}$ is as in
Subsection~\ref{subsec:standard-frame}; its proof is in
Section~\ref{sec:rigidity-standard-frame}.
We record the remaining ordinary local inputs and their ranges
of validity.

\smallskip
\noindent\emph{\textup{(LR1)} Residual standardness.}
For an order-two twisting with $\Phi\ne D_4$, every isomorphism
$$
        \alpha:E(K)\xrightarrow{\ \sim\ }E(L)
$$
between groups over residue algebras of local models has the form
$\alpha=i_h\circ\nu_*$, where $h\in G(L)$ and
$\nu:K\xrightarrow{\sim}L$ is an equivariant ring isomorphism.
In particular, $K,L$ are both fixed or both split. This is proved
in Section~\ref{sec:residual-standardness}.

\smallskip
\noindent\emph{\textup{(LR2)} Relative frame normalization.}
Let $(S,\theta_S)$ be arbitrary, let $(A,\theta_A)$ be a local
model, and let $\Psi:E(S)\to E(A)$ be a homomorphism with
$$
        \Psi_{\mathrm{res}}:=(q_A)_*\circ\Psi
             =i_h\circ(\nu_{\mathrm{res}})_*,
        \qquad h\in G(A_{\mathrm{res}}),
$$
where $\nu_{\mathrm{res}}:S\to A_{\mathrm{res}}$ is an equivariant
unital ring homomorphism. Then there is $g_A\in G(A)$ such that
$$
        (q_A)_*(g_A)=h,\qquad
        g_Ax_\lambda(1)g_A^{-1}=\Psi(x_\lambda(1))
        \quad(\lambda\in\Phi_\rho).
$$
Frame rigidity makes $g_A$ unique.

\smallskip
\noindent\emph{\textup{(LR4)} Root-subgroup recognition.}
For the same source and target, suppose that $\Psi$ carries the
source frame to the target frame and
$\Psi_{\mathrm{res}}=(\nu_{\mathrm{res}})_*$ for an equivariant
unital ring homomorphism $\nu_{\mathrm{res}}$. Then
$$
        \Psi(X_\lambda(S))\subseteq X_\lambda(A)
        \qquad(\lambda\in\Phi_\rho).
$$

The source in $\mathrm{(LR2)}$ and $\mathrm{(LR4)}$ need not be
local, and $\Psi$ need be neither injective nor surjective.
These properties are verified in
Section~\ref{sec:relative-local-rigidity} for
${}^2A_\ell$ ($\ell\geq5$), ${}^2D_\ell$ ($\ell\geq4$),
and ${}^2E_6$; the ordinary ${}^2D_4$ case is recorded separately
in Lemma~\ref{lem:relative-local-rigidity-2D4}.
The residual graph variant $\mathrm{(LR1}_{D_4}\mathrm{)}$ is
Theorem~\ref{thm:residual-graph-standardness-D4}.
The exceptional unitary, graph, and triality normalizations are
handled in Sections~\ref{sec:exceptional-unitary-types},
\ref{sec:graph-automorphisms-2D4}, and~\ref{sec:triality}, respectively.
Once the residual identity is known, only $\mathrm{(LR2)}$ and
$\mathrm{(LR3)}$ enter the normalization below;
$\mathrm{(LR4)}$ is used in
Section~\ref{sec:recovery-of-coefficient-automorphism}.

\subsection{Maximal congruence quotients}
\label{subsec:congruence-spectrum}

We identify the residual quotients from the abstract group before
applying any standardness theorem to their isomorphisms.
For a $\theta$-invariant ideal $I\lhd R$, let $q_I:R\to R/I$
be reduction and put
$$
        N_I:=\ker\bigl(q_{I,*}:E(R)\longrightarrow E(R/I)\bigr).
$$

\begin{lemma}
\label{lem:surjectivity-reduction}
For every $\theta$-invariant ideal $I$, the map
$q_{I,*}:E(R)\to E(R/I)$ is surjective.
\end{lemma}

\begin{proof}
It suffices to lift relative root parameters.
For a root of type $A_1$, average any lift $a\in R$ by
$m^{-1}\sum_{r=0}^{m-1}\theta^r(a)$ to obtain a fixed lift.
Parameters for $A_1^2$- and $A_1^3$-groups lift arbitrarily.
For type $A_2$, write a residue parameter as
$(t,t\overline t/2+z)$, with $z$ anti-invariant, using
Subsection~\ref{subsec:relative-root-subgroups}.
Lift $t$ arbitrarily and lift $z$ by replacing any lift $b$ with
$(b-\overline b)/2$. The same formula gives an admissible lift
of the pair.
\end{proof}

Thus $E(R)/N_I\simeq E(R/I)$.
Let $R_\lambda(I)$ be the kernel of
$R_\lambda\to(R/I)_\lambda$. Explicitly,
$$
R_\lambda(I)=
\begin{cases}
 I\cap R_0,&\lambda\sim A_1,\\
 I,&\lambda\sim A_1^2\text{ or }A_1^3,\\
 \mathcal A(R)\cap(I\times I),&\lambda\sim A_2.
\end{cases}
$$
Denote by $E(R,I)$ the normal closure in $E(R)$ of all
$x_\lambda(a)$ with $a\in R_\lambda(I)$.
For order-two twistings we use the following form of
\cite[Theorems~1.2 and~1.4]{GargeMakadiyaNormal}.

\begin{theorem}[Normal structure for order-two twistings]
\label{thm:normal-structure}
If $m=2$, every normal subgroup $H\lhd E(R)$ has a unique
$\theta$-invariant ideal $I\lhd R$ such that
$$
        E(R,I)\leq H\leq N_I.
$$
\end{theorem}

\begin{proof}
For $\theta\ne\operatorname{id}_R$, apply
\cite[Theorem~1.4]{GargeMakadiyaNormal}.
For $\theta=\operatorname{id}_R$, apply
\cite[Theorem~1.2]{GargeMakadiyaNormal} to the corresponding
folded split group. The root-coordinate identifications respect
reduction and relative elementary subgroups. In either case,
frame rigidity over $R/I$ makes the central congruence bound
exactly $N_I$.
\end{proof}

The ideal $I$ is the \emph{level} of $H$.
The analogous ideal-level statement fails in general for triality.
We use instead the independent maximal-normal-subgroup theorem
of Subsection~\ref{subsec:triality-maximal-normal-subgroups},
which also gives a counterexample to the ideal-level statement.

\begin{lemma}
\label{lem:maximal-invariant-ideals}
The ideals $I_i$, $i\in\mathcal I$, are precisely the maximal
proper $\theta$-invariant ideals of $R$, without repetitions.
\end{lemma}

\begin{proof}
Every proper invariant ideal $L$ lies in some maximal ideal
$\mathfrak m$, hence in the intersection $I_i$ of its orbit.
If $L$ is maximal invariant, then $L=I_i$.
Conversely,
$$
        R/I_i\simeq
        \prod_{\mathfrak m\in\mathcal O_i}R/\mathfrak m
$$
has no nonzero proper invariant ideal, since its field factors
are permuted transitively. Thus $I_i$ is maximal invariant.
The maximal ideals containing $I_i$ are exactly the members of
$\mathcal O_i$, so $I_i$ determines its orbit.
\end{proof}

Put $N_i:=N_{I_i}=\ker(q_i)_*$.

\begin{proposition}[Maximal normal subgroups]
\label{prop:maximal-normal-subgroups}
Every proper normal subgroup of $E(R)$ is contained in some $N_i$.
The map $i\mapsto N_i$ is a bijection from $\mathcal I$ to the
maximal proper normal subgroups of $E(R)$.
\end{proposition}

\begin{proof}
Every type under consideration has a relative root
$\lambda_0\sim A_1^m$ with parameter group $R$.
Injectivity of its root coordinates gives
$$
        N_i\cap X_{\lambda_0}(R)=x_{\lambda_0}(I_i),
$$
so the $N_i$ are distinct.
For $m=3$, containment and maximality are
Theorem~\ref{thm:triality-maximal-normal-subgroups}, whose proof
uses neither the residual-action construction below nor
an automorphism classification.

Let $m=2$. The only invariant ideals of $K_i$ are $0$ and $K_i$,
so Theorem~\ref{thm:normal-structure} makes $E(K_i)$ simple.
It is nontrivial because $x_{\lambda_0}(1)\ne1$.
Surjectivity of reduction gives $E(R)/N_i\simeq E(K_i)$,
hence $N_i$ is maximal normal.
If $H\lhd E(R)$ is proper, its level $I$ is proper and lies in
some $I_i$ by the preceding proof. Thus
$$
        H\leq N_I\leq N_i<E(R).
$$
In particular, a maximal proper normal subgroup equals that $N_i$.
\end{proof}

\smallskip
\noindent\emph{The induced residual isomorphisms.}
The automorphism $\varphi$ therefore defines a permutation
$\varphi^\sharp:\mathcal I\to\mathcal I$ by
$$
        \varphi(N_i)=N_{\varphi^\sharp(i)}.
$$
For a target index $i$, put $j=(\varphi^\sharp)^{-1}(i)$.
Since $\varphi(N_j)=N_i$, passage to the quotients gives
$$
\begin{aligned}
        \varphi_i^{\mathrm{res}}:E(K_j)&\xrightarrow{\ \sim\ }E(K_i),\\
        \varphi_i^{\mathrm{res}}\bigl((q_j)_*(x)\bigr)
             &=(q_i)_*(\varphi(x))\qquad(x\in E(R)).
\end{aligned}
$$
The source index is determined by $\varphi$. Neither this
permutation nor the residual isomorphisms require
$\mathrm{(LR1)}$; both are available in every type.

\subsection{Local and global frame normalization}
\label{subsec:transporter-of-standard-frame}
\label{subsec:orbit-neighborhoods}
\label{subsec:global-gluing-transporter}

We now assume ordinary residual standardness and apply
Proposition~\ref{prop:unique-affine-gluing} to the frame transporter.
Write
$$
        \varphi_i^{\mathrm{res}}=i_{h_i}\circ(\nu_i)_*,
        \qquad h_i\in G(K_i),\qquad
        \nu_i:K_j\xrightarrow{\ \sim\ }K_i,
$$
where $\nu_i$ is equivariant and $j=(\varphi^\sharp)^{-1}(i)$.
Since $\nu_i$ preserves the number of field factors,
$\varphi^\sharp$ preserves fixed and split orbits separately.

Identify $A_{i,\mathrm{res}}=K_i$ by
Lemma~\ref{lem:orbit-neighborhoods-general}, and put
$$
        \Psi_i:=(\iota_i)_*\circ\varphi,\qquad
        \nu_i^{\mathrm{res}}:=\nu_i\circ q_j:R\longrightarrow K_i.
$$
The defining residual identity gives
$$
        (q_{A_i})_*\circ\Psi_i
             =i_{h_i}\circ(\nu_i^{\mathrm{res}})_*.
$$
Before identifying the residue algebras, the coefficient map and
conjugator are $\kappa_i\nu_iq_j$ and $(\kappa_i)_*(h_i)$.
This supplies $\mathrm{(LR2)}$. Only the conjugator is lifted
at this stage; the coefficient map already carries the constant
frame to itself.

\smallskip
\noindent\emph{The affine transporter.}
Put $u_\lambda=x_\lambda(1)$ and $v_\lambda=\varphi(u_\lambda)$.
For an $R$-algebra $T$, denote their images by
$u_{\lambda,T},v_{\lambda,T}$ and define
$$
\mathcal T_\varphi(T):=
\{c\in G_{\mathrm{ad}}(\Phi,T)\mid
     cu_{\lambda,T}c^{-1}=v_{\lambda,T}
                   \ (\lambda\in\Phi_\rho)\}.
$$
Only these finite families are extended to $T$, not $\varphi$;
no operator on $T$ is needed. The transport equations define a
closed affine $R$-scheme of finite presentation.

\begin{lemma}
\label{lem:uniqueness-transporter}
For every commutative $R$-algebra $T$, the set
$\mathcal T_\varphi(T)$ has at most one element.
\end{lemma}

\begin{proof}
For two transporters $c,d$, the element $d^{-1}c$ centralizes
the standard frame and is one by $\mathrm{(LR3)}$.
\end{proof}

\begin{theorem}[Global frame normalization]
\label{thm:global-frame-normalization}
\label{prop:orbitwise-transporter}
Assume frame rigidity $\mathrm{(LR3)}$.
Suppose that every $\varphi_i^{\mathrm{res}}$ has the ordinary
form above and that $\mathrm{(LR2)}$ applies to each $\Psi_i$.
Then there is a unique $g\in G(R)$ such that
$$
        \varphi(x_\lambda(1))=gx_\lambda(1)g^{-1}
        \qquad(\lambda\in\Phi_\rho).
$$
Its image $c_i:=(\iota_i)_*(g)$ is the unique local frame
transporter and satisfies $(q_{A_i})_*(c_i)=h_i$.
Consequently, $\psi:=i_{g^{-1}}\circ\varphi$ fixes the standard
frame pointwise.
\end{theorem}

\begin{proof}
Property $\mathrm{(LR2)}$ gives $c_i\in G(A_i)$ lifting $h_i$
with $c_i\in\mathcal T_\varphi(A_i)$.
By Lemma~\ref{lem:uniqueness-transporter} and
Proposition~\ref{prop:unique-affine-gluing}, these points give a
unique $g\in\mathcal T_\varphi(R)$ with
$(\iota_i)_*(g)=c_i$ for every $i$.
Both frames are fixed by $\sigma_R$, so $\sigma_R(g)$ is another
transporter. Uniqueness gives $\sigma_R(g)=g$, hence $g\in G(R)$.
The local reductions and fixation of the frame by $\psi$ follow
from the transport identities.
\end{proof}

\begin{corollary}
\label{cor:local-reduction-after-global-normalization}
For $\psi=i_{g^{-1}}\circ\varphi$ as above, put
$\Psi_i^0:=i_{c_i^{-1}}\circ\Psi_i$. Then
$$
        (\iota_i)_*\circ\psi=\Psi_i^0,\qquad
        (q_{A_i})_*\circ\Psi_i^0=(\nu_i^{\mathrm{res}})_*.
$$
Moreover, $\psi^\sharp=\varphi^\sharp$ and
$\psi_i^{\mathrm{res}}=(\nu_i)_*$.
For $j=(\varphi^\sharp)^{-1}(i)$, the residual isomorphism of
$\psi^{-1}$ at target $j$ is $(\nu_i^{-1})_*$.
\end{corollary}

\begin{proof}
The identities $(\iota_i)_*(g)=c_i$ and $(q_{A_i})_*(c_i)=h_i$
give the local formulas and $(q_i)_*(g)=h_i$.
Since $(q_i)_*\circ i_g=i_{h_i}\circ(q_i)_*$, conjugation by
$g$ preserves every $N_i$. Thus $\psi^\sharp=\varphi^\sharp$ and
$$
        \psi_i^{\mathrm{res}}
             =i_{h_i^{-1}}\circ\varphi_i^{\mathrm{res}}=(\nu_i)_*.
$$
Inverting this isomorphism gives the assertion about $\psi^{-1}$.
\end{proof}

Both $\psi$ and $\psi^{-1}$ therefore fix the frame and have
coefficient-induced residual maps. Their localizations meet the
hypotheses of $\mathrm{(LR4)}$. The next section uses this
recognition to recover the global coefficient automorphism from
arbitrary root parameters.
\medskip

\section{Coefficient recovery and the standard decomposition}
\label{sec:recovery-of-coefficient-automorphism}

We recover the coefficient action once the standard frame is fixed
and the relative root subgroups are preserved. Addition comes from
the additive root groups, the involution from Weyl conjugation,
and multiplication from rank-two commutators. The argument is
independent of the preceding normalization and therefore also
applies after removal of the exceptional or graph factors in
Sections~\ref{sec:exceptional-unitary-types}
and~\ref{sec:graph-automorphisms-2D4}.

We consider the order-two types ${}^2A_\ell$ ($\ell\geq3$),
${}^2D_\ell$ ($\ell\geq4$), and ${}^2E_6$, with the notation
of Section~\ref{sec:preliminaries-local-models}.
Thus $\theta^2=\operatorname{id}_R$, $2\in R^\times$,
$E(R)=E_\sigma(\Phi,R)$, and $G(R)=G_\sigma(\Phi,R)$.
Triality is treated separately in Section~\ref{sec:triality}.

\subsection{Preservation of the relative root subgroups}
\label{subsec:preservation-relative-root-subgroups}

Local recognition for an automorphism and its inverse will give
global root-subgroup preservation.

\begin{lemma}
\label{lem:local-detection-relative-root-subgroups}
Let $\lambda\in\Phi_\rho$ and $y\in G_{\mathrm{ad}}(\Phi,R)$.
If $(\iota_i)_*(y)\in X_\lambda(A_i)$ for every $i\in\mathcal I$,
then $y\in X_\lambda(R)$.
\end{lemma}

\begin{proof}
Let $\mathbf U_\lambda$ be the closed absolute unipotent subgroup
generated by the root groups in the class $\lambda$.
Since every maximal localization of $R$ is a factor of some $A_i$,
Lemma~\ref{lem:affine-descent-and-local-detection} first gives
$y\in\mathbf U_\lambda(R)$.
Its ordered absolute coordinates therefore lie in $R$.
Over every $A_i$ they satisfy the relative coordinate formulas of
Subsection~\ref{subsec:relative-root-subgroups}, including the
fixed-coordinate condition for $A_1$ and
$t\overline t=u+\overline u$ for $A_2$.
Equivariance of $\iota_i$ and local detection give the same
identities over $R$, so $y=x_\lambda(a)$ for a unique
$a\in R_\lambda$.
\end{proof}

\begin{proposition}[Local recognition implies root preservation]
\label{prop:preservation-relative-root-subgroups}
Let $\psi\in\operatorname{Aut}E(R)$. Suppose that, for every
$i\in\mathcal I$ and $\lambda\in\Phi_\rho$, both maps
$(\iota_i)_*\circ\psi$ and $(\iota_i)_*\circ\psi^{-1}$ carry
$X_\lambda(R)$ into $X_\lambda(A_i)$.
Then $\psi(X_\lambda(R))=X_\lambda(R)$ for every $\lambda$.
There are consequently unique automorphisms of parameter groups
$\eta_\lambda:R_\lambda\to R_\lambda$ such that
$$
        \psi(x_\lambda(a))=x_\lambda(\eta_\lambda(a))
        \qquad(a\in R_\lambda).
$$
\end{proposition}

\begin{proof}
The preceding lemma, applied to the images of root elements under
$\psi$ and $\psi^{-1}$, gives both inclusions. The injective
root parameterizations identify the restrictions of $\psi$ with
the asserted parameter-group automorphisms.
\end{proof}

If $\psi$ fixes the frame, these local inclusions follow from
root-subgroup recognition $\mathrm{(LR4)}$ whenever both maps have
equivariant coefficient reductions. Once preservation is known,
coefficient recovery takes place entirely over $R$.

\subsection{The common parameter map}
\label{subsec:common-parameter-map}

We first identify the additive parameter maps by Weyl conjugation,
then recover multiplication and, in type $A_2$, the second coordinate.

\begin{theorem}[Coefficient recovery from a fixed frame]
\label{thm:coefficient-recovery-fixed-frame}
Let $\psi\in\operatorname{Aut}E(R)$ satisfy
$$
        \psi(x_\lambda(1))=x_\lambda(1),\qquad
        \psi(X_\lambda(R))=X_\lambda(R)
        \qquad(\lambda\in\Phi_\rho).
$$
Then there is a unique $\mu\in\operatorname{Aut}(R,\theta)$
such that $\psi=\mu_*$. In particular,
$$
        \psi(x_\lambda(a))=x_\lambda(\mu_\lambda(a))
        \qquad(\lambda\in\Phi_\rho,\ a\in R_\lambda),
$$
where $\mu_\lambda$ is the parameter map of
Subsection~\ref{subsec:relative-root-subgroups}.
\end{theorem}

\begin{proof}
Let $\eta_\lambda$ be the automorphism of $R_\lambda$ induced
by $\psi|_{X_\lambda(R)}$.
For types $A_1$ and $A_1^2$, it is additive and hence
$\mathbb Z[1/2]$-linear; fixation of the frame gives
$\eta_\lambda(1)=1$.
For type $A_2$, it fixes $(1,1/2)$.

\smallskip
\noindent\emph{Weyl conjugation and the involution.}
The standard relative Weyl representatives are words in the frame
and hence fixed by $\psi$. Their negative-root factors use
$x_{-\lambda}(-1)=x_{-\lambda}(1)^{-1}$ for additive roots and
$$
        x_{-\lambda}(-2,2)=x_{-\lambda}(1,1/2)^{-2}
$$
for type $A_2$. The latter equality follows from the parameter
law in Subsection~\ref{subsec:relative-root-subgroups}.

If a product $n$ of these representatives carries an additive
root $\lambda$ to $\lambda'$, its action in the chosen coordinates
is
$$
        n x_\lambda(t)n^{-1}
        =x_{\lambda'}(c\theta^j(t)),
        \qquad c\in\{\pm2^k\mid k\in\mathbb Z\},\quad j\in\{0,1\},
$$
with $j=0$ for type $A_1$.
These are the absolute Weyl formulas in relative coordinates;
the powers of $2$ come from the Weyl words for $A_2$-classes;
see \cite[Subsection~2.4]{BuninaMakadiya2ALarge}.
Since $n$ is fixed, applying $\psi$ gives
$$
        \eta_{\lambda'}(c\theta^j(t))
        =c\theta^j(\eta_\lambda(t)).
$$

For each class $\lambda=\{\alpha,\rho(\alpha)\}$ of type $A_1^2$,
there is an element of $W(\Phi)^\rho$ interchanging its two
absolute roots. In types $A_\ell$ and $E_6$, take
$s_\alpha s_{\rho(\alpha)}w_0$, where $w_0=-\rho$.
In type $D_\ell$, the two roots have opposite $e_\ell$-coordinates
and the same $e_i$-coordinate. Changing the signs of $e_\ell$
and one $e_j$ with $j\ne i,\ell$ interchanges them.
A lift in the subgroup generated by the fixed relative Weyl
representatives acts on $X_\lambda(R)$ by $t\mapsto c\overline t$.
The preceding equality and $\mathbb Z[1/2]$-linearity therefore give
$$
        \eta_\lambda(\overline t)=\overline{\eta_\lambda(t)}.
$$
Weyl transitivity on roots of each type now identifies their
parameter maps, since the factors $c$ and $\theta^j$ commute
with these maps.
Thus all $A_1^2$-groups have one additive bijection $\mu:R\to R$,
and all $A_1$-groups, when present, have one additive bijection
$\mu_0:R_0\to R_0$. They satisfy
$$
        \mu(1)=1,\qquad \mu(\overline t)=\overline{\mu(t)},
        \qquad \mu_0(1)=1.
$$

It remains to compare the two root types and recover multiplication.
We use $[x,y]=xyx^{-1}y^{-1}$. Ordered absolute unipotent
coordinates make the rank-two root factorizations below unique.

\smallskip
\noindent\emph{Roots of types $A_1$ and $A_1^2$.}
This case covers ${}^2A_{2n-1}$, ${}^2D_\ell$, and ${}^2E_6$.
Choose a relative $B_2$-subsystem with long simple root
$\alpha\sim A_1$ and short simple root $\beta\sim A_1^2$.
In compatible root coordinates, the commutator formulas of
\cite[Subsection~3.3]{GargeMakadiyaNormal} are
$$
\begin{aligned}
 {}[x_\beta(t),x_\alpha(s)]
 &=x_{\alpha+\beta}(\varepsilon_1st)
   x_{\alpha+2\beta}(\varepsilon_2st\overline t),\\
 [x_\beta(t),x_{\alpha+\beta}(u)]
 &=x_{\alpha+2\beta}
   \bigl(\varepsilon_3(t\overline u+\overline t\,u)\bigr),
\end{aligned}
$$
where $s\in R_0$, $t,u\in R$, and $\varepsilon_i\in\{\pm1\}$.
Changes of coordinates by signs or $\theta$ commute with the
parameter maps just obtained.

Apply $\psi$ to the first identity and compare the
$X_{\alpha+\beta}$-parameter. This gives
$\mu(st)=\mu_0(s)\mu(t)$.
Taking $t=1$ identifies $\mu_0$ with $\mu|_{R_0}$, and hence
\begin{equation}
\label{eq:multiplicativity-fixed-arbitrary}
        \mu(st)=\mu(s)\mu(t)\qquad(s\in R_0,\ t\in R).
\end{equation}
Thus multiplication is preserved whenever one factor is fixed.
The second commutator identity gives
$$
        \mu(t\overline u+\overline t\,u)
        =\mu(t)\overline{\mu(u)}+\overline{\mu(t)}\mu(u).
$$
These two identities recover arbitrary multiplication through
$$
        2tu=(t+\overline t)u+(u+\overline u)t
                         -(t\overline u+\overline t\,u).
$$
The first two products have fixed factors, and the last term is
the expression preserved by the second commutator. Applying $\mu$
and using equivariance gives $2\mu(tu)=2\mu(t)\mu(u)$, hence
\begin{equation}
\label{eq:multiplicativity-anti-invariant}
        \mu(tu)=\mu(t)\mu(u)\qquad(t,u\in R).
\end{equation}

\smallskip
\noindent\emph{Roots of types $A_1^2$ and $A_2$.}
This case occurs in ${}^2A_{2n}$, $n\geq2$.
In a terminal relative $B_2$-subsystem write
$$
        x_1=x_\alpha,\quad x_2=x_\beta,\quad
        x_3=x_{\alpha+\beta},\quad x_4=x_{\alpha+2\beta},
$$
where $\alpha\sim A_1^2$ and $\beta\sim A_2$.
Use the five-dimensional charts of
\cite[Subsection~2.2]{BuninaMakadiya2ASmall}, with the $A_2$
convention of Subsection~\ref{subsec:relative-root-subgroups}.
Put
$$
        q_j(t):=x_j(t,t\overline t/2)\quad(j=2,3),
        \qquad w_1:=w_\alpha(1).
$$
We first determine the images of $q_j(t)$ and recover multiplication;
the central coordinate is treated afterwards.
The chart identities are
$$
\begin{aligned}
        w_1x_2(t,u)w_1^{-1}&=x_3(t,u),\\
        [x_1(a),q_2(t)]&=q_3(at)x_4(-at\overline t/2),\\
        [x_{-1}(t),x_4(1/2)]&=x_2(0,(t-\overline t)/2).
\end{aligned}
$$
On all signed $A_1^2$-subgroups the parameter map is already $\mu$.

Set $t=1$ in the mixed commutator and apply $\psi$.
Since $q_2(1)$ is fixed and $\mu(a/2)=\mu(a)/2$, we obtain
$$
        \psi(q_3(a))x_4(-\mu(a)/2)
        =q_3(\mu(a))x_4(-\mu(a)/2).
$$
Cancellation determines $\psi(q_3(a))$; conjugation by the fixed
$w_1^{-1}$ then gives
\begin{equation}
\label{eq:image-of-special-A2-parameter}
        \psi(q_j(t))=q_j(\mu(t))
        \qquad(j=2,3,\ t\in R).
\end{equation}
No multiplicativity has been used here: the second coordinate on
the right is $\mu(t)\overline{\mu(t)}/2$.
Now apply $\psi$ to $[x_1(a),q_2(t)]$ with arbitrary $a,t$.
The first parameter of its $X_3$-factor is $\mu(at)$ when the
right-hand side is transformed, and $\mu(a)\mu(t)$ when the
commutator of the images is evaluated. Uniqueness of the ordered
factorization gives $\mu(at)=\mu(a)\mu(t)$.

For $z\in R^-_\theta$, the last chart identity, together with
$\mu(1/2)=1/2$ and $\overline{\mu(z)}=-\mu(z)$, gives
\begin{equation}
\label{eq:image-of-central-A2-parameter}
        \psi(x_\beta(0,z))=x_\beta(0,\mu(z))
        \qquad(z\in R^-_\theta).
\end{equation}

Finally, for $(t,u)\in\mathcal A(R)$ put
$z=u-t\overline t/2\in R^-_\theta$ and use
$x_\beta(t,u)=q_2(t)x_\beta(0,z)$.
The two image formulas, together with multiplicativity, give
$$
        \psi(x_\beta(t,u))=x_\beta(\mu(t),\mu(u)).
$$
The pair is admissible because $\mu$ is multiplicative and
equivariant. Its coefficient action $\mu_*$ now agrees with $\psi$
on $X_\beta(R)$ and on all standard Weyl representatives, hence
on every $A_2$-group by Weyl conjugation. This includes the
negative charts: their change $u\mapsto\overline u$ commutes
with $\mu$.

In both cases the additive bijection $\mu$ preserves $1$,
multiplication, and $\theta$, and gives the action on every
relative root subgroup. Root generation yields $\psi=\mu_*$;
its restriction to any $A_1^2$-subgroup determines $\mu$ uniquely.
\end{proof}

\subsection{The local--global automorphism theorem}
\label{subsec:local-global-automorphism-theorem}

Combining coefficient recovery with the frame normalization of
Section~\ref{sec:global-frame-normalization} gives the ordinary
standard decomposition.

\begin{theorem}[Local--global automorphism theorem]
\label{thm:local-global-automorphism}
Let $\varphi\in\operatorname{Aut}E(R)$ satisfy the hypotheses
of Theorem~\ref{thm:global-frame-normalization}, and assume
root-subgroup recognition $\mathrm{(LR4)}$.
Then there is a unique pair
$$
        g\in G(R),\qquad \mu\in\operatorname{Aut}(R,\theta)
$$
such that $\varphi=i_g\circ\mu_*$. Conversely, every such pair
defines an automorphism of $E(R)$.
\end{theorem}

\begin{proof}
Let $g$ be the frame transporter of
Theorem~\ref{thm:global-frame-normalization}, and put
$\psi=i_{g^{-1}}\circ\varphi$.
Both $\psi$ and $\psi^{-1}$ fix the frame.
By Corollary~\ref{cor:local-reduction-after-global-normalization},
their residual isomorphisms are induced by equivariant ring
isomorphisms. Hence $\mathrm{(LR4)}$ applies to all their
localizations. Proposition~\ref{prop:preservation-relative-root-subgroups}
and Theorem~\ref{thm:coefficient-recovery-fixed-frame} give
$\psi=\mu_*$ for a unique $\mu\in\operatorname{Aut}(R,\theta)$.

In any such decomposition, $g$ transports the frame, so frame
rigidity determines $g$ and coefficient recovery determines $\mu$.
Conversely, $G(R)$ normalizes $E(R)$, and every equivariant
ring automorphism induces an automorphism of $E(R)$.
\end{proof}

\begin{corollary}
\label{cor:automorphism-group}
If the hypotheses of Theorem~\ref{thm:local-global-automorphism}
hold for every automorphism of $E(R)$, then
$$
        \operatorname{Aut}E(R)
        \simeq G(R)\rtimes\operatorname{Aut}(R,\theta),
        \qquad \mu\cdot g:=\mu_*(g).
$$
\end{corollary}

\begin{proof}
The unique decomposition gives the bijection
$(g,\mu)\mapsto i_g\circ\mu_*$. The identity
$$
        (i_g\circ\mu_*)\circ(i_h\circ\nu_*)
        =i_{g\mu_*(h)}\circ(\mu\circ\nu)_*
$$
is precisely the multiplication rule for the indicated semidirect
product.
\end{proof}
\medskip


\section{Rigidity of the standard frame}
\label{sec:rigidity-standard-frame}

We prove the frame rigidity property $\mathrm{(LR3)}$ used in
Sections~\ref{sec:global-frame-normalization}
and~\ref{sec:recovery-of-coefficient-automorphism}:
$$
        C_{G_{\mathrm{ad}}(\Phi,S)}(\mathcal F_\Phi(S))=1
$$
for every commutative ring $S$ satisfying the standing invertibility
assumptions. No coefficient operator on $S$ is needed: the frame
is constant, and its centralizer is taken in the ambient split group.

For the classical types, we first remove projective scalars and
then compute the matrix centralizer. In the unitary case it is
scalar; in the orthogonal cases it acts scalarly on two orthogonal
summands, and determinant one makes the two scalars equal.
For triality and ${}^2E_6$, the required matrix centralizers are
available over arbitrary bases in the corresponding local papers.

Write $U_\lambda$ for the standard matrix of $x_\lambda(1)$
and $E_{ij}$ for matrix units. All group notation denotes points
of group schemes, not naive quotients of matrix groups.
After faithfully flat base change, a point of
$G_{\mathrm{ad}}(A_\ell,S)$ lifts to $\mathrm{GL}_{\ell+1}$,
and a point of $G_{\mathrm{ad}}(D_\ell,S)$ lifts to
$\mathrm{SO}_{2\ell}$. When making such a lift, we keep the
notation $S$ for the base and $C$ for the representing matrix.
Projective centralization means
$$
        CU_\lambda C^{-1}=\xi_\lambda U_\lambda,
        \qquad \xi_\lambda\in S^\times.
$$
If $C$ is scalar, its adjoint class is one, and this equality
descends by Lemma~\ref{lem:affine-descent-and-local-detection}.
In the faithful adjoint representation of $E_6$, no lift is needed.

\begin{lemma}
\label{lem:projective-scalar-square-zero}
Let $2\in S^\times$ and $U=I_n+N$, where $N^2=0$ and an
off-diagonal entry of $N$ is $1$ or $-1$.
If $CUC^{-1}=\xi U$ for $C\in\mathrm{GL}_n(S)$ and
$\xi\in S^\times$, then $\xi=1$.
\end{lemma}

\begin{proof}
Conjugating $(U-I_n)^2=0$ gives
$$
        0=(\xi U-I_n)^2
          =(\xi-1)^2I_n+2\xi(\xi-1)N.
$$
The chosen off-diagonal entry gives $2\xi(\xi-1)=0$.
Since $2\xi$ is a unit, $\xi=1$.
\end{proof}

\begin{lemma}
\label{lem:orthogonal-two-parameter-centralizer}
Let $2\in S^\times$, let $n$ be even, and let $J$ be a
nondegenerate symmetric matrix. Suppose that
$$
        Q^2=I_n,\qquad Q^tJ=JQ,
$$
and that the projectors $\pi_\pm=(I_n\pm Q)/2$ have free images
of ranks $n-1$ and $1$, respectively.
If $C=pI_n+qQ$ belongs to $\mathrm{SO}(J)$, then $q=0$.
\end{lemma}

\begin{proof}
The summands $\pi_\pm S^n$ are orthogonal and nondegenerate.
On them, $C$ acts by the scalars $u=p+q$ and $v=p-q$.
Orthogonality gives $u^2=v^2=1$, while determinant one gives
$$
        1=u^{n-1}v=uv.
$$
Thus $v=u^{-1}=u$, and $2q=u-v=0$.
\end{proof}

\subsection{The unitary types \texorpdfstring{${}^2A_\ell$}{2A(l)}}
\label{subsec:rigidity-unitary-types}

We use $G_{\mathrm{ad}}(A_\ell,S)=\mathrm{PGL}_n(S)$, where
$n=\ell+1$ and $\ell\geq3$, with paired indices $i^*:=n+1-i$.
The additive-root matrices generate all matrix units if $n$ is
even. If $n$ is odd, they first determine the action off the
middle line, which one $A_2$-frame element then connects to it.

\begin{proposition}
\label{prop:rigidity-unitary-frame}
For every commutative ring $S$ with $2\in S^\times$ and every
$\ell\geq3$,
$$
        C_{G_{\mathrm{ad}}(A_\ell,S)}(\mathcal F_\Phi(S))=1.
$$
\end{proposition}

\begin{proof}
Let $c$ centralize the frame and choose its matrix lift $C$.
For roots of type $A_1$ or $A_1^2$, the nilpotent part
$N_\lambda=U_\lambda-I_n$ has square zero and an entry $\pm1$.
Lemma~\ref{lem:projective-scalar-square-zero} makes $C$ commute
with these matrices and the algebra they generate.
For paired indices $i,j$ with $i\ne j$ and $j\ne i^*$, the
$A_1^2$-matrix through $\varepsilon_i-\varepsilon_j$ has the form
$$
        N_{ij}=\epsilon_{ij}E_{ij}+\delta_{ij}E_{j^*,i^*},
        \qquad \epsilon_{ij},\delta_{ij}\in\{\pm1\}.
$$

\smallskip
\noindent\emph{Even dimension.}
The $A_1$-matrices give $E_{i,i^*}$ and $E_{i^*,i}$ up to signs,
and their products give every diagonal unit $E_{ii}$.
The equality $E_{ii}N_{ij}E_{jj}=\epsilon_{ij}E_{ij}$ then
supplies all remaining matrix units. Thus the generated algebra
is $M_n(S)$, and $C$ is scalar.

\smallskip
\noindent\emph{Odd dimension.}
Put $i_0=(n+1)/2$ and $W=\langle e_i\mid i\ne i_0\rangle_S$.
For $i\ne i_0$, choose $k\notin\{i,i^*,i_0\}$, possible since
$n\geq5$. Multiplication gives
$$
        N_{ik}N_{k,i^*}
        =\epsilon_{ik}\epsilon_{k,i^*}E_{i,i^*},
$$
all other products of matrix-unit summands being zero.
The opposite products give $E_{i^*,i}$ and then $E_{ii}$.
These diagonal units isolate the remaining $E_{ij}$ on $W$.
Thus the algebra contains $\operatorname{End}_S(W)$, embedded
by zero on $Se_{i_0}$, including when $n=5$.
Commutation with its identity projector and matrix units forces
$$
        C=aI_W\oplus b,\qquad a,b\in S^\times.
$$
An $A_2$-frame matrix has the form
$$
        U_\lambda=I_n+\epsilon E_{i,i_0}
                +\delta E_{i_0,i^*}+\frac{\gamma}{2}E_{i,i^*},
        \qquad \epsilon,\delta,\gamma\in\{\pm1\}.
$$
In $CU_\lambda C^{-1}=\xi_\lambda U_\lambda$, the diagonal
entries give $\xi_\lambda=1$, and the $(i,i_0)$-entry gives
$a=b$. Hence $C$ is scalar in both cases, and $c=1$ by descent.
\end{proof}

\subsection{The orthogonal types \texorpdfstring{${}^2D_\ell$}{2D(l)}}
\label{subsec:rigidity-type-2D}

Let $\ell\geq4$, $n=2\ell$, and use the natural orthogonal model
with basis $e_1,\ldots,e_\ell,e_{-1},\ldots,e_{-\ell}$ and
$$
        J=\sum_{i=1}^{\ell}(E_{i,-i}+E_{-i,i}),\qquad
        V_i=\langle e_i,e_{-i}\rangle_S.
$$
Let $Q$ interchange $e_\ell,e_{-\ell}$ and fix the other basis
vectors. It represents the graph involution and satisfies
$$
        Q^2=I_n,\qquad Q^tJ=JQ,\qquad\det Q=-1.
$$
Writing $W=V_1\oplus\cdots\oplus V_{\ell-1}$, its eigenspaces are
$$
        W\oplus S(e_\ell+e_{-\ell}),\qquad
        S(e_\ell-e_{-\ell}),
$$
of ranks $n-1$ and $1$. We show that the frame centralizer acts
scalarly on these two summands and apply the preceding lemma.

\begin{proposition}
\label{prop:rigidity-type-2D}
For every commutative ring $S$ with $2\in S^\times$ and every
$\ell\geq4$,
$$
        C_{G_{\mathrm{ad}}(D_\ell,S)}(\mathcal F_\Phi(S))=1.
$$
\end{proposition}

\begin{proof}
Let $c$ centralize the frame and lift it to $C\in\mathrm{SO}(J)$
after faithfully flat base change.
For a long root, $U_\lambda-I_n$ has square zero and an entry
$\pm1$, so Lemma~\ref{lem:projective-scalar-square-zero} gives
exact commutation.

The long-root nilpotent parts generate $\operatorname{End}_S(W)$,
embedded by zero on $V_\ell$. Indeed, for distinct $i,j,k<\ell$,
$$
 (E_{i,j}-E_{-j,-i})(E_{j,k}-E_{-k,-j})=E_{i,k}.
$$
Since $\ell-1\geq3$, this gives all off-diagonal units on the
positive part of $W$, and products give the diagonal units.
Multiplying $E_{i,-j}-E_{j,-i}$ and the opposite long-root
matrices by these units isolates the units between positive and
negative coordinates. Their products give the remaining units.
Consequently,
$$
        C=aI_W\oplus D,\qquad
        a\in S^\times,\quad D\in\mathrm{GL}_2(S).
$$

For a short root, take the standard block on $V_i\oplus V_\ell$,
$i<\ell$, in the order $(e_i,e_{-i},e_\ell,e_{-\ell})$:
$$
        U=\begin{pmatrix}
        1&-1&1&1\\
        0&1&0&0\\
        0&-1&1&0\\
        0&-1&0&1
        \end{pmatrix}.
$$
In $CU=\xi UC$, comparison on the complementary part of $W$
gives $\xi=1$. The remaining equations make every row and column
sum of $D$ equal to $a$. Hence
$$
        D=\begin{pmatrix}p&q\\q&p\end{pmatrix},\qquad
        a=p+q,\qquad C=pI_n+qQ.
$$
This is the two-parameter matrix centralizer of
\cite[Lemma~3.3]{BuninaMakadiya2D}, obtained here without locality
or a congruence assumption. Since $C\in\mathrm{SO}(J)$,
Lemma~\ref{lem:orthogonal-two-parameter-centralizer} gives $q=0$.
Thus $C$ is scalar and $c=1$ by descent.
\end{proof}

\subsection{The triality type \texorpdfstring{${}^3D_4$}{3D4}}
\label{subsec:rigidity-triality}

Use the preceding orthogonal model with $\ell=4$, $n=8$.
The matrix $Q$ still interchanges $e_4,e_{-4}$; it does not
represent triality, but describes the matrix centralizer below.
Choose
$$
        \alpha_1=e_1-e_2,\quad \alpha_2=e_2-e_3,\quad
        \alpha_3=e_3-e_4,\quad \alpha_4=e_3+e_4,
$$
so that triality fixes $\alpha_2$ and cycles the other three.
The simple relative roots are $\beta_s=[\alpha_1]$ and
$\beta_\ell=[\alpha_2]$.
Use the root-vector matrices $X_\alpha$ with the signs of
\cite[Subsections~2.1--2.2]{BuninaMakadiyaTriality}.
The four signed simple frame matrices are
$$
\begin{aligned}
 U_s&=(I_8+X_{\alpha_4})(I_8+X_{\alpha_1})(I_8+X_{\alpha_3}),
 &U_\ell&=I_8+X_{\alpha_2},\\
 U_{-s}&=(I_8+X_{-\alpha_4})(I_8+X_{-\alpha_1})(I_8+X_{-\alpha_3}),
 &U_{-\ell}&=I_8+X_{-\alpha_2}.
\end{aligned}
$$
They suffice to determine the centralizer.

\begin{proposition}
\label{prop:rigidity-triality-frame}
For the triality folding and every commutative ring $S$ with
$2,3\in S^\times$,
$$
        C_{G_{\mathrm{ad}}(D_4,S)}(\mathcal F_\Phi(S))=1.
$$
\end{proposition}

\begin{proof}
Let $c$ centralize the frame and choose a lift $C\in\mathrm{SO}(J)$
after faithfully flat base change. Each displayed frame matrix
has trace $8$, so taking traces in
$CU_\gamma C^{-1}=\xi_\gamma U_\gamma$ gives $\xi_\gamma=1$.
Thus $C$ commutes with all four matrices and with
$$
        W_s=U_sU_{-s}^{-1}U_s,\qquad
        W_\ell=U_\ell U_{-\ell}^{-1}U_\ell.
$$
By \cite[Proposition~3.7]{BuninaMakadiyaTriality}, the common
matrix centralizer of $U_s,U_\ell,W_s,W_\ell$ is $SI_8+SQ$
over every commutative ring with $2$ invertible.
Thus $C=pI_8+qQ$, and the orthogonal summands of ranks $7$ and
$1$ give $q=0$ by
Lemma~\ref{lem:orthogonal-two-parameter-centralizer}.
The scalar lift has trivial adjoint class, so $c=1$ by descent.
This also follows directly from
\cite[Lemma~3.8]{BuninaMakadiyaTriality}.
\end{proof}

\subsection{The type \texorpdfstring{${}^2E_6$}{2E6}}
\label{subsec:rigidity-type-2E6}

We use the arbitrary-base centralizer theorem of
\cite[Lemma~5.2]{BuninaMakadiya2E6}, retaining the module notation
for the later normalization. Let
$L_S=L_{\mathbb Z[1/2]}\otimes_{\mathbb Z[1/2]}S$ be the
$78$-dimensional Chevalley module with bracket $\{\ ,\ \}$,
in the $\rho$-stable basis of
\cite[Subsection~1.1]{BuninaMakadiya2E6}.
The graph involution has matrix $Q$, with $Q^2=I_{78}$; put
$$
        \pi_\pm:=\frac{I_{78}\pm Q}{2},\qquad
        L_S^\pm:=\pi_\pm L_S,\qquad
        L_S=L_S^+\oplus L_S^-.
$$
The constant frame matrices commute with $Q$ and preserve the
bracket. Their linear centralizer acts scalarly on the two summands,
but only its identity belongs to the adjoint group.

\begin{proposition}
\label{prop:rigidity-type-2E6}
For every commutative ring $S$ with $2\in S^\times$,
$$
        C_{G_{\mathrm{ad}}(E_6,S)}(\mathcal F_\Phi(S))=1.
$$
\end{proposition}

\begin{proof}
For this frame, \cite[Lemma~5.2]{BuninaMakadiya2E6} gives
$$
        \operatorname{Cent}_{\operatorname{End}_S(L_S)}
                    (\mathcal F_\Phi(S))=S\pi_+\oplus S\pi_-
$$
and proves that its intersection with $G_{\mathrm{ad}}(E_6,S)$
is $\{I_{78}\}$. The result holds over every commutative
$\mathbb Z[1/2]$-algebra, without locality or flatness assumptions,
including characteristic three.
\end{proof}

Propositions~\ref{prop:rigidity-unitary-frame},
\ref{prop:rigidity-type-2D}, \ref{prop:rigidity-triality-frame},
and~\ref{prop:rigidity-type-2E6} prove $\mathrm{(LR3)}$ in every
case. Their arbitrary-base formulations apply in particular to
the tensor products and open intersections used in
Proposition~\ref{prop:unique-affine-gluing}.

\medskip


\section{Residual standardness outside absolute type \texorpdfstring{$D_4$}{D4}}
\label{sec:residual-standardness}

We verify the residual standardness property $\mathrm{(LR1)}$
for ${}^2A_\ell$ ($\ell\geq3$), ${}^2D_\ell$ ($\ell\geq5$),
and ${}^2E_6$. Throughout, $\rho$ has order two and all residue
fields have characteristic different from two. We first distinguish
the coefficient models, then absorb the possible diagram factor
into an equivariant coefficient isomorphism.

For a residue algebra $(K,\theta)$, the fixed subring $k=K^\theta$
is a field. The root formulas and
Lemma~\ref{lem:split-product-identification} identify
$E_{\rho\theta}(\Phi,K)$ with the elementary subgroup of the
following adjoint group over $k$:
$$
\begin{array}{c|c}
 \text{coefficient model}&\text{adjoint group over }k\\ \hline
 K=k,\ \theta=\operatorname{id}
     &\text{split of type }\Phi_\rho\\
 {}[K:k]=2,\ \operatorname{ord}(\theta)=2
     &\text{quasi-split of type }{}^2\Phi\\
 K=k\times k,\ \theta(a,b)=(b,a)
     &\text{split of type }\Phi.
\end{array}
$$
Thus a fixed residue algebra may give a split group, but then its
type is the folded type, not $\Phi$.

For finite fields, use
\cite[Theorems~30 and~36]{SteinbergLectures} for automorphisms
and its Theorem~37 for isomorphisms.
For infinite fields, use the Borel--Tits theorem and its isomorphism
complement \cite[8.11]{BorelTitsAbstractHomomorphisms}; see
\cite[Theorem~1.3 and Section~3.9]{SteinbergAbstractHomomorphisms}.
The groups are absolutely simple and isotropic, their elementary
subgroups are Zariski dense, and the target groups are adjoint.
The recovered special isogenies are central in the present types
and characteristics, hence are isomorphisms between the adjoint
groups; see \cite[Section~2.5]{SteinbergAbstractHomomorphisms}.

For each fixed $\Phi$, these results distinguish the three rows:
the folded group has a different absolute type, while the other
two forms differ in splitness. None of the finite exceptional
isomorphisms identifies these models under our assumptions.
The field theorems also give the inner--diagonal, field, and
diagram factors. In every absolute type considered here,
$$
        \operatorname{Aut}(\operatorname{Dyn}(\Phi))
             =\langle\rho\rangle.
$$
We show that its two elements contribute no independent factor.

\subsection{Fixed residue algebras}
\label{subsec:fixed-residue-algebras}

We treat the two field cases in the table. For a nontrivial
involution, the field theorem is applied over $k=K^\theta$;
recovering the coefficient pair also requires recovering $K/k$.

\begin{proposition}
\label{prop:fixed-residue-standardness}
Let $(K_1,\theta_1)$ and $(K_2,\theta_2)$ be fixed residue
algebras, put $\sigma_i=\rho\theta_i$, and let
$$
        \alpha:E_{\sigma_1}(\Phi,K_1)
                  \xrightarrow{\ \sim\ }E_{\sigma_2}(\Phi,K_2)
$$
be a group isomorphism. Then $\theta_1$ and $\theta_2$ have the
same order, and
$$
        \alpha=i_g\circ\nu_*
$$
for some $g\in G_{\sigma_2}(\Phi,K_2)$ and an equivariant field
isomorphism $\nu:(K_1,\theta_1)\xrightarrow{\ \sim\ }(K_2,\theta_2)$.
\end{proposition}

\begin{proof}
The distinction of the two field models shows that the involutions
are either both trivial or both nontrivial.

If both are trivial, the field theorem for the folded split groups
gives a field isomorphism $\nu:K_1\to K_2$ and an inner--diagonal
factor. The folded adjoint group embeds in
$G_{\sigma_2}(\Phi,K_2)$, so this factor is $i_g$ for an element
$g$ of that group.
The folded types $B_n$, $C_n$, and $F_4$ have no additional
diagram factor in characteristic different from two.
Thus $\alpha=i_g\circ\nu_*$, with equivariance automatic.

Suppose that both involutions are nontrivial, and put
$k_i=K_i^{\theta_i}$.
The field theorem gives an isomorphism $k_1\to k_2$ identifying
the corresponding quasi-split adjoint forms.
The kernel of their Galois action on the absolute Dynkin diagram
determines the quadratic splitting extension $K_i/k_i$.
Hence the recovered field isomorphism extends to
$f:K_1\xrightarrow{\sim}K_2$ with $f\theta_1=\theta_2f$.
After this coefficient identification, the remaining factors give
$$
        \alpha=i_g\circ(\rho^j)_*\circ f_*,
        \qquad g\in G_{\sigma_2}(\Phi,K_2),\quad j\in\{0,1\}.
$$
On the target twisted group, the fixed-point equation implies
$$
        \rho_*\theta_{2,*}=\operatorname{id},
        \qquad \rho_*=\theta_{2,*}.
$$
Thus $(\rho^j)_*\circ f_*=(\theta_2^j\circ f)_*$.
The map $\nu:=\theta_2^j\circ f$ is still equivariant and gives
the required decomposition.
\end{proof}

\subsection{Split residue algebras}
\label{subsec:split-residue-algebras}

Here first projection gives an untwisted group of absolute type
$\Phi$. Switching the coefficient factors realizes its diagram
involution.

Choose one field factor and identify the other with it using the
involution. Then
$$
        K=k\times k,\qquad \theta(a_0,a_1)=(a_1,a_0).
$$
By Lemma~\ref{lem:split-product-identification}, first projection
is an isomorphism on the full groups, with inverse
$\Delta_\rho(h)=(h,\rho(h))$, and restricts to
\begin{equation}
\label{eq:split-residue-elementary-identification}
        \operatorname{pr}_0:E_{\rho\theta}(\Phi,k\times k)
                    \xrightarrow{\ \sim\ }E_{\mathrm{ad}}(\Phi,k).
\end{equation}

\smallskip
\noindent\emph{Switching coefficients and the diagram action.}
For $K_i=k_i\times k_i$ with switching involutions, every
equivariant ring isomorphism $\nu:K_1\to K_2$ has the form
$$
        \nu=\theta_2^j\circ(\nu_0\times\nu_0),
        \qquad j\in\{0,1\},\qquad
        \nu_0:k_1\xrightarrow{\ \sim\ }k_2.
$$
Indeed, it permutes the two primitive idempotents, and equivariance
makes the component field maps equal. Conversely, every displayed
map is equivariant.
Writing $\operatorname{pr}_{i,0}$ for first projection over $K_i$,
we obtain
$$
        \operatorname{pr}_{2,0}\circ\nu_*\circ\Delta_\rho
        =(\rho^j)_*\circ(\nu_0)_*.
$$
For $j=1$, switching the image of $(x,\rho(x))$ makes its first
coordinate $(\nu_0)_*(\rho(x))=\rho_*((\nu_0)_*(x))$.
Thus the two choices of component order account for both diagram
factors.

\begin{proposition}
\label{prop:split-residue-standardness}
Let $(K_1,\theta_1)$ and $(K_2,\theta_2)$ be residue algebras
of local models, put $\sigma_i=\rho\theta_i$, and let
$$
        \alpha:E_{\sigma_1}(\Phi,K_1)
                  \xrightarrow{\ \sim\ }E_{\sigma_2}(\Phi,K_2)
$$
be a group isomorphism. If at least one of $K_1,K_2$ is split,
then both are split and
$$
        \alpha=i_g\circ\nu_*
$$
for some $g\in G_{\sigma_2}(\Phi,K_2)$ and an equivariant ring
isomorphism $\nu:(K_1,\theta_1)\xrightarrow{\ \sim\ }(K_2,\theta_2)$.
\end{proposition}

\begin{proof}
The distinction of the three residual models implies that both
algebras are split.
Write $K_i=k_i\times k_i$ and use first projection to obtain
$$
        \alpha_0:=\operatorname{pr}_{2,0}\circ\alpha\circ\Delta_\rho:
        E_{\mathrm{ad}}(\Phi,k_1)
                  \xrightarrow{\ \sim\ }E_{\mathrm{ad}}(\Phi,k_2).
$$
The field theorem gives
$$
        \alpha_0=i_h\circ(\rho^j)_*\circ(\nu_0)_*,
        \qquad h\in G_{\mathrm{ad}}(\Phi,k_2),\quad j\in\{0,1\},
$$
for a field isomorphism $\nu_0:k_1\xrightarrow{\sim}k_2$.
Set
$$
        g:=\Delta_\rho(h)\in G_{\sigma_2}(\Phi,K_2),\qquad
        \nu:=\theta_2^j\circ(\nu_0\times\nu_0).
$$
The preceding calculation shows that $i_g\circ\nu_*$ induces
$\alpha_0$ after projection. Since projection is an isomorphism,
$\alpha=i_g\circ\nu_*$.
\end{proof}

The two propositions prove $\mathrm{(LR1)}$, including preservation
of fixed and split residue algebras. Absolute type $D_4$ is treated
in Sections~\ref{sec:graph-automorphisms-2D4} and~\ref{sec:triality}.

\medskip


\section{Relative local normalization and root recognition}
\label{sec:relative-local-rigidity}

We prove $\mathrm{(LR2)}$ and $\mathrm{(LR4)}$ for
${}^2A_\ell$ ($\ell\geq5$), ${}^2D_\ell$ ($\ell\geq4$),
and ${}^2E_6$. The source ring is arbitrary: we normalize its
constant frame in a local target, then recognize arbitrary root
images from their relations with that frame. These are the local
inputs for Sections~\ref{sec:global-frame-normalization}
and~\ref{sec:recovery-of-coefficient-automorphism}.

For the classical types, the matrix arguments of the local papers
use source relations, prescribed residues, and the target form,
not injectivity or surjectivity. For ${}^2E_6$, the required
relative results are already
\cite[Propositions~6.2 and~6.7]{BuninaMakadiya2E6}.
A common reduction handles split targets through an auxiliary
local ring. We state it for both orders $m=2,3$, so that
Section~\ref{sec:triality} can use the same argument.

The type-dependent proofs here concern order-two twistings.
For ${}^2D_4$, residual data are assumed ordinary; graph removal
belongs to Section~\ref{sec:graph-automorphisms-2D4}.
The exceptional unitary types are treated in
Section~\ref{sec:exceptional-unitary-types}.

\subsection{Common reductions and split local models}
\label{subsec:common-local-reductions}

Let $(R,\theta_R)$ be arbitrary and $(A,\theta_A)$ a local model.
Put $J_A=\operatorname{Jac}(A)$ and $A_{\mathrm{res}}=A/J_A$,
and abbreviate $E(S)=E_{\sigma_S}(\Phi,S)$ and
$G(S)=G_{\sigma_S}(\Phi,S)$.
For the order-two normalization, the residual hypothesis is
$$
        \Psi:E(R)\longrightarrow E(A),\qquad
        (q_A)_*\Psi=i_{g_{\mathrm{res}}}\circ(\nu_{\mathrm{res}})_*,
$$
where $g_{\mathrm{res}}\in G(A_{\mathrm{res}})$ and
$\nu_{\mathrm{res}}:R\to A_{\mathrm{res}}$ is equivariant and
unital. Removing a frame transporter lifting $g_{\mathrm{res}}$
leaves a frame-fixing map with coefficient reduction, as required
for recognition.

\begin{lemma}[Descent of a local transporter]
\label{lem:descent-local-transporter}
Let $f:A\to S$ be equivariant and faithfully flat. If
$g_S\in G(S)$ satisfies
$$
        g_Sx_\lambda(1)g_S^{-1}
        =f_*\Psi(x_\lambda(1))\qquad(\lambda\in\Phi_\rho),
$$
then it descends to a unique transporter $g_A\in G(A)$, with
$$
        (q_A)_*(g_A)=g_{\mathrm{res}},\qquad
        (i_{g_A^{-1}}\Psi)_{\mathrm{res}}=(\nu_{\mathrm{res}})_*.
$$
No assumption on the residue algebra of $S$ is needed.
\end{lemma}

\begin{proof}
The two pullbacks of $g_S$ to $S\otimes_A S$ transport the same
frames, so frame rigidity and affine descent give a unique
$g_A\in G_{\mathrm{ad}}(\Phi,A)$ with the transport identities.
Twisting those identities gives another transporter, hence
$g_A\in G(A)$. Its reduction and $g_{\mathrm{res}}$ transport
the same residual frame and therefore agree. Removing this
conjugator gives the asserted residual map.
\end{proof}

An ambient adjoint transporter over $S$ is automatically
twisting-fixed by the same uniqueness argument. Thus no separate
form of the descent lemma is needed for ambient transporters.

\smallskip
\noindent\emph{Projection of a split model.}
For $m=\operatorname{ord}(\rho)\in\{2,3\}$, write $A=B^m$
with cyclic coefficient operator, where $B$ is local, and put
$k=B/\operatorname{Jac}(B)$.
The following is an immediate consequence of
Lemma~\ref{lem:split-product-identification}.

\begin{lemma}
\label{lem:split-local-model-identification}
\label{lem:split-local-reduction}
First projection identifies $G(A),E(A)$ with
$G_{\mathrm{ad}}(\Phi,B),E_{\mathrm{ad}}(\Phi,B)$, respectively,
with inverse
$$
        \Delta_\rho(c)=(c,\rho(c),\ldots,\rho^{m-1}(c)).
$$
For $\Psi:E(R)\to E(A)$, put
$$
        \Psi_{\mathrm{pr}}=\operatorname{pr}_0\Psi,\qquad
        Y_\lambda(B)=\operatorname{pr}_0(X_\lambda(A)).
$$
A transporter of the projected frames lifts by $\Delta_\rho$
to a transporter over $A$, with the corresponding residual value.
Moreover,
$$
        \Psi(X_\lambda(R))\subseteq X_\lambda(A)
        \quad\Longleftrightarrow\quad
        \Psi_{\mathrm{pr}}(X_\lambda(R))\subseteq Y_\lambda(B),
$$
where $Y_\lambda(B)$ has independent absolute root coordinates
in the class $\lambda$.
\end{lemma}

\begin{proof}
The projection identifications over $B$ and $k$ preserve the
frames, root subgroups, and reduction maps.
\end{proof}

Projection identifies the normalization problem but does not solve
it. The next lemma replaces the cyclic product by a local ring
whose components have the same residue. It needs the residual
coefficient rule only on the fixed source.

\begin{lemma}[From fixed to split frame normalization]
\label{lem:fixed-to-split-frame-normalization}
Let $m=\operatorname{ord}(\rho)\in\{2,3\}$ and $A=B^m$ be a split local model with
$k=B/\operatorname{Jac}(B)$, and $\Psi:E(R)\to E(A)$ a
homomorphism. Put $R_0=R^{\theta_R}$ and
$E_0=E_\rho(\Phi,R_0)$, with trivial coefficient operator,
and let $j_*:E_0\to E(R)$ be induced by inclusion.
Suppose that
$$
        (q_A)_*\Psi j_*
          =i_{g_{\mathrm{res}}}\circ(\Delta_k\circ f)_*,
        \qquad g_{\mathrm{res}}\in G(k^m),
$$
where $f:R_0\to k$ is unital and $\Delta_k:k\to k^m$ is diagonal.
Assume that $E_0$ is perfect and that fixed-target frame
normalization holds for maps from $E_0$ with coefficient
reduction and trivial operator on the target residue field.
Then a unique $g_A\in G(A)$ lifts $g_{\mathrm{res}}$ and satisfies
$$
        g_Ax_\lambda(1)g_A^{-1}=\Psi(x_\lambda(1))
        \qquad(\lambda\in\Phi_\rho).
$$
\end{lemma}

\begin{proof}
Lift $\operatorname{pr}_0(g_{\mathrm{res}})$ to
$h_B\in G_{\mathrm{ad}}(\Phi,B)$ by split Bruhat decomposition:
root parameters and torus units lift, and Weyl representatives
are constant. Set
$$
        \beta=i_{h_B^{-1}}\Psi_{\mathrm{pr}},\qquad
        \beta_0=\beta j_*.
$$
Normality of the split elementary subgroup makes these maps
into $E_{\mathrm{ad}}(\Phi,B)$. The reduction of $\beta_0$ is
$f_*:E_0\to E_\rho(\Phi,k)$ and is therefore fixed by $\rho$.

\smallskip
\noindent\emph{A local ring with diagonal residue.}
Put $J=\operatorname{Jac}(B)$ and
$$
 D_m:=B^m\times_{k^m}k
     =\{(b_0,\ldots,b_{m-1})\in B^m\mid
                         b_0+J=\cdots=b_{m-1}+J\},
$$
where $k\to k^m$ is diagonal. A tuple is a unit exactly when
its common residue is nonzero. Thus $D_m$ is local with radical
$J\times\cdots\times J$, and cyclic permutation gives a coefficient operator
$\theta_{D_m}$ inducing the identity on its residue field $k$.
Affineness identifies group points over $D_m$ with tuples of
points over $B$ having the same reduction. Hence
$$
        \Gamma(z)=\bigl(\beta_0(z),\rho(\beta_0(z)),\ldots,
                                      \rho^{m-1}(\beta_0(z))\bigr)
$$
defines a homomorphism $E_0\to G(D_m)$ with reduction $f_*$.

\smallskip
\noindent\emph{The image is elementary.}
Write $E_D=E(D_m)$ and $H_D=\mathbf T(D_m)^{\rho\theta_{D_m}}$.
The fixed torus normalizes the relative root subgroups, so
$E_DH_D/E_D$ is abelian. For a root generator $z$ of $E_0$,
the residue of $\Gamma(z)$ is a root element and lies in the
absolute big cell $\mathbf U^+\mathbf T\mathbf U^-$;
see \cite[Theorem~5.1.13]{ConradReductive}.
Since $D_m$ is local, $\Gamma(z)$ itself lies in this cell.
Its unique factors are twisting-fixed. Collecting the absolute
unipotent coordinates by relative classes puts the unipotent
factors in $E_D$ and the torus factor in $H_D$.
For an $A_2$-class this uses its complete three-root subgroup
and the equation $t\overline t=u+\overline u$.
Thus $\Gamma(E_0)\subseteq E_DH_D$; perfectness of $E_0$
then gives $\Gamma(E_0)\subseteq E_D$.

\smallskip
\noindent\emph{The transporter.}
The fixed-target input applies to $\Gamma$ with residual
conjugator $1$. Let $d_D\in G(D_m)$ be its transporter
reducing to $1$, and let $c_B$ be its first component.
The image of $j_*$ contains every constant frame element, so
$$
        c_B\operatorname{pr}_0(x_\lambda(1))c_B^{-1}
            =\beta(x_\lambda(1)),\qquad c_B\equiv1\pmod J.
$$
Therefore $g_A=\Delta_\rho(h_Bc_B)$ has the required frame
identities and reduction. Frame rigidity gives uniqueness.
\end{proof}

For the order-two types treated here, equivariance of
$\nu_{\mathrm{res}}$ gives its diagonal restriction to $R_0$.
The folded groups $E_0$ have types $C_r$, $B_r$, $B_{\ell-1}$,
or $F_4$, of rank at least three, and are perfect by the
rank-two identities with $2$ invertible. Thus fixed-target
$\mathrm{(LR2)}$ implies split-target $\mathrm{(LR2)}$.
For triality the fixed source has type $G_2$ and is perfect with
$2,3$ invertible. A trace-preserving residual map also sends fixed
coefficients to the diagonal; Section~\ref{sec:triality} will apply
the lemma without assuming equivariance on the whole source.

Root recognition below is carried out over the entire coefficient
pair $(A,\theta_A)$, including split models. The form equations
then relate components, rather than restricting either separately.
A common block argument is
$$
        U^2=I,\quad U\equiv I\pmod{J_A}
        \quad\Longrightarrow\quad U=I:
$$
indeed, $U+I$ has invertible reduction and
$(U-I)(U+I)=0$.

\subsection[The unitary types]{The unitary types \texorpdfstring{${}^2A_\ell$}{2A(l)}}
\label{subsec:relative-local-rigidity-unitary}

Let $\ell\geq5$ and put $N=\ell+1$.
We use the natural projective realization of
\cite[Sections~4--8]{BuninaMakadiya2ALarge}. After normalizing
the frame, it suffices to recognize one $A_1^2$-subgroup; the
remaining root types follow from rank-two relations.

\begin{proposition}
\label{prop:relative-local-rigidity-unitary-stable}
For ${}^2A_\ell$ with $\ell\geq5$, relative frame normalization
$\mathrm{(LR2)}$ and root-subgroup recognition $\mathrm{(LR4)}$
hold over every local model.
\end{proposition}

\begin{proof}
Write $\Psi:E(R)\to E(A)$, with the residual data fixed above.

\smallskip
\noindent\emph{Frame normalization.}
Suppose first that $A$ is fixed. The Bruhat lifting argument in
\cite[Section~4]{BuninaMakadiya2ALarge} gives a lift
$u\in G(A)$ of $g_{\mathrm{res}}$: root parameters and torus
units lift, using averaging for fixed coordinates. Put
$\Psi_1=i_{u^{-1}}\Psi$, whose reduction is
$(\nu_{\mathrm{res}})_*$.

The representative construction of the local Lemma~4.2 applies
with arbitrary source. Write a source element in the natural
simply connected elementary group as a product of commutators,
and lift the images of their factors to the corresponding target
group. The residual discrepancies are central and disappear in
the commutators, giving the prescribed residues.

The normalizations of toral involutions, Weyl elements, and
parameter-one roots in the local Subsections~5.1--5.2, 6.1--6.3,
and 7.1--7.3 now apply. They use the constant source relations and
entries with prescribed unit residues, and give
$$
        C\in\mathrm{GL}_N(A),\qquad C\equiv I_N\pmod{M_N(J_A)},
        \qquad [C]x_\lambda(1)[C]^{-1}=\Psi_1(x_\lambda(1)).
$$
Since $G_{\mathrm{ad}}(A_\ell,A)=\mathrm{PGL}_N(A)$, the element
$g_A=u[C]$ is an ambient adjoint transporter. Both frames are
fixed by the twisting, so frame rigidity gives $g_A\in G(A)$.
Its reduction is $g_{\mathrm{res}}$.
Lemma~\ref{lem:fixed-to-split-frame-normalization} gives the split
case.

For recognition, assume from now on that $\Psi$ fixes the frame
and has reduction $(\nu_{\mathrm{res}})_*$; $A$ may be fixed or
split. The standard toral and Weyl words are then fixed as well.

\smallskip
\noindent\emph{Roots of type $A_1^2$.}
Consider $\Psi(x_{[\alpha_1]}(t))$, $t\in R$.
Choose a determinant-one representative $Y$ with prescribed
residue as in \cite[Subsection~8.1]{BuninaMakadiya2ALarge}.
The inverse relation with $h_{[\alpha_2]}(-1)$ initially holds
up to a scalar $\xi$, where $\xi^N=1$ and $\xi\in1+J_A$.
Its $2$-primary part is trivial, so $\xi$ has odd order.
Rescaling by $\zeta\in\langle\xi\rangle$ with $\zeta^2=\xi^{-1}$
makes the relation exact without changing the determinant or
residue. This uses no invertibility of $N$.

Write $\ell=2r-1$ or $2r$ and use
$V_j=\langle e_j,e_{N+1-j}\rangle$; in the second case add the
middle line. The equations in the local Subsubsections~8.2.1
and~8.3.1 separate $V_1\oplus V_2$ from the complementary blocks.
The diagonal involutions suffice except for the even cases
$r=3,4$, where the additional fixed signed root elements on the
complementary three- or five-dimensional space are used.
The inverse relation makes each complementary block the identity.
Commuting with the fixed elements at
$[\alpha_1]$, $[\alpha_1]+[\alpha_2]$, and $-[\alpha_2]$ then gives
$$
        Y=\begin{pmatrix}
        1&a&b&0\\
        0&1&0&0\\
        0&0&1&0\\
        0&c&d&1
        \end{pmatrix}\oplus I_{N-4}.
$$
The projective scalars in these commutations are removed by
comparison of the unchanged invertible entries.
The inverse relation gives $2a=2d=0$.
If $Q_N$ is the standard unitary form matrix, the target condition is
$$
        YQ_N\overline Y^{\,t}=\lambda Q_N.
$$
A nondegenerate complementary paired summand gives $\lambda=1$;
the remaining form entries give $c=\overline b$.
Thus $[Y]=x_{[\alpha_1]}(b)$, and fixed Weyl conjugation recognizes
all $A_1^2$-subgroups.

For $A=B\times B$, the same equation says that
$b=(b_0,b_1)$ and $c=(b_1,b_0)$.
The two absolute parameters remain independent, as required by
the projected root family.

\smallskip
\noindent\emph{The remaining roots for odd $\ell$.}
The relative system is $C_r$, and these roots have type $A_1$.
Choose a $C_2$-subsystem with $\gamma,\zeta\sim A_1^2$ and
$\gamma+\zeta=\lambda\sim A_1$.
The commutator formula in
\cite[Subsubsection~8.2.3]{BuninaMakadiya2ALarge} expresses
$x_\lambda(s)$, $s\in R^{\theta_R}$, as a commutator of an
$X_\gamma$-element of parameter $\pm s/2$ with a fixed
$X_\zeta$-element. Since
$[X_\gamma(A),X_\zeta(A)]\subseteq X_\lambda(A)$, recognition
of the two $A_1^2$-subgroups gives the required inclusion.

\smallskip
\noindent\emph{The remaining roots for even $\ell$.}
Use a terminal $B_2$-subsystem and the five-dimensional charts of
\cite[Subsection~2.2]{BuninaMakadiya2ASmall}, as in
Subsection~\ref{subsec:common-parameter-map}. Write
$$
        x_1=x_\alpha,\quad x_2=x_\beta,\quad
        x_3=x_{\alpha+\beta},\quad x_4=x_{\alpha+2\beta},
$$
where $\alpha\sim A_1^2$, $\beta\sim A_2$, and put
$q(t)=x_2(t,t\overline t/2)$.
For the fixed words $w_1=w_\alpha(1)$ and $v_2=w_\beta(1,1/2)$,
$$
\begin{aligned}
 w_1x_1(t)w_1^{-1}&=x_{-1}(-t),&
 v_2x_1(2t)v_2^{-1}&=x_4(t),\\
 w_1x_2(t,u)w_1^{-1}&=x_3(t,u).&&
\end{aligned}
$$
Recognition of $X_1$ gives an additive map $\eta:R\to A$ with
$\eta(1)=1$. The first two Weyl formulas give the same map on
$X_{-1}$ and $X_4$, since $\eta(t/2)=\eta(t)/2$.
Now use the identities
$$
\begin{aligned}
        {}[x_1(t),q(1)]
          &=x_3(t,t\overline t/2)x_4(-t/2),\\
        [x_{-1}(t),x_4(1/2)]
          &=x_2(0,(t-\overline t)/2).
\end{aligned}
$$
Apply $\Psi$ to the first. The factor $x_4(-\eta(t)/2)$ occurs
on both sides and cancels; conjugating by $w_1^{-1}$ gives
$$
        \Psi(q(t))
        =x_2\bigl(\eta(t),\eta(t)\overline{\eta(t)}/2\bigr).
$$
For $z\in R^-_{\theta_R}$, the second identity gives
$$
        \Psi(x_2(0,z))
        =x_2\bigl(0,(\eta(z)-\overline{\eta(z)})/2\bigr).
$$
Both values lie in $X_2(A)$ without using multiplicativity or
equivariance of $\eta$. The factorization
$x_2(t,u)=q(t)x_2(0,u-t\overline t/2)$ recognizes the whole
subgroup, and fixed Weyl conjugation recognizes the other
$A_2$-subgroups. These identities also hold over split models,
completing $\mathrm{(LR4)}$.
\end{proof}

In the excluded ranks ${}^2A_3$ and ${}^2A_4$, a nilpotent
anti-invariant parameter may remain in the frame. Its extraction
and removal are treated in Section~\ref{sec:exceptional-unitary-types}.

\subsection[The orthogonal types]{The orthogonal types \texorpdfstring{${}^2D_\ell$}{2D(l)}}
\label{subsec:relative-local-rigidity-orthogonal}

We use compatible matrix representatives to retain exact source
relations, then normalize the frame and recognize one root of each
length. The relative system is $B_{\ell-1}$: long parameters are
fixed and short parameters lie in the whole ring.
Use the natural model of
\cite[Subsections~2.2--2.5]{BuninaMakadiya2D}, with
$$
 S_{2\ell}=\begin{pmatrix}0&I_\ell\\I_\ell&0\end{pmatrix},
 \qquad V_j=\langle e_j,e_{-j}\rangle.
$$
The graph matrix $Q$ interchanges $e_\ell,e_{-\ell}$ and fixes
all other basis vectors. Thus the twisting is
$M\mapsto Q\overline M Q^{-1}$.
Write $U_\lambda(t)$ for a standard orthogonal root matrix,
$x_\lambda(t)$ for its adjoint class, and $H_i$ for the matrix
of $h_{[\alpha_i]}(-1)$.

\begin{proposition}
\label{prop:relative-local-rigidity-orthogonal}
For ${}^2D_\ell$ with $\ell\geq4$, properties $\mathrm{(LR2)}$
and $\mathrm{(LR4)}$ hold over every local model under their
ordinary equivariant residual hypotheses.
\end{proposition}

\begin{proof}
We first adapt the compatible representatives of
\cite[Lemma~3.2]{BuninaMakadiya2D} to an arbitrary source.

\smallskip
\noindent\emph{Matrix representatives.}
Suppose that $\Psi$ has coefficient reduction
$(\nu_{\mathrm{res}})_*$, as it does after removal of the residual
conjugator. For a coefficient pair $S$, let $E_{\mathrm{nat}}(S)$
be the elementary group in the natural orthogonal representation,
and let $\pi_S:E_{\mathrm{nat}}(S)\to E(S)$ be its projection.
There is a unique homomorphism
$Y:E_{\mathrm{nat}}(R)\to E_{\mathrm{nat}}(A)$ satisfying
$$
\begin{aligned}
        \pi_A(Y(z))&=\Psi(\pi_R(z)),\\
        q_A(Y(z))&=(\nu_{\mathrm{res}})_{\mathrm{nat},*}(z)
                   \qquad(z\in E_{\mathrm{nat}}(R)),
\end{aligned}
$$
where the right-hand map applies $\nu_{\mathrm{res}}$ to matrix entries.

To obtain $Y(z)$, write $z$ as a product of commutators and lift
the images of their factors through $\pi_A$.
The source group is perfect: each root subgroup is inverted by
a long-root toral element, so division of its parameter by $2$
expresses each generator as a commutator.
The discrepancies between the reductions of the chosen lifts
and the coefficient images are central, and hence disappear in
these commutators. This gives the required residue.
Two resulting representatives differ by a scalar $\varepsilon$
with $\varepsilon^2=1$ and $\varepsilon\in1+J_A$, so they are
equal. Uniqueness makes $Y$ a homomorphism.

Thus every source matrix identity remains exact, over fixed
and split targets alike. After a change of basis $C$, the matrices
$C^{-1}Y(z)C$ preserve $C^tS_{2\ell}C$; no normalization of the
whole target group is assumed.

\smallskip
\noindent\emph{Frame normalization.}
For a fixed target, lift $g_{\mathrm{res}}$ to $u\in G(A)$ by
\cite[Subsection~3.1]{BuninaMakadiya2D} and put
$\Psi_1=i_{u^{-1}}\Psi$. Use its homomorphic representatives
$Y_1$ just constructed. The toral, Weyl, and root normalizations
in the local Sections~4--6 use only their prescribed residues
and exact source relations, and therefore apply to $Y_1$.
This includes the rank-four support calculation in Section~6.
They give a matrix $C\equiv I_{2\ell}\pmod{M_{2\ell}(J_A)}$
whose projective conjugation transports the complete frame.

To make the final frame identities exact, choose for each root
a long-root toral word $h$ inverting its root subgroup.
Its projective image is fixed. The exact relation
$hxh^{-1}=x^{-1}$ then forces the remaining scalar $\xi$ on
the image of $x=x_\lambda(1)$ to satisfy $\xi^2=1$.
Since $\xi\in1+J_A$, it is one.
Thus $Y_1(U_\lambda(1))C=CU_\lambda(1)$ for every $\lambda$.
Both matrix families are twisting-fixed, so simultaneous averaging
by the local Lemma~2.1 replaces $C$ by
$$
        C_1=\frac{C+Q\overline C Q^{-1}}2,
        \qquad
        C_1\equiv I_{2\ell}\pmod{M_{2\ell}(J_A)},
$$
with the same exact intertwining identities and
$Q\overline{C_1}Q^{-1}=C_1$.

By \cite[Proposition~3.4]{BuninaMakadiya2D}, there is
$d\in J_A\cap A^{\theta_A}$ such that
$C_2=C_1(I_{2\ell}+dQ)$ is a twisting-fixed proper similitude:
$$
        C_2^tS_{2\ell}C_2=\lambda S_{2\ell},\qquad
        \det C_2=\lambda^\ell.
$$
The correcting factor centralizes the frame by the local
Lemma~3.3. Thus $g_A=u[C_2]\in G(A)$ has the required transport
identities and reduction. The correction uses only conjugates of
the frame, not normalization of the entire target group.
Lemma~\ref{lem:fixed-to-split-frame-normalization} gives the split
case, including $\ell=4$.

For recognition, assume that $\Psi$ fixes the frame and has
coefficient reduction. Its homomorphic lift $Y$ then fixes the
frame and its Weyl and toral words as matrices, by uniqueness of
the representatives. Every root image satisfies
$$
        Y^tS_{2\ell}Y=S_{2\ell},\qquad
        Q\overline YQ^{-1}=Y.
$$
We apply the calculations of the local Subsections~7.1--7.2
over the entire local model $A$.

\smallskip
\noindent\emph{A long root.}
Let $Y$ represent $\Psi(x_{[\alpha_1]}(t))$ for
$t\in R^{\theta_R}$. Commutation with $H_1$ separates
$V_1\oplus V_2$ from its complement. Write
$V'=V_4\oplus\cdots\oplus V_\ell$.
Commutation with the fixed root matrices at
$[\alpha_1]+[\alpha_2]$ and $-[\alpha_2]$ eliminates the blocks
$F:V'\to V_3$ and $G:V_3\to V'$. Indeed, the equations are
$$
        P_1F=P_2F=0,\qquad GP_1=GP_2=0,
        \qquad
        P_1=\begin{pmatrix}1&0\\0&0\end{pmatrix},\quad
        P_2=\begin{pmatrix}0&0\\0&1\end{pmatrix}.
$$
Since $P_1+P_2=I_2$, both blocks vanish.
The inverse relation $H_2YH_2^{-1}=Y^{-1}$ makes the blocks on
$V_3$ and $V'$ involutions with identity reduction, so both are
identities. This support argument works for every $\ell\geq4$.
The same commutations leave
$$
        Y=\begin{pmatrix}
        1&a&b&0\\
        0&1&0&0\\
        0&0&1&0\\
        0&c&d&1
        \end{pmatrix}\oplus I_{2\ell-4}.
$$
The inverse relation gives $2a=2d=0$, orthogonality gives $c=-b$,
and twisting gives $b=\overline b$. Thus
$Y=U_{[\alpha_1]}(b)$ with $b\in A^{\theta_A}$.

\smallskip
\noindent\emph{A short root.}
Let $Y$ represent $\Psi(x_{[\alpha_{\ell-1}]}(t))$, $t\in R$.
For $\ell\geq5$, the involutions $H_i$, $1\leq i\leq\ell-3$,
separate $V_1,\ldots,V_{\ell-2}$ from the terminal block
$V_{\ell-1}\oplus V_\ell$. The relation
$H_{\ell-2}YH_{\ell-2}^{-1}=Y^{-1}$ makes each complementary
block an involution with identity reduction, hence the identity.
For $\ell=4$, $H_1$ gives two four-dimensional blocks.
The fixed root elements at
$$
        \pm[\alpha_1],\qquad\pm[\widetilde\alpha_1],\qquad
        [\widetilde\alpha_1]=[\alpha_1]+2[\alpha_2]+2[\alpha_3]
$$
commute with the source short-root subgroup and have scalar
simultaneous centralizer on $V_1\oplus V_2$, as in the local
Section~6. The complementary block is therefore $aI_4$, with
$a\in1+J_A$; the inverse relation gives $a^2=1$, hence $a=1$.

In every rank, commutation with $U_{-[\alpha_{\ell-2}]}(1)$
now gives
$$
        Y=I_{2\ell-4}\oplus Z,\qquad
        Z=\begin{pmatrix}
        1&a&b&c\\
        0&1&0&0\\
        0&d&u&v\\
        0&e&w&z
        \end{pmatrix},\qquad
        K=\begin{pmatrix}u&v\\w&z\end{pmatrix}\equiv I_2\pmod{J_A}.
$$
The inverse relation is
$Z\operatorname{diag}(-I_2,I_2)Z=\operatorname{diag}(-I_2,I_2)$.
Its lower-right block gives $K^2=I_2$, so $K=I_2$.
Orthogonality then gives $e=-b$, $d=-c$, and $a=-bc$;
twisting gives $c=\overline b$. Hence
$$
        Z=\begin{pmatrix}
        1&-b\overline b&b&\overline b\\
        0&1&0&0\\
        0&-\overline b&1&0\\
        0&-b&0&1
        \end{pmatrix},
$$
which is $U_{[\alpha_{\ell-1}]}(b)$ on the terminal block.

Fixed Weyl conjugation recognizes every root of the same length.
On a split model, a long parameter is diagonal, whereas a short
parameter $b=(b_0,b_1)$ retains two independent components.
The short-root matrix projects to the product of the two absolute
root elements with parameters $b_0,b_1$. Thus recognition holds
for both kinds of local model.
\end{proof}

We retain the following explicit rank-four formulation for use
after local graph removal in
Section~\ref{sec:graph-automorphisms-2D4}.

\begin{lemma}[Relative local rigidity in type ${}^2D_4$]
\label{lem:relative-local-rigidity-2D4}
Let $\Psi:E(R)\to E(A)$ in type ${}^2D_4$ satisfy the ordinary
residual hypothesis at the start of this section, with arbitrary
source and a local model as target.
There is a unique frame transporter $g_A\in G(A)$ lifting
$g_{\mathrm{res}}$. The normalized map $\Psi^0=i_{g_A^{-1}}\Psi$
has reduction $(\nu_{\mathrm{res}})_*$ and satisfies
$$
        \Psi^0(X_\lambda(R))\subseteq X_\lambda(A)
        \qquad(\lambda\in\Phi_\rho).
$$
\end{lemma}

\begin{proof}
Apply the proposition with $\ell=4$ and use frame rigidity for
uniqueness. No global graph datum is used: the residual action
is already ordinary.
\end{proof}

\subsection[The type 2E6]{The type \texorpdfstring{${}^2E_6$}{2E6}}
\label{subsec:relative-local-rigidity-2E6}

Here the relative results of the local paper already have the required
arbitrary-source hypotheses. Use the faithful $78$-dimensional
adjoint module, bracket, and graph matrix $Q$ of
Subsection~\ref{subsec:rigidity-type-2E6}, with the coordinates of
\cite[Subsection~1.1]{BuninaMakadiya2E6}.
Thus $\pi_\pm=(I_{78}\pm Q)/2$, $L_A^\pm=\pi_\pm L_A$,
and the twisting is $M\mapsto Q\overline M Q^{-1}$.
The relative system is $F_4$, with long parameters fixed and
short parameters in the whole ring.

\begin{proposition}
\label{prop:relative-local-rigidity-2E6}
For ${}^2E_6$, relative frame normalization $\mathrm{(LR2)}$ and
root-subgroup recognition $\mathrm{(LR4)}$ hold over every local model.
\end{proposition}

\begin{proof}
For a fixed local target $A$, apply
\cite[Proposition~6.7]{BuninaMakadiya2E6} with source $R$ and
target $A$. It gives $g_A\in G(A)$ and a unital equivariant ring
homomorphism $f:R\to A$ satisfying
$$
        \Psi=i_{g_A}\circ f_*,\qquad
        (q_A)_*(g_A)=g_{\mathrm{res}},\qquad
        q_Af=\nu_{\mathrm{res}}.
$$
The element $g_A$ transports the frame because $f_*$ fixes it.
Split-target normalization follows from
Lemma~\ref{lem:fixed-to-split-frame-normalization}.

For recognition, use
\cite[Proposition~6.2]{BuninaMakadiya2E6}. It applies to every
frame-fixing homomorphism $E(R)\to G(A)$ over arbitrary rings
with involution and $2$ invertible, and gives a unique unital
equivariant ring homomorphism $f:R\to A$ with
$$
        \Psi(x_\lambda(t))=x_\lambda(f(t))
        \qquad(\lambda\in\Phi_\rho,\ t\in R_\lambda).
$$
This proves $\mathrm{(LR4)}$, in fact without its residual
hypothesis. The result includes $A=B\times B$ with the exchange
involution: a short parameter $(b_0,b_1)$ projects to
$x_\beta(b_0)x_{\rho\beta}(b_1)$, with no equality imposed on its
two components. Thus both properties hold over every local model.
\end{proof}

The three type-dependent propositions establish the local inputs
for the ordinary global argument. The rank-four orthogonal result
also supplies the local step used after graph removal in
Section~\ref{sec:graph-automorphisms-2D4}.
\medskip


\section{The exceptional unitary types \texorpdfstring{${}^2A_3$}{2A3} and \texorpdfstring{${}^2A_4$}{2A4}}
\label{sec:exceptional-unitary-types}

We classify the automorphisms in the two unitary ranks excluded from
Section~\ref{sec:relative-local-rigidity}. Throughout,
$n\in\{3,4\}$, $\theta^2=\operatorname{id}_R$, $2\in R^\times$,
and
$$
        E^{(n)}(R)=E_\sigma(A_n,R),\qquad
        G^{(n)}(R)=G_\sigma(A_n,R).
$$
The local normalization of \cite{BuninaMakadiya2ASmall} may leave
an anti-invariant nilpotent parameter. We construct its automorphism,
then recover and glue the parameter and conjugating element.
The additional step uses the image equality inherited from a global
automorphism, not surjectivity onto the local target. Removing these
factors reduces the problem to
Section~\ref{sec:recovery-of-coefficient-automorphism}.

\subsection{The exceptional parameter module}
\label{subsec:global-exceptional-parameters}

Put
$$
\mathfrak b_\theta(R)
 :=\langle 3,\ r-\overline r,\ r^3-r\mid r\in R\rangle_R,
\qquad
\mathcal E_{\mathrm{glob}}(R,\theta)
 :=R^-_\theta\cap\operatorname{Ann}_R\mathfrak b_\theta(R).
$$
The ideal $\mathfrak b_\theta(R)$ is invariant; its quotient is the
largest one satisfying $\theta=1$, $3=0$, and $r^3=r$.
The parameter space is an ideal of $R$, since
$\overline{a\delta}=-\overline a\delta=-a\delta$ for every $a\in R$.
Its scalar action factors through $R/\mathfrak b_\theta(R)$.

\begin{lemma}
\label{lem:global-exceptional-basic-relations}
Every $\delta\in\mathcal E_{\mathrm{glob}}(R,\theta)$ satisfies
$$
        3\delta=\delta^2=0,\qquad \delta R^-_\theta=0.
$$
In particular, $\mathcal E_{\mathrm{glob}}(R,\theta)^2=0$.
Moreover, $\delta(u^2-1)=0$ for every $u\in(R^\theta)^\times$.
\end{lemma}

\begin{proof}
By definition, $3\delta=0$ and
$2\delta a=\delta(a-\overline a)=0$ for $a\in R^-_\theta$.
Thus $\delta R^-_\theta=0$, including products of any two exceptional
parameters. For a fixed unit $u$, multiply $\delta(u^3-u)=0$
by $u^{-1}$.
\end{proof}

\begin{lemma}
\label{lem:exceptional-support-localization}
For every $\theta$-invariant multiplicative set $\Sigma$,
$$
        \mathfrak b_\theta(\Sigma^{-1}R)
        =\Sigma^{-1}\mathfrak b_\theta(R).
$$
In particular, localization carries exceptional parameters to
exceptional parameters.
\end{lemma}

\begin{proof}
The generators give $\Sigma^{-1}\mathfrak b_\theta(R)$ contained
in the left-hand ideal. Conversely, the defining identities of
$R/\mathfrak b_\theta(R)$ survive localization: an inverted element
satisfies $s^3=s$, hence $s^2=1$, so $(c/s)^3=c/s$.
Thus every generator on the left vanishes modulo the ideal on the
right. Localizing the annihilator identities proves the last assertion.
\end{proof}

\begin{proposition}
\label{prop:exceptional-support-local-model}
For a fixed or split local model $(A,\theta)$, with
$J_A:=\operatorname{Jac}(A)$,
$$
\mathfrak b_\theta(A)=
\begin{cases}
J_A,&A\text{ fixed and }A/J_A\simeq\mathbb F_3,\\
A,&\text{otherwise}.
\end{cases}
$$
Consequently, $\mathcal E_{\mathrm{glob}}(A,\theta)$ is
$\operatorname{Ann}_A(J_A)\cap A^-_\theta$ in the first case
and zero otherwise. This agrees with the local exceptional
parameter space of \cite[Lemma~3.1]{BuninaMakadiya2ASmall}.
\end{proposition}

\begin{proof}
For a fixed model, every $x\in J_A$ satisfies
$$
        x^3-x=x(x^2-1),\qquad x^2-1\in A^\times,
$$
so $J_A\subseteq\mathfrak b_\theta(A)$.
The image of $\mathfrak b_\theta(A)$ in the residue field is zero
precisely when that field is $\mathbb F_3$: the identities require
$3=0$ and $y^3=y$, and its involution is then automatically trivial.
Otherwise this image contains a unit.
For a split model, $(1,0)-\theta(1,0)=(1,-1)$ is already a unit
in $\mathfrak b_\theta(A)$.
\end{proof}

Likewise, a prime containing $\mathfrak b_\theta(R)$ has quotient
$\mathbb F_3$ and is fixed by $\theta$, and the converse holds.
Thus $V(\mathfrak b_\theta(R))$ is precisely the set of these
maximal ideals and supports every exceptional parameter.

For example, $R=\mathbb F_3[\epsilon]/(\epsilon^2)$ with
$\theta(\epsilon)=-\epsilon$ has
$\mathcal E_{\mathrm{glob}}(R,\theta)=\mathbb F_3\epsilon$.
For $\theta=\operatorname{id}_R$ the module is zero: residue
characteristic three alone is not sufficient.

\subsection{Exceptional root formulas and automorphisms}
\label{subsec:global-exceptional-automorphisms}
\label{subsec:exceptional-coefficient-endomorphism}

We distinguish the polynomial root formulas, which can be written
for $\overline d=-d$, $d^2=3d=0$, from the automorphisms they
define when $d$ is admissible. Use the root numbering, matrices,
and Weyl words of \cite[Section~2]{BuninaMakadiya2ASmall}, with
the conventions of Subsection~\ref{subsec:relative-root-subgroups}.
Square brackets denote projective classes; root and Weyl symbols
in matrix identities denote their standard representatives.
Write $b^\pm=(b\pm\overline b)/2$. In type ${}^2A_3$, put
\begin{equation}
\label{eq:exceptional-A3-x1-general}
X_1^{(3),d}(b):=
\begin{pmatrix}
1&b&-b^+d&-(b^+)^2d\\
0&1&0&-b^+d\\
-b^+d&-(b^+)^2d&1&\overline b\\
0&-b^+d&0&1
\end{pmatrix},
\end{equation}
and, for $s\in R^\theta$,
\begin{equation}
\label{eq:exceptional-A3-x2-general}
X_2^{(3),d}(s):=
\begin{pmatrix}
1+2sd&0&0&0\\
0&1+sd&s-2s^2d&0\\
0&0&1+sd&0\\
0&0&0&1+2sd
\end{pmatrix}.
\end{equation}
In type ${}^2A_4$, put
\begin{equation}
\label{eq:exceptional-A4-x1-general}
X_1^{(4),d}(b):=D(b^+d)x_1(b),\qquad
D(c):=\operatorname{diag}(1+c,1+c,1-c,1+c,1+c),
\end{equation}
and, for $(t,u)\in\mathcal A(R)$ and $p=t^+$, put
\begin{equation}
\label{eq:exceptional-A4-x2-general}
X_2^{(4),d}(t,u):=
\begin{pmatrix}
1&0&-pd&p^2d&0\\
p^2d&1&t&u+p(p^2-1)d&p^2d\\
-pd&0&1&\overline t&-pd\\
0&0&0&1&0\\
0&0&-pd&p^2d&1
\end{pmatrix}.
\end{equation}
The other signed formulas use the fixed Weyl conjugations and their
parameter changes. All matrices are invertible, since they reduce to
standard root matrices modulo the square-zero ideal $Rd$.
Their definitions commute with equivariant base change.

\smallskip
\noindent\emph{The exceptional frame.}
Write $u_1=x_1(1)$ and let $u_2=x_2(1)$ for $n=3$, or
$u_2=x_2(1,1/2)$ for $n=4$.
Set
$$
\begin{aligned}
 X_1^{(n)}(d)&:=X_1^{(n),d}(1),\qquad
 X_2^{(3)}(d):=X_2^{(3),d}(1),\\
 X_2^{(4)}(d)&:=X_2^{(4),d}(1,1/2).
\end{aligned}
$$
Define the remaining positive frame matrices by
$$
\begin{aligned}
 X_3^{(3)}(d)&=w_2X_1^{(3)}(d)w_2^{-1},&
 X_4^{(3)}(d)&=w_1X_2^{(3)}(d)w_1^{-1},\\
 X_3^{(4)}(d)&=w_1X_2^{(4)}(d)w_1^{-1},&
 X_4^{(4)}(d)&=v_2\bigl(X_1^{(4)}(d)\bigr)^2v_2^{-1}.
\end{aligned}
$$
The negative elements are obtained by the standard fixed Weyl words.
Their projective classes form
$\mathcal F_d^{(n)}=(u_\lambda^{(n)}(d))_{\lambda\in\Phi_\rho}$,
with $\mathcal F_0^{(n)}=\mathcal F_\Phi$.
Unlike the arbitrary-parameter formulas, these constant matrices are
polynomials in $d$ over $\mathbb Z[1/2]$: they are defined over every
ring $T$ with $d^2=3d=0$, without an involution on $T$.
The standard Weyl words retain their values on these frames.

\begin{lemma}[Detection of an exceptional root parameter]
\label{lem:detection-exceptional-root-parameter}
Let $(T,\theta)$ be a ring with involution, $2\in T^\times$, and
$d\in T^-_\theta$ with $d^2=3d=0$.
The map $b\mapsto[X_1^{(n),d}(b)]$ is injective.
\end{lemma}

\begin{proof}
The $(1,1)$-entry is $1$ or $1+b^+d$, hence a unit.
The projectively invariant ratio of the $(1,2)$-entry to it is $b$.
\end{proof}

\begin{proposition}[Global exceptional automorphisms]
\label{prop:global-exceptional-automorphisms}
Every $\delta\in\mathcal E_{\mathrm{glob}}(R,\theta)$ defines a
unique automorphism $\Theta_\delta^{(n)}$ of $E^{(n)}(R)$ with
$$
        \Theta_\delta^{(n)}(x_\lambda(t))
        =[X_\lambda^{(n),\delta}(t)].
$$
These automorphisms satisfy
$$
 \Theta_0^{(n)}=\operatorname{id},\qquad
 \Theta_\varepsilon^{(n)}\Theta_\delta^{(n)}
       =\Theta_{\delta+\varepsilon}^{(n)},\qquad
 (\Theta_\delta^{(n)})^{-1}=\Theta_{-\delta}^{(n)},
$$
and, for $\nu\in\operatorname{Aut}(R,\theta)$,
$$
        \nu_*\Theta_\delta^{(n)}\nu_*^{-1}
        =\Theta_{\nu(\delta)}^{(n)}.
$$
They commute with orbit localization and induce the identity
on every maximal congruence quotient $E^{(n)}(K_i)$.
\end{proposition}

\begin{proof}
\smallskip
\noindent\emph{Elementary values.}
For an admissible $\delta$ we have
$\delta r=\delta\overline r=\delta r^+$,
$\delta^2=3\delta=0$, and $\delta(r^3-r)=0$.
The elementary identities in
\cite[Subsection~3.2]{BuninaMakadiya2ASmall} give
$$
\begin{aligned}
 X_1^{(3),\delta}(a)
   &=x_{-3}(a^+\delta)x_1(a)x_3(a^+\delta),\\
 X_2^{(3),\delta}(s)
   &=h_1(1+2s\delta)x_2(s),\\
 D(c)&=h_1(1+c)w_2\left(1,\frac{1-c}{2}\right)v_2
                   \qquad(c=a^+\delta),\\
 X_2^{(4),\delta}(t,u)
   &=x_{-3}(-t\delta,0)x_{-1}(t^2\delta)x_2(t,u)
     x_3(-t\delta,0)x_4(t^2\delta).
\end{aligned}
$$
All correcting parameters are admissible, and $1+r\delta$ is
a unit. The torus and Weyl factors are elementary words.
Thus every prescribed value, including the Weyl conjugates,
belongs to $E^{(n)}(R)$.

\smallskip
\noindent\emph{Relations and composition.}
For each local model put $\delta_i=\iota_i(\delta)$.
On a fixed model, localization and
Proposition~\ref{prop:exceptional-support-local-model} make
$\delta_i$ admissible.
For a split orbit, choose $r\in R$ with residues $1,0$ at its
two maximal ideals. Then $r-\overline r$ is a unit in $B_i$;
the equality $\delta(r-\overline r)=0$ gives $\delta_i=0$.
Thus over every $A_i$ the prescribed values agree with the local
automorphism of \cite[Proposition~3.6]{BuninaMakadiya2ASmall},
or with the identity when $\delta_i=0$.

Let a word in relative root elements represent the identity in
$E^{(n)}(R)$. Replacing its factors by the prescribed elementary
values gives the identity over every $A_i$, and hence over $R$
by affine local detection. Thus all group relations are preserved,
and the formulas define an endomorphism satisfying
$$
        (\iota_i)_*\circ\Theta_\delta^{(n)}
        =\Theta_{\delta_i}^{(n)}\circ(\iota_i)_*.
$$
The local addition law and local detection give the global addition
law. The parameters $0$ and $-\delta$ give the identity and inverse,
so the endomorphism is an automorphism.
Since $\delta_i^2=0$, every residual parameter is zero; hence
$\Theta_\delta^{(n)}$ induces the identity on every $E^{(n)}(K_i)$.
Finally, an equivariant ring automorphism carries the formulas for
$\delta$ to those for $\nu(\delta)$.
Root generation proves the conjugation identity and uniqueness.
\end{proof}

\subsection{Orbitwise exceptional normalization}
\label{subsec:relative-exceptional-normal-form}

For $\Psi_i=(\iota_i)_*\varphi$, the preliminary local calculation
gives a candidate frame modulo a finitely generated defect ideal $I$.
The remaining argument follows the sequence
$$
 \text{frame modulo }I
 \ \longrightarrow\ \varepsilon\in\mu(R)
 \ \longrightarrow\ I/J_AI=0
 \ \longrightarrow\ I=0.
$$
The first lemma recovers coefficients for a prescribed frame.
The second proves admissibility and realizes the parameter by a
source root element, which supplies the final defect equation.

\begin{lemma}[Coefficients for a prescribed exceptional frame]
\label{lem:exceptional-one-sided-coefficients}
Let $(S,\theta_S)$ be a ring with involution and $(B,\theta_B)$ a
fixed or split local model, with $2$ invertible. Suppose that
$\chi:E^{(n)}(S)\to E^{(n)}(B)$ has reduction induced by an
equivariant unital $\nu_0:S\to B/\operatorname{Jac}(B)$,
$$
        \chi(u_j)=[X_j^{(n)}(\varepsilon)]\quad(j=1,2),
        \qquad \overline\varepsilon=-\varepsilon,
        \quad \varepsilon^2=3\varepsilon=0,
$$
and carries $w_1,w_2$ to the corresponding target words for $n=3$,
or $w_1,w_4,v_2$ for $n=4$.
Then a unique equivariant unital ring homomorphism $\mu:S\to B$
lifts $\nu_0$ and gives all root images by
$$
        \chi(x_\lambda(t))
        =[X_\lambda^{(n),\varepsilon}(\mu_\lambda(t))].
$$
Moreover,
$$
\begin{aligned}
 \chi(x_1(\xi))&=x_1(\mu(\xi)),\qquad
 \varepsilon\mu(\xi)=0 &&(\xi\in S^-_{\theta_S}),\\
 \varepsilon(\mu(s)^3-\mu(s))&=0 &&(s\in S).
\end{aligned}
$$
Every image element has a representative $W$ with
$W,W^{-1}\in M_{n+1}(P_0)$, where
$P_0:=\mu(S)+\varepsilon\mu(S)$.
No surjectivity is assumed; the annihilator identities hold here
only on $\mu(S)$.
\end{lemma}

\begin{proof}
Use the calculations of
\cite[Lemmas~4.9 and~4.12]{BuninaMakadiya2ASmall} before their
first use of surjectivity. The representatives of its
Subsubsection~4.1.1 use source commutators and prescribed residues,
so the source may be arbitrary. The commutation and form equations
give additive parameter maps
$\mu_+:S^{\theta_S}\to B^{\theta_B}$ and
$\mu_-:S^-_{\theta_S}\to B^-_{\theta_B}$, with $\mu_+(1)=1$.
Put $\mu(s)=\mu_+(s^+)+\mu_-(s^-)$.
The following comparisons remove the exceptional errors in
multiplicativity.

For $n=3$, compare the $(1,4)$-entries in the commutation of
$\chi(x_1(\xi))$ and $X_1^{(3)}(\varepsilon)$.
This gives $-4\varepsilon\mu_-(\xi)=0$.
The mixed commutator, with $v\in S^{\theta_S}$ and
$p=\mu(s)^+$, gives
$$
        \mu(vs)=\mu_+(v)\mu(s)
                     -\varepsilon\mu_+(v)^2p(p^2-1).
$$
Taking $v=1$ and using $\varepsilon\mu(S^-_{\theta_S})=0$
gives $\varepsilon(\mu(s)^3-\mu(s))=0$.
The correction term therefore vanishes for every $v$.
The relation
$[x_1(\xi),x_3(\zeta)]=x_4(2\xi\zeta)$ for anti-invariant
$\xi,\zeta$ gives multiplicativity on their products.
The decomposition into fixed and anti-invariant parts proves
multiplicativity on $S$.

For $n=4$, write $b=\mu_-(\xi)$ and let $c$ be the additional
entry in the anti-invariant calculation. Its remaining equations are
$$
 c(1/2-2b\varepsilon)=0,\qquad
 -c+bc\varepsilon=-(b+c)\varepsilon.
$$
The first coefficient is a unit, so $c=0$ and $\varepsilon b=0$.
The form equation then gives $\overline b=-b$.
The mixed commutator gives
$$
 \mu(rs)=\mu(r)\mu(s)
       +\varepsilon\mu(r)^2\mu(s)(\mu(s)^2-1).
$$
Taking $r=1$ gives cubic annihilation, and substitution back
gives multiplicativity.

These comparisons use source relations, $2^{-1}$, and inverses of
entries with unit residues. They apply to split models over the
whole coefficient pair, as in
Subsection~\ref{subsec:common-local-reductions}.
Mixed commutators and Weyl words give the other root formulas,
using the special--central factorization for $A_2$.
Their matrices and inverses have entries in the invariant subring
$P_0=\mu(S)+\varepsilon\mu(S)$. Comparing the first root parameter
shows that $\mu$ lifts $\nu_0$, and
Lemma~\ref{lem:detection-exceptional-root-parameter} gives uniqueness.
\end{proof}

The natural coefficient image $\ell(S)$ and the recovered image
$\mu(S)$ need not coincide. The next lemma first recovers the local
ring from $\mu(S)+\varepsilon\mu(S)$, then places $\varepsilon$
in $\mu(S)$ itself. It does not assert surjectivity of $\mu$.

\begin{lemma}[Recovering a parameter after localization]
\label{lem:exceptional-localized-image}
In the preceding lemma, suppose that $B=S_{\mathfrak m}/I$, where
$\mathfrak m$ is fixed by $\theta_S$ and
$I\subseteq\operatorname{Jac}(S_{\mathfrak m})$ is invariant.
Let $\ell:S\to B$ be the natural map and suppose that
$C\in\mathrm{GL}_{n+1}(B)$, $[C]\in G^{(n)}(B)$, and
$$
        \chi(E^{(n)}(S))
        =i_{[C]^{-1}}\bigl(\ell_*(E^{(n)}(S))\bigr).
$$
Then $\varepsilon\in\mathcal E_{\mathrm{glob}}(B,\theta_B)$
and $\varepsilon\in\mu(S)$. In particular, some
$\xi\in S^-_{\theta_S}$ satisfies
$\chi(x_1(\xi))=x_1(\varepsilon)$.
\end{lemma}

\begin{proof}
The case $\varepsilon=0$ is immediate. Otherwise $3\varepsilon=0$
forces residue characteristic three. Hence $N:=n+1\in\{4,5\}$
is a unit of $B$. Put
$$
        P=\mu(S),\qquad P_0=P+\varepsilon P,\qquad
        D=(P_0\cap B^\times)^{-1}P_0\subseteq B.
$$
The ring $D$ is closed under inversion of its elements that are
units in $B$. Recovering $B=D$ will prove admissibility; a separate
use of the group image will then give $\varepsilon\in P$.

\smallskip
\noindent\emph{Recovering the target ring.}
Traces remove the projective scalar, and matrix units recover the
coefficients. For a constant relative root matrix $V$, the image
hypothesis and the preceding lemma give a representative
$W=\lambda C^{-1}VC$ with $W,W^{-1}\in M_N(P_0)$.
Both $V$ and $V^{-1}$ have trace $N$, so
$$
        \lambda=\frac{\operatorname{tr}W}{N}\in D,\qquad
        \lambda^{-1}=\frac{\operatorname{tr}W^{-1}}{N}\in D.
$$
Here $N^{-1}\in D$ since $N\in P_0\cap B^\times$.
Thus $C^{-1}VC\in M_N(D)$.
The constant root matrices and their Weyl words generate $M_N$
over $\mathbb Z[1/2]$ by
\cite[Lemmas~2.1--2.2]{BuninaMakadiya2ASmall}; hence
$C^{-1}E_{pq}C\in M_N(D)$ for every $p,q$.
For $s\in S$, the same trace argument applies to $x_1(\ell(s))$,
and its $(1,2)$-coordinate is recovered by
$$
 \ell(s)=\operatorname{tr}\bigl(
 C^{-1}x_1(\ell(s))C\,C^{-1}E_{21}C\bigr)\in D.
$$
Since $B$ consists of fractions $\ell(r)/\ell(s)$ with
$s\notin\mathfrak m$, closure under these inverses gives $B=D$.

\smallskip
\noindent\emph{Admissibility and the natural image.}
The preceding lemma and $\varepsilon^2=3\varepsilon=0$ give
$$
        \varepsilon(z-\overline z)=0,\qquad
        \varepsilon(z^3-z)=0\qquad(z\in P_0).
$$
Since $P_0$ and the denominator set are invariant,
Lemma~\ref{lem:exceptional-support-localization} carries these
identities to $B=D$, proving admissibility.
Proposition~\ref{prop:exceptional-support-local-model} now gives
$$
        B/\operatorname{Jac}(B)\simeq\mathbb F_3,\qquad
        \varepsilon\operatorname{Jac}(B)=0.
$$
Every unit of $B$ consequently multiplies $\varepsilon$ by its
residue $1$ or $-1$. Writing $\varepsilon=\ell(r)/\ell(s)$ gives
$\ell(r)=\pm\varepsilon$, hence $\varepsilon\in\ell(S)$.

\smallskip
\noindent\emph{Passing to the recovered coefficient image.}
The local exceptional automorphism is now defined, and the root
formulas give $\chi_0:=\Theta_{-\varepsilon}^{(n)}\circ\chi=\mu_*$.
Over the local ring $B$, this exceptional automorphism fixes the
elementary reduction kernel pointwise: its residual cocycle
vanishes at the identity and on the scalar center; see
\cite[Subsection~3.3]{BuninaMakadiya2ASmall}.
Thus the two image subgroups agree on the elementary reduction
kernel. Since $\varepsilon\in\ell(S)$, the image hypothesis puts
the class of
$$
 Z:=C^{-1}x_1(\varepsilon)C=I_N+\varepsilon C^{-1}MC,\qquad
 M=\begin{cases}E_{12}-E_{34},&n=3,\\E_{12}-E_{45},&n=4,
 \end{cases}
$$
in $\chi(E^{(n)}(S))$. It reduces to one, so it also belongs to
$\chi_0(E^{(n)}(S))$.
Choose a root-word representative $W=\lambda Z$ with
$W,W^{-1}\in M_N(P)$.
Since $\operatorname{tr}Z=\operatorname{tr}Z^{-1}=N$, the trace
argument gives $Z\in M_N(P[1/N])$.

The nonzero reduction of $C^{-1}MC$ gives a unit entry, whose
product with $\varepsilon$ is $\pm\varepsilon$.
Hence $\varepsilon\in P[1/N]$ and $N^a\varepsilon\in P$ for
some $a\geq0$. Since $3\varepsilon=0$ and $N\in\{4,5\}$,
we have $N^a\varepsilon=\pm\varepsilon$, so $\varepsilon\in P$.
Choose $t\in S$ with $\mu(t)=\varepsilon$ and set
$\xi=(t-\overline t)/2$. Equivariance gives
$\mu(\xi)=\varepsilon$, and the preceding lemma yields
$\chi(x_1(\xi))=x_1(\varepsilon)$.
\end{proof}

\begin{proposition}[Orbitwise exceptional normalization]
\label{prop:relative-exceptional-normalization}
\label{prop:exceptional-orbit-normalization}
Let $\varphi\in\operatorname{Aut}E^{(n)}(R)$ and put
$\Psi_i=(\iota_i)_*\varphi$.
Use the residual isomorphisms of
Subsection~\ref{subsec:congruence-spectrum} and their standard
form from Section~\ref{sec:residual-standardness}, writing
$$
        (q_{A_i})_*\Psi_i
        =i_{g_i^{\mathrm{res}}}(\nu_i^{\mathrm{res}})_*.
$$
There exist $g_i\in G^{(n)}(A_i)$ and
$\delta_i\in\mathcal E_{\mathrm{glob}}(A_i,\theta)$ such that
$$
 (q_{A_i})_*(g_i)=g_i^{\mathrm{res}},\qquad
 g_i u_\lambda^{(n)}(\delta_i)g_i^{-1}
       =\Psi_i(x_\lambda(1))\quad(\lambda\in\Phi_\rho).
$$
In particular, $\delta_i=0$ on split models. The map
$i_{g_i^{-1}}\Psi_i$ has reduction $(\nu_i^{\mathrm{res}})_*$.
\end{proposition}

\begin{proof}
Fix $i$ and abbreviate $A=A_i$, $\Psi=\Psi_i$.
Surjectivity of $\varphi$ and the residual isomorphism give
$$
        \Psi(E^{(n)}(R))=(\iota_i)_*(E^{(n)}(R)),
        \qquad \nu_i^{\mathrm{res}}:R\twoheadrightarrow A_{\mathrm{res}}.
$$
Indeed, $\nu_i^{\mathrm{res}}=\nu_iq_j$, where
$j=(\varphi^\sharp)^{-1}(i)$ and $\nu_i:K_j\xrightarrow{\sim}K_i$.
The image equality survives a target quotient and conjugation;
no invariance of its kernel under $\varphi$ is required.
Lift and remove $g_i^{\mathrm{res}}$ using
\cite[Subsection~4.1]{BuninaMakadiya2ASmall} for fixed models
and the split identification otherwise; call the lift $g_0$.
The representative and Weyl normalizations of the local
Subsubsection~4.1.1 and Subsection~4.2 use only source relations
and these residues. We retain their defect calculations below.
In the preliminary comparisons, frame matrices are evaluated as
polynomials in a radical parameter before $d^2=3d=0$ is known.

\smallskip
\noindent\emph{A fixed model, $n=3$.}
Here $A=R_{\mathfrak m}$.
The preliminary calculation of
\cite[Lemmas~4.6 and~4.8]{BuninaMakadiya2ASmall} takes place
over an equivariant faithfully flat local extension $A\to S$
with the same residue field. Its parameters $a,b,e$ satisfy
$a\in1+\operatorname{Jac}(S)$ and $a\overline a=a^3=1$.
The partial-trace invariant in the local Lemma~4.10 gives
$$
        \kappa=\frac{2a}{a+1}\in A,\qquad
        a=\frac{\kappa}{2-\kappa}\in A.
$$
Put $\zeta=(a-1)/2$, $I=\zeta^2A$, $B=A/I$, and
$\varepsilon=\zeta+I$.
The identities $\overline\zeta=-\zeta/a$ and
$3a\zeta=-4\zeta^3$ show that $I$ is invariant and
$\overline\varepsilon=-\varepsilon$, $\varepsilon^2=3\varepsilon=0$.
The balancing similitude
$$
        L=I_4+\gamma Q_4,\qquad \gamma=a(b-e)/2,
$$
centralizes the normalized Weyl elements and gives
$X_j\equiv X_j^{(3)}(\zeta)\pmod{M_4(IS)}$ for $j=1,2$.
The candidate matrices and their negative Weyl conjugates are
defined over $B$. Their reductions generate
$M_4(B/\operatorname{Jac}(B))$, so Nakayama's lemma makes their
$B$-algebra equal to $M_4(B)$.
Their conjugates over $S/IS$ are also defined over $B$.
The matrix-algebra descent of the local Lemma~4.7
descends the correcting class. Including $g_0$, the resulting
class $[C_B]\in G^{(3)}(B)$ gives, for $\ell:R\to B$,
$$
        \chi=i_{[C_B]^{-1}}\ell_*\varphi,
        \qquad
        \chi(E^{(3)}(R))=i_{[C_B]^{-1}}(\ell_*(E^{(3)}(R))).
$$
It has the candidate frame and fixes $w_1,w_2$.
Lemma~\ref{lem:exceptional-localized-image} supplies
$\xi\in R^-_\theta$ with $\chi(x_1(\xi))=x_1(\varepsilon)$.
The normalized representative of its image over $S$ is
$$
        Y=I_4+\zeta(E_{12}-E_{34})+D,\qquad D\in M_4(IS).
$$
In $M=IS/\operatorname{Jac}(S)IS$, put $d=[\zeta^2]$ and
write $D_{pq}$ for the entry classes. The exact commutations
with $w_3,X_1,X_4$ give, as in the local Lemma~4.10,
$$
\begin{gathered}
 D_{13}+D_{31}=D_{24}+D_{42}=0,\\
 D_{42}+d=0,\qquad -D_{31}+d=0,\qquad D_{13}-D_{24}-2d=0.
\end{gathered}
$$
Thus $4d=0$. Since $IS$ is generated by $\zeta^2$, Nakayama's
lemma gives $IS=0$, and faithful flatness gives $I=0$.
The balancing and descent identities are now exact over $A$.
Applying Lemma~\ref{lem:exceptional-localized-image} once more,
now over $A$ itself, proves admissibility of $\delta_i=\zeta$.

\smallskip
\noindent\emph{A fixed model, $n=4$.}
Use the representative choices of the local Subsubsection~4.4.2,
so that $(X_1w_1)^3=V_2^2=I_5$, and the column normalization
of \cite[Lemma~4.11]{BuninaMakadiya2ASmall}.
Here $V_2$ is the image of the short Weyl word, and the
transported Hermitian form is retained throughout.
Put $d=1-(X_1)_{33}\in\operatorname{Jac}(A)$.
Before imposing $3d=0$, the local calculation uses the polynomial
representatives
$$
 \widetilde X_1^{(4)}(d)=X_1^{(4)}(d),\qquad
 \widetilde X_2^{(4)}(d)=X_2^{(4)}(d)-3d(E_{14}+E_{54}).
$$
They reduce to the exceptional frame once $d^2=3d=0$.
Let $I\subseteq J_A$ be the invariant ideal generated by $d^2,3d$
and the entries of $X_j(\widetilde X_j^{(4)}(d))^{-1}-I_5$,
$j=1,2$. This ideal is finitely generated, since the involution
has order two. Equations~(31)--(37) in
\cite[Lemma~4.13]{BuninaMakadiya2ASmall} give
$$
 I/\operatorname{Jac}(A)I=\langle e,q\rangle_A,
 \qquad e=[d^2],\quad q=[3d],\quad 3e=0,\quad q=0.
$$
The form identity $\overline d=-d/(1-d)$ makes
$\varepsilon=d+I$ anti-invariant in $B=A/I$.
The candidate matrices over $B$ are unitary, have short Weyl
word $v_2$, and generate $M_5(B)$ together with the fixed Weyl
matrices. Their invariant transported form is therefore scalar.
Thus the total correcting class reduces to $[C_B]\in G^{(4)}(B)$,
and the map $\chi=i_{[C_B]^{-1}}\ell_*\varphi$ satisfies
Lemma~\ref{lem:exceptional-localized-image}.
For the resulting $\xi\in R^-_\theta$, normalize the $(4,4)$-entry
of its image to obtain
$$
        Y=I_5+d(E_{12}-E_{45})+D,\qquad D\in M_5(I).
$$
The additional root comparison, equation~(38) in the local
Lemma~4.13, gives $5e=0$.
Since $3e=0$, we have $e=2(3e)-5e=0$.
Thus $I/\operatorname{Jac}(A)I=0$, and Nakayama gives $I=0$.
The form and root identities are exact, and
Lemma~\ref{lem:exceptional-localized-image}, applied over $A$,
proves admissibility of $\delta_i=d$.

\smallskip
\noindent\emph{A split model.}
Work over $A=A_0\times A_0$ with the interchange involution.
The preliminary equations give the same invariant ideals $I$ and
exceptional frames over $B=A/I$, using the preceding polynomial
representatives in rank four. For $n=3$, use the product of
the two local extensions, interchanged by the involution;
partial traces and matrix-algebra descent are componentwise.
For $n=4$, the column system is invertible in both residue
factors and the selected defect subsystem has determinant
$-2^{20}$, so its calculation remains valid over
$I/\operatorname{Jac}(A)I$.

Over $B=A/I$, apply
Lemma~\ref{lem:exceptional-one-sided-coefficients}.
The residual coefficient map is onto $k\times k$.
Choose $r\in R$ with residual image $(1,0)$ and set
$\xi=r-\overline r$. Then $\mu(\xi)$ has residue $(1,-1)$,
so it is a unit. The equation $\varepsilon\mu(\xi)=0$ forces
$\varepsilon=0$.
For $n=3$, this means $\zeta\in\zeta^2A$; since $\zeta\in J_A$,
it follows that $\zeta=0$.
For $n=4$, it gives $d\in I\subseteq J_A$, so
$e=[d^2]=0$ in $I/J_AI$;
together with $q=0$ and the defect equations, Nakayama gives
$I=0$. Thus the split frame is ordinary.

In every case the correcting matrices reduce to the identity
after removing $g_0$. Restoring $g_0$ gives the prescribed
reduction of $g_i$ and the asserted residual formula.
\end{proof}

\subsection{The exceptional transporter and global normalization}
\label{subsec:uniqueness-exceptional-frame}
\label{subsec:exceptional-transporter-global-gluing}
\label{subsec:exceptional-parameter-admissibility}

A ratio of traces on fixed Weyl eigenspaces recovers the exceptional
parameter without a choice of projective representative. Once the
parameters agree, the square-zero deformation reduces uniqueness
of the conjugator to ordinary frame rigidity.
Both arguments work over every base algebra, as required for gluing.

\begin{proposition}[Uniqueness of the exceptional transporter]
\label{prop:uniqueness-exceptional-transporter}
Let $T$ be a commutative ring with $2\in T^\times$, and let
$d,e\in T$ satisfy $d^2=e^2=3d=3e=0$.
If $g,h\in G_{\mathrm{ad}}(A_n,T)$ satisfy
$$
        g u_\lambda^{(n)}(d)g^{-1}
        =h u_\lambda^{(n)}(e)h^{-1}
        \qquad(\lambda\in\Phi_\rho),
$$
then $d=e$ and $g=h$.
\end{proposition}

\begin{proof}
By affine local detection we may assume that $T$ is local, so
$h^{-1}g\in\operatorname{PGL}_{n+1}(T)$ has a matrix
representative $C$.

\smallskip
\noindent\emph{The parameter.}
The parameter-independent Weyl word is $w_2$ for
$n=3$ and $w_1$ for $n=4$, so $CwC^{-1}=aw$ for a scalar
unit $a$. Since $\operatorname{tr}w$ is respectively $2$ or $1$,
we have $a=1$. Thus $C$ preserves the two eigenspaces of
$$
 H_3=w_2^2=\operatorname{diag}(1,-1,-1,1),\qquad
 H_4=w_1^2=\operatorname{diag}(-1,-1,1,-1,-1).
$$
For $n=3$, let $\lambda_+(Y)$ be the scalar action on
$\langle e_1,e_4\rangle$ and $\operatorname{tr}_-(Y)$ the trace
on $\langle e_2,e_3\rangle$. The projectively well-defined
ratio $2\lambda_+(Y)/\operatorname{tr}_-(Y)$ is invariant
under $C$ and, on $X_2^{(3)}(z)$, equals
$$
        \frac{1+2z}{1+z}=1+z.
$$
For $n=4$, use the scalar action on $\langle e_3\rangle$ and
the trace on its complementary four-dimensional eigenspace.
The ratio $4\lambda_+(Y)/\operatorname{tr}_-(Y)$ on
$X_1^{(4)}(z)$ is
$$
        \frac{1-z}{1+z}=1-2z.
$$
All denominators are units. Applying these invariants at $d,e$
gives $d=e$.

\smallskip
\noindent\emph{The conjugating element.}
Now $[C]$ centralizes $\mathcal F_d^{(n)}$.
Modulo $I:=Td$, it centralizes the standard frame and is therefore
one by Proposition~\ref{prop:rigidity-unitary-frame}.
Since $I^2=0$, its scalar reduction lifts to a unit. Rescale to
write $C=I_{n+1}+Z$ with $Z\in M_{n+1}(I)$.
Write each frame matrix as
$U_\lambda+V_\lambda$, with $U_\lambda$ standard and
$V_\lambda\in M_{n+1}(I)$.
Projective centralization and $I^2=0$ give
$$
        ZU_\lambda-U_\lambda Z=s_\lambda U_\lambda
        \quad(s_\lambda\in I),
        \qquad
        CU_\lambda C^{-1}=(1+s_\lambda)U_\lambda.
$$
Thus $[C]$ centralizes the ordinary frame and is one by frame
rigidity. Local detection gives both conclusions over $T$.
\end{proof}

For $\varphi\in\operatorname{Aut}E^{(n)}(R)$ and an $R$-algebra
$T$, let $v_{\lambda,T}$ be the image of
$\varphi(x_\lambda(1))$ over $T$, and define
$$
\mathcal T_\varphi^{(n)}(T):=
\left\{(c,d)\in G_{\mathrm{ad}}(A_n,T)\times T\ \middle|
\begin{array}{l}
 d^2=3d=0,\\
 c u_\lambda^{(n)}(d)c^{-1}=v_{\lambda,T}
                    \quad(\lambda\in\Phi_\rho)
\end{array}\right\}.
$$
The invertible frame matrices over $R[d]/(d^2,3d)$ make this an
affine $R$-scheme of finite presentation. The preceding proposition
gives at most one point over every $R$-algebra. Only the indexed
frames are extended to $T$, not $\varphi$; no involution on $T$
is required. The full annihilator conditions are checked after
gluing, rather than added to the finite system.

\begin{theorem}[Global exceptional frame normalization]
\label{thm:global-exceptional-frame-normalization}
Let $\varphi\in\operatorname{Aut}E^{(n)}(R)$.
There is a unique pair
$$
        g\in G^{(n)}(R),\qquad
        \delta\in\mathcal E_{\mathrm{glob}}(R,\theta)
$$
such that
$$
        \varphi(x_\lambda(1))
        =g u_\lambda^{(n)}(\delta)g^{-1}
        \qquad(\lambda\in\Phi_\rho).
$$
Its image over $A_i$ is the pair $(g_i,\delta_i)$ of
Proposition~\ref{prop:relative-exceptional-normalization}.
For $\psi:=i_{g^{-1}}\circ\varphi$, one has
$$
        (q_{A_i})_*(\iota_i)_*\psi=(\nu_i^{\mathrm{res}})_*
        \qquad(i\in\mathcal I).
$$
\end{theorem}

\begin{proof}
The orbitwise normalization supplies a point of
$\mathcal T_\varphi^{(n)}(A_i)$ for every $i$.
Proposition~\ref{prop:unique-affine-gluing} gives its unique
global point $(g,\delta)$, whose localizations are the given
local pairs. Their twisting invariance and anti-invariance,
detected over $R$, give $\sigma(g)=g$ and
$\overline\delta=-\delta$.
For every $r\in R$, local admissibility gives
$$
        \iota_i\bigl(\delta(r-\overline r)\bigr)=0,
        \qquad \iota_i\bigl(\delta(r^3-r)\bigr)=0
        \qquad(i\in\mathcal I).
$$
These identities descend by local detection.
Together with $3\delta=0$ from the transporter equations,
they prove $\delta\in\mathcal E_{\mathrm{glob}}(R,\theta)$.
The residual formula follows from the prescribed reductions
of $g_i$. Uniqueness is the preceding proposition.
\end{proof}

\subsection{Removal of the exceptional factor}
\label{subsec:exceptional-standard-decomposition}

Removing the global exceptional automorphism leaves the ordinary
frame fixed. We can then apply
Theorem~\ref{thm:coefficient-recovery-fixed-frame}.

\begin{theorem}[Automorphisms in the exceptional unitary types]
\label{thm:exceptional-unitary-automorphisms}
Every automorphism $\varphi\in\operatorname{Aut}E^{(n)}(R)$
has a unique decomposition
$$
        \varphi=i_g\circ\Theta_\delta^{(n)}\circ\mu_*,
        \qquad
        g\in G^{(n)}(R),\quad
        \delta\in\mathcal E_{\mathrm{glob}}(R,\theta),\quad
        \mu\in\operatorname{Aut}(R,\theta).
$$
Conversely, every such composition is an automorphism.
\end{theorem}

\begin{proof}
Take $(g,\delta)$ from
Theorem~\ref{thm:global-exceptional-frame-normalization} and put
$$
        \psi=i_{g^{-1}}\circ\varphi,\qquad
        \psi_0=\Theta_{-\delta}^{(n)}\circ\psi.
$$
The root rule for $\Theta_\delta^{(n)}$ makes $\psi_0$ fix the
ordinary frame. Its residual isomorphisms are equivariant coefficient
isomorphisms by the preceding theorem and residual triviality of
$\Theta_{-\delta}^{(n)}$. The inverse $\psi_0^{-1}$ also fixes
the frame and has the inverse residual coefficient maps, as in
Corollary~\ref{cor:local-reduction-after-global-normalization}.

Apply Lemma~\ref{lem:exceptional-one-sided-coefficients} with
$\varepsilon=0$ to every localization of $\psi_0$ and
$\psi_0^{-1}$. It puts their root images in the corresponding
relative root subgroups.
Proposition~\ref{prop:preservation-relative-root-subgroups}
therefore gives $\psi_0(X_\lambda(R))=X_\lambda(R)$.
Theorem~\ref{thm:coefficient-recovery-fixed-frame} supplies a
unique $\mu\in\operatorname{Aut}(R,\theta)$ with $\psi_0=\mu_*$,
proving the asserted decomposition. The converse follows from
the constructions and normality of $E^{(n)}(R)$ in $G^{(n)}(R)$.

If $i_g\Theta_\delta^{(n)}\mu_*=i_h\Theta_\varepsilon^{(n)}\nu_*$,
coefficient automorphisms fix the frame, so
Proposition~\ref{prop:uniqueness-exceptional-transporter} gives
$g=h$ and $\delta=\varepsilon$.
Cancelling these factors and comparing on $X_1(R)$ gives $\mu=\nu$.
\end{proof}

\begin{corollary}
\label{cor:exceptional-unitary-standardness}
Every automorphism of $E^{(n)}(R)$ is standard if and only if
$\mathcal E_{\mathrm{glob}}(R,\theta)=0$.
\end{corollary}

\begin{proof}
If the module is zero, the theorem gives only standard factors.
Conversely, a standard expression for $\Theta_\delta^{(n)}$
would give a second decomposition with exceptional parameter zero.
Uniqueness forces $\delta=0$, so every nonzero parameter gives a
nonstandard automorphism.
\end{proof}

\medskip


\section[Graph factors in type 2D4]{Graph factors in type \texorpdfstring{${}^2D_4$}{2D4}}
\label{sec:absolute-type-D4}
\label{sec:graph-automorphisms-2D4}

We remove the residual diagram factors left in type ${}^2D_4$.
Choosing a cyclic subgroup of order three as representatives
makes the reduced graph factors a group. We separate their frames,
glue the labels and conjugating elements, and recover the coefficient
automorphism by Section~\ref{sec:recovery-of-coefficient-automorphism}.

Throughout, $\rho$ is the graph involution of $D_4$,
$\theta^2=\operatorname{id}_R$, and $2\in R^\times$. Write
$$
        E(R)=E_{\rho\theta}(D_4,R),\qquad
        G(R)=G_{\rho\theta}(D_4,R).
$$
As in the introduction, choose a three-cycle $\tau$ in
$\mathcal D=\operatorname{Aut}(\operatorname{Dyn}(D_4))$, so that
$$
        \mathcal D=\langle\tau\rangle\rtimes\langle\rho\rangle,
        \qquad \rho\tau\rho^{-1}=\tau^{-1},\qquad
        \mathcal R_\rho=\{1,\tau,\tau^2\}.
$$
Every diagram automorphism is uniquely $\tau^j\rho^a$, with
$j\in\mathbb Z/3\mathbb Z$ and $a\in\{0,1\}$.
All indices on reduced graph data below are taken modulo $3$.

\subsection{Componentwise graph automorphisms}
\label{subsec:compatible-graph-automorphisms}

The pinned maps $\gamma_*$, $\gamma\in\mathcal D$, commute with
base change, preserve the untwisted elementary subgroup, and satisfy
$\gamma_*\eta_*=(\gamma\eta)_*$; see
\cite[Theorem~6.1.17]{ConradReductive}.
We retain their induced signs on non-simple root coordinates.

\begin{definition}
\label{def:idempotent-graph-datum}
\label{def:compatible-graph-datum}
A \emph{reduced graph datum} is a triple
$\mathbf e=(e_0,e_1,e_2)$ of pairwise orthogonal idempotents with
$e_0+e_1+e_2=1$. It is \emph{$\rho$-compatible} if
\begin{equation}
\label{eq:compatible-graph-idempotents}
        \theta(e_j)=e_{-j}\qquad(j\in\mathbb Z/3\mathbb Z).
\end{equation}
Write $\mathcal C_\rho(R,\theta)$ for the set of compatible data.
We also write $e_{\tau^j}=e_j$ and $e_\gamma=0$ for
$\gamma\notin\mathcal R_\rho$.
\end{definition}

Apply $(\tau^j)_*$ on $e_jR$. This defines an ambient split-group
automorphism $\Gamma_{\mathbf e}$, with
$$
        \Gamma_{\mathbf e}(x_\alpha(t))
          =\prod_{j=0}^2(\tau^j)_*(x_\alpha(e_jt)).
$$
The factors have orthogonal supports. For every unital $f:R\to S$,
componentwise naturality gives
$f_*\Gamma_{\mathbf e}=\Gamma_{f(\mathbf e)}f_*$.

\begin{lemma}
\label{lem:composition-idempotent-graph-data}
For reduced data, put
$$
        (\mathbf e*\mathbf f)_k=\sum_{i+j=k}e_if_j,
        \qquad e^\vee_j=e_{-j}.
$$
Then $\Gamma_{\mathbf e}\Gamma_{\mathbf f}
=\Gamma_{\mathbf e*\mathbf f}$ and
$\Gamma_{\mathbf e}^{-1}=\Gamma_{\mathbf e^\vee}$.
These operations make $\mathcal C_\rho(R,\theta)$ an abelian group
of exponent dividing three, with identity $(1,0,0)$.
\end{lemma}

\begin{proof}
On $e_if_jR$ the composition is $(\tau^{i+j})_*$.
The products $e_if_j$ form an orthogonal partition, and compatibility
is preserved by convolution and inversion. On $e_jR$, the third
power of the datum has label $3j=0$.
\end{proof}

\begin{proposition}
\label{prop:compatible-graph-automorphisms}
For compatible $\mathbf e$, the action $\Gamma_{\mathbf e}$
commutes with $\sigma_R$ and restricts to automorphisms of
$G(R)$ and $E(R)$, naturally under equivariant base change.
\end{proposition}

\begin{proof}
Conjugation by $\sigma_R=\rho_*\theta_*$ changes the idempotent
at label $j$ to $\theta(e_{-j})=e_j$, so preserves the action.
For elementary preservation, decompose by the invariant idempotents
$e_0$ and $p=e_1+e_2$. The action on $e_0R$ is the identity.
On $pR$, identify the exchanged factors with $B\times B$,
$B=e_1R$. By Lemma~\ref{lem:split-product-identification},
first projection identifies the twisted elementary group with
$E_{\mathrm{ad}}(D_4,B)$, and the action becomes $\tau_*$:
$$
        (x,\rho_*(x))\longmapsto
                  (\tau_*(x),\rho_*\tau_*(x)).
$$
It therefore preserves the elementary subgroup, as does its inverse.
\end{proof}

For any reduced datum over $T$, without an operator on $T$, put
$\mathcal F_{\Phi,\mathbf e}(T)
=(\Gamma_{\mathbf e}(x_\lambda(1)))_{\lambda\in\Phi_\rho}$.

\begin{corollary}
\label{cor:rigidity-graph-transformed-frame}
For every ring $T$ with $2\in T^\times$ and every reduced datum,
$$
        C_{G_{\mathrm{ad}}(D_4,T)}
                  (\mathcal F_{\Phi,\mathbf e}(T))=1.
$$
\end{corollary}

\begin{proof}
Apply the ambient automorphism $\Gamma_{\mathbf e}$ to the
centralizer in Proposition~\ref{prop:rigidity-type-2D}.
\end{proof}

\begin{lemma}
\label{lem:powers-rho-absorbed-by-ring-automorphisms}
For a $\theta$-fixed idempotent $e$, the map
$$
        \mu_e(r)=(1-e)r+e\theta(r)
$$
is an involutive member of $\operatorname{Aut}(R,\theta)$.
Its action on $G(R)$ and $E(R)$ is $\rho_*$ on $eR$ and the
identity on $(1-e)R$.
\end{lemma}

\begin{proof}
Use the invariant product decomposition and
$\rho_* =\theta_*$ on the twisted group.
\end{proof}

Thus powers of $\rho$ contribute no separate graph factor.
For a split pair $A=B\times B$ with a designated first factor,
we will use
$$
        \Gamma_{j,A}:=\Delta_\rho\circ(\tau^j)_*
                                \circ\operatorname{pr}_0,
        \qquad \Delta_\rho(x)=(x,\rho_*(x)).
$$
This is the compatible action with labels $j,-j$ on the two factors;
its inverse is $\Gamma_{-j,A}$. Switching the designated factor
changes $j$ to $-j$. The action corresponding to $\tau^j\rho^a$
is $\Gamma_{j,A}\theta_{A,*}^a$, so its last factor is absorbed
into the coefficient map. These identities commute with reduction.

\subsection{Residual graph labels}
\label{subsec:residual-graph-standardness-D4}

We apply the field isomorphism theorems of
Section~\ref{sec:residual-standardness}, retaining only the exponent
of $\tau$ at each target field factor.

\begin{theorem}[Residual graph-standardness]
\label{thm:residual-graph-standardness-D4}
Let $\alpha:E(K_1)\xrightarrow{\sim}E(K_2)$ be an isomorphism
for residue algebras of ${}^2D_4$-local models.
The algebras are either both fixed or both split.
For fixed fields, the involutions have the same order and
$\alpha=i_g\circ\nu_*$, with $g\in G(K_2)$ and
$\nu:K_1\xrightarrow{\sim}K_2$ equivariant.
For split algebras, designate a target factor. Then
$$
        \alpha=i_g\circ\Gamma_{j,K_2}\circ\nu_*,
        \qquad g\in G(K_2),\quad j\in\mathbb Z/3\mathbb Z,
$$
where $\nu:K_1\xrightarrow{\sim}K_2$ is equivariant.
The label $j$ is independent of the decomposition and any source
factor choice. Switching the target factor replaces it by $-j$.
\end{theorem}

\begin{proof}
The three field models are split $B_3$, genuine ${}^2D_4$, and
split $D_4$; the field theorems distinguish them.
For fixed algebras, the proof of
Proposition~\ref{prop:fixed-residue-standardness} applies:
$B_3$ has no diagram factor, while in the genuinely twisted case
the coefficient extension is recovered and the compatible diagram
factor lies in $C_{\mathcal D}(\rho)=\langle\rho\rangle$.
It is absorbed into $\nu$.

For split algebras $K_i=k_i^2$, first projection gives an untwisted
isomorphism $i_h\circ\gamma_*\circ(\nu_0)_*$, where
$h\in G_{\mathrm{ad}}(D_4,k_2)$ and
$\nu_0:k_1\xrightarrow{\sim}k_2$.
Write $\gamma=\tau^j\rho^a$, take $g=\Delta_\rho(h)$, and put
$\nu=\theta_2^a\circ(\nu_0\times\nu_0)$.
The split-pair identities give the required expression.
Uniqueness of the diagram factor in the field theorem determines
$j$: a source switch multiplies $\gamma$ on the right by $\rho$,
leaving $j$ unchanged, while a target switch multiplies it on the
left and replaces $j$ by $-j$.
\end{proof}

\begin{corollary}[Residual graph labels]
\label{cor:residual-graph-class-D4}
At each maximal ideal $\mathfrak n$ of the target residue algebra,
$\alpha$ has a unique label
$\operatorname{gr}_{\mathfrak n}(\alpha)\in\mathbb Z/3\mathbb Z$,
zero on a fixed field, satisfying
\begin{equation}
\label{eq:residual-graph-label-equivariance}
        \operatorname{gr}_{\theta_2(\mathfrak n)}(\alpha)
                   =-\operatorname{gr}_{\mathfrak n}(\alpha).
\end{equation}
All labels vanish exactly when $\alpha=i_g\circ\nu_*$
with $\nu$ equivariant.
\end{corollary}

\begin{proof}
Use the theorem and $\Gamma_{0,K_2}=\operatorname{id}$.
\end{proof}

For $\varphi\in\operatorname{Aut}E(R)$, use its residual
isomorphisms from Subsection~\ref{subsec:congruence-spectrum}
and put
$$
        c_{\mathfrak m}
          :=\operatorname{gr}_{\mathfrak m/I_i}
                          (\varphi_i^{\mathrm{res}})
                    \qquad(\mathfrak m\in\mathcal O_i).
$$
Then $\varphi^\sharp$ preserves orbit sizes and
$c_{\theta(\mathfrak m)}=-c_{\mathfrak m}$.
We next prove that these pointwise labels come from idempotents;
local constancy is a conclusion, not an assumption.

\subsection{Separation and gluing of graph frames}
\label{subsec:global-graph-reduction}

The three frames are
$$
        \mathcal F_{\Phi,j}(T)
          :=\bigl((\tau^j)_*(x_\lambda(1))\bigr)_{\lambda\in\Phi_\rho}.
$$
To separate them, we lift to $\operatorname{Spin}_8$ over an
algebraically closed residue field. A fixed Weyl word in the lifted
frame squares to one of three distinct central involutions; inner
conjugation cannot interchange them.

\begin{lemma}[Separation of the reduced graph frames]
\label{lem:separation-D4-graph-frames}
Let $T\ne0$ be a commutative ring with $2\in T^\times$.
If $g,h\in G_{\mathrm{ad}}(D_4,T)$ and
$j,j'\in\mathbb Z/3\mathbb Z$ satisfy
$$
        g(\tau^j)_*(x_\lambda(1))g^{-1}
          =h(\tau^{j'})_*(x_\lambda(1))h^{-1}
                  \qquad(\lambda\in\Phi_\rho),
$$
then $j=j'$ and $g=h$.
\end{lemma}

\begin{proof}
Reduce at a maximal ideal and extend to an algebraic closure $k$.
Lift the image of $h^{-1}g$ to $\widetilde c\in\operatorname{Spin}_8(k)$.
Let $\widetilde u_\lambda$ be the standard relative root element
of parameter $1$ in the simply connected group.
The central isogeny has kernel $\boldsymbol\mu_2^2$.
Two unipotent lifts with the same adjoint image therefore differ
by a central semisimple element, which must be one by uniqueness
of Jordan decomposition. Hence the conjugacy identities lift exactly:
$$
        \widetilde c(\tau^j)_*(\widetilde u_\lambda)\widetilde c^{-1}
                    =(\tau^{j'})_*(\widetilde u_\lambda).
$$

Number the outer nodes so that $\rho$ exchanges $\alpha_3,\alpha_4$,
and put $\beta=[\alpha_3]$.
Their absolute rank-one subgroups commute, so the word
$\widetilde w=\widetilde u_\beta
\widetilde u_{-\beta}^{-1}\widetilde u_\beta$ satisfies
\begin{equation}
\label{eq:D4-frame-central-word}
        \widetilde w=w_{\alpha_3}(1)w_{\alpha_4}(1),\qquad
        \widetilde w^{\,2}
          =h_{\alpha_3}(-1)h_{\alpha_4}(-1)=:z_{34}.
\end{equation}
The three elements
$z_{rs}=h_{\alpha_r}(-1)h_{\alpha_s}(-1)$,
$\{r,s\}\subset\{1,3,4\}$, $r\ne s$, are central by the Cartan
pairings and distinct by their fundamental-weight characters.
The action of $\tau$ cycles them. Applying the lifted transport
identities to $\widetilde w^{\,2}$ gives
$$
        (\tau^j)_*(z_{34})
          =\widetilde c(\tau^j)_*(z_{34})\widetilde c^{-1}
          =(\tau^{j'})_*(z_{34}),
$$
so $j=j'$. Back over $T$, the quotient $h^{-1}g$ centralizes
this transformed frame and is one by
Corollary~\ref{cor:rigidity-graph-transformed-frame}.
\end{proof}

\smallskip
\noindent\emph{The affine normalization scheme.}
Put $u_\lambda=x_\lambda(1)$ and $v_\lambda=\varphi(u_\lambda)$.
For an $R$-algebra $T$, define
$$
 \mathcal T_{\varphi,j}(T)
   :=\{g\in G_{\mathrm{ad}}(D_4,T)\mid
      g(\tau^j)_*(u_{\lambda,T})g^{-1}=v_{\lambda,T}
                         \ (\lambda\in\Phi_\rho)\},
 \qquad
 \mathcal X_\varphi:=\coprod_{j\in\mathbb Z/3\mathbb Z}
                                      \mathcal T_{\varphi,j}.
$$
This finite disjoint union of closed affine transporters is affine
of finite presentation. A $T$-point is a pair $(g,\mathbf e)$,
where $\mathbf e$ is a reduced datum and $g$ transports
$\mathcal F_{\Phi,\mathbf e}(T)$ to $(v_{\lambda,T})_\lambda$.
Only the frames are extended to $T$, not $\varphi$.
No operator or compatibility condition on $T$ is imposed.

\begin{corollary}
\label{cor:uniqueness-D4-graph-transporter}
For every $R$-algebra $T$, the set $\mathcal X_\varphi(T)$ has
at most one point.
\end{corollary}

\begin{proof}
For two points with data $\mathbf e,\mathbf f$, separation on each
nonzero factor $e_jf_rT$ forces $j=r$.
Summing gives $e_j=e_jf_j=f_j$; frame rigidity then identifies
the transporters on all factors.
\end{proof}

\begin{proposition}
\label{prop:local-D4-graph-transporter}
For every $i\in\mathcal I$, the set $\mathcal X_\varphi(A_i)$
is nonempty.
\end{proposition}

\begin{proof}
Put $A=A_i$, $\Psi=(\iota_i)_*\varphi$, and
$i'=(\varphi^\sharp)^{-1}(i)$, and identify $A_{\mathrm{res}}=K_i$.
The residual theorem gives
$$
        (q_A)_*\Psi=i_{h_i}\circ\Gamma_i^{\mathrm{res}}
                                    \circ(\nu_iq_{i'})_*,
$$
where $\nu_i:K_{i'}\xrightarrow{\sim}K_i$ is equivariant and
$\Gamma_i$ is the identity for a fixed model, or $\Gamma_{j,A}$
for a split model. The latter is an actual target automorphism
constructed above, without a graph datum over $R$.

The normalized map $\widetilde\Psi=\Gamma_i^{-1}\Psi$ has ordinary
residual data, with conjugator $(\Gamma_i^{\mathrm{res}})^{-1}(h_i)$.
By Lemma~\ref{lem:relative-local-rigidity-2D4}, it has a transporter
$\widetilde g_i\in G(A)$. Put $g_i=\Gamma_i(\widetilde g_i)$; then
$$
        (q_A)_*(g_i)=h_i,\qquad
        \Psi(u_\lambda)=g_i\Gamma_i(u_\lambda)g_i^{-1}.
$$
The associated local datum has label $0$ on a fixed model and
labels $j,-j$ on the two factors of a split model. Together with
$g_i$, it gives a point of $\mathcal X_\varphi(A)$.
\end{proof}

\begin{theorem}[Global graph reduction]
\label{thm:global-graph-reduction-2D4}
Every $\varphi\in\operatorname{Aut}E(R)$ determines a unique
$g\in G(R)$ and $\mathbf e\in\mathcal C_\rho(R,\theta)$ such that
$$
        \varphi(x_\lambda(1))
           =g\Gamma_{\mathbf e}(x_\lambda(1))g^{-1}
                         \qquad(\lambda\in\Phi_\rho).
$$
The idempotents realize the residual labels:
$$
        e_j\bmod\mathfrak m
           =\begin{cases}1,&j=c_{\mathfrak m},\\0,&j\ne c_{\mathfrak m}.
             \end{cases}
$$
The automorphism $\psi=\Gamma_{\mathbf e}^{-1}\circ i_{g^{-1}}
\circ\varphi$ fixes the frame, satisfies
$\psi^\sharp=\varphi^\sharp$, and has residual isomorphisms
$$
        \psi_i^{\mathrm{res}}=(\nu_i)_*,\qquad
        \nu_i:K_{i'}\xrightarrow{\ \sim\ }K_i\text{ equivariant},
        \qquad i'=(\varphi^\sharp)^{-1}(i).
$$
The residual maps of $\psi^{-1}$ are the inverse coefficient
isomorphisms.
\end{theorem}

\begin{proof}
The preceding two statements and
Proposition~\ref{prop:unique-affine-gluing} give the unique point
$(g,\mathbf e)\in\mathcal X_\varphi(R)$.
Twisting its identities gives the point
$(\sigma_R(g),(\theta(e_{-j}))_j)$ of the same scheme.
Uniqueness therefore gives $g\in G(R)$ and
$\theta(e_j)=e_{-j}$. Its localizations are the constructed local
points, so the idempotents have the asserted residues.

Both $i_g$ and $\Gamma_{\mathbf e}$ preserve each $N_i$ by
naturality of reduction. Thus $\psi^\sharp=\varphi^\sharp$.
Over $A_i$, the graph action is $\Gamma_i$ and $g$ reduces to
$h_i$. Cancelling them in the residual identity gives
$$
        (q_{A_i})_*\circ(\iota_i)_*\circ\psi=(\nu_iq_{i'})_*.
$$
This proves the coefficient formula; inversion gives the assertion
for $\psi^{-1}$. Uniqueness follows from that of the global point.
\end{proof}

\subsection{The automorphism theorem}
\label{subsec:automorphism-theorem-D4}

After graph reduction, the ordinary recognition and coefficient
recovery results apply to the normalized automorphism and its inverse.
The cyclic choice of representatives also gives a semidirect product.

\begin{theorem}[Automorphisms of type ${}^2D_4$]
\label{thm:automorphisms-type-2D4}
Every automorphism of $E(R)$ has a unique decomposition
$$
        \varphi=i_g\circ\Gamma_{\mathbf e}\circ\mu_*,\qquad
        g\in G(R),\quad\mathbf e\in\mathcal C_\rho(R,\theta),\quad
        \mu\in\operatorname{Aut}(R,\theta).
$$
Conversely, every such composition is an automorphism.
\end{theorem}

\begin{proof}
Remove $g,\mathbf e$ by
Theorem~\ref{thm:global-graph-reduction-2D4}.
The resulting $\psi$ and $\psi^{-1}$ fix the frame and have
coefficient reductions. Apply
Lemma~\ref{lem:relative-local-rigidity-2D4} to their localizations:
the transporter is the identity, and root recognition applies.
Proposition~\ref{prop:preservation-relative-root-subgroups} gives
global root preservation, and
Theorem~\ref{thm:coefficient-recovery-fixed-frame} gives $\psi=\mu_*$.
Every displayed factor is an automorphism. In two decompositions,
global graph reduction identifies $g,\mathbf e$; cancellation and
one short-root parameterization then identify $\mu$.
\end{proof}

\begin{corollary}
\label{cor:automorphism-group-2D4}
There is an isomorphism
$$
        \operatorname{Aut}E(R)\simeq
        G(R)\rtimes\bigl(\mathcal C_\rho(R,\theta)
                           \rtimes\operatorname{Aut}(R,\theta)\bigr),
$$
where $\mu\cdot\mathbf e=(\mu(e_j))_j$ and
$(\mathbf e,\mu)\cdot g=\Gamma_{\mathbf e}(\mu_*(g))$.
\end{corollary}

\begin{proof}
Naturality gives
$\mu_*\Gamma_{\mathbf e}\mu_*^{-1}=\Gamma_{\mu(\mathbf e)}$,
and both actions are defined on $G(R)$. Hence composition of the
unique normal forms is
$$
 (g,\mathbf e,\mu)(h,\mathbf f,\nu)
   =\bigl(g\Gamma_{\mathbf e}(\mu_*(h)),\,
                 \mathbf e*\mu(\mathbf f),\,\mu\circ\nu\bigr),
$$
which is the asserted semidirect-product law.
\end{proof}

A nonneutral graph factor requires nonzero orthogonal idempotents
exchanged by $\theta$. If $\theta$ fixes every idempotent, then
$e_1=e_2=0$ and every automorphism is standard. In particular, this
holds when $R$ has no nontrivial idempotents.
\medskip


\section[The triality type 3D4]{The triality type \texorpdfstring{${}^3D_4$}{3D4}}
\label{sec:triality}

We classify the adjoint elementary automorphisms in triality.
The root relations recover the ring together with its trace,
not necessarily with the chosen coefficient operator. Moreover,
normal subgroups need not have invariant ideal levels. We prove
the maximal-normal-subgroup statement needed for the global argument
using a regular function extracted from two long-root subgroups.

Throughout, $\rho$ has order three, $\theta^3=\operatorname{id}_R$,
and $2,3\in R^\times$. For any such coefficient pair $S$, write
$$
        E(S)=E_{\rho\theta_S}(D_4,S),\qquad
        G(S)=G_{\rho\theta_S}(D_4,S).
$$
These are the adjoint elementary and full twisted groups;
$G_{\mathrm{ad}}(D_4,S)$ denotes the ambient split group.
We use the arbitrary-target coefficient recovery and fixed-target
relative normalization of \cite{BuninaMakadiyaTriality}.
The split-target step is now supplied by
Lemma~\ref{lem:fixed-to-split-frame-normalization}.
Construction of the converse coefficient actions remains separate
from recovery for an existing homomorphism.

\subsection{The trace and coefficient recovery}
\label{subsec:triality-coefficient-automorphisms}

Use the Weyl-coherent short-root coordinates of
\cite[Subsections~2.1--2.3]{BuninaMakadiyaTriality}.
They differ from those of Subsection~\ref{subsec:relative-root-subgroups}
only by $t\mapsto\theta^j(t)$ and leave the frame unchanged;
Weyl transport between roots of the same length changes parameters
only by signs. Label the positive relative roots
$$
        b=\beta_s,\quad a=\beta_\ell,\quad c=b+a,\quad d=2b+a,
        \quad e=3b+a,\quad f=3b+2a.
$$
The roots $b,c,d$ are short, with parameters in $S$; $a,e,f$ are
long, with parameters in $S^{\theta_S}$.
All signs refer to these coordinates, with
$[x,y]=xyx^{-1}y^{-1}$. Put
$$
\begin{aligned}
 \operatorname{Tr}_\theta(t)&=t+\theta(t)+\theta^2(t),&
 t^\#&=\theta(t)\theta^2(t),\\
 N_\theta(t)&=tt^\#,&
 t\times u&=(t+u)^\#-t^\#-u^\#.
\end{aligned}
$$
The trace divided by $3$ is the projector onto $S^{\theta_S}$,
so $S=S^{\theta_S}\oplus\ker\operatorname{Tr}_{\theta_S}$ additively.
Preserving this projector need not preserve the operator on its
kernel, as the square-zero example below will show.

\begin{definition}
\label{def:triality-compatible-homomorphism}
A unital ring homomorphism $\nu:S\to T$ is \emph{trace-preserving},
or \emph{triality-compatible}, if
$$
        \nu\operatorname{Tr}_{\theta_S}
        =\operatorname{Tr}_{\theta_T}\nu.
$$
Write
$$
 \operatorname{Aut}_{\mathrm{tr}}(R,\theta)
   =\{\mu\in\operatorname{Aut}(R)\mid
         \mu\operatorname{Tr}_\theta=\operatorname{Tr}_\theta\mu\}.
$$
\end{definition}

\begin{lemma}
\label{lem:triality-compatible-basic-properties}
A trace-preserving $\nu:S\to T$ maps $S^{\theta_S}$ into
$T^{\theta_T}$ and preserves $t^\#$, $N_\theta(t)$, and $t\times u$.
If $\nu$ is an isomorphism, it maps the fixed subrings onto one another.
\end{lemma}

\begin{proof}
Compatibility with the projectors gives the first assertion;
apply it also to $\nu^{-1}$ when this exists. The identity
$$
 2t^\#=(\operatorname{Tr}_\theta(t)-t)^2
                  -(\operatorname{Tr}_\theta(t^2)-t^2)
$$
recovers $t^\#$ from the ring and trace. Multiplication by $t$
and polarization recover the norm and $\times$.
\end{proof}

For such a $\nu$, the notation $\Phi_\nu$ initially denotes only
$$
        x_\lambda(t)\longmapsto x_\lambda(\nu(t)).
$$
Long parameters are admissible by the lemma. This is not generally
an ambient coefficientwise action, so preservation of all group
relations requires proof. In the next proposition the group
homomorphism is already given.

\begin{proposition}[Coefficient recovery from the triality frame]
\label{prop:global-triality-coefficient-recovery}
\label{prop:oriented-coefficient-recovery-triality}
Let $S,T$ be coefficient pairs and let $\psi:E(S)\to E(T)$ be a
homomorphism carrying every constant frame element to the corresponding one.
Then there is a unique trace-preserving unital ring homomorphism
$\mu:S\to T$ such that
$$
        \psi(x_\lambda(t))=x_\lambda(\mu(t))
        \qquad(\lambda\in\Phi_\rho,\ t\in S_\lambda).
$$
If $\psi$ is an isomorphism, then so is $\mu$.
In particular, every frame-fixing automorphism of $E(R)$ is
$\Phi_\mu$ for a unique $\mu\in\operatorname{Aut}_{\mathrm{tr}}(R,\theta)$.
\end{proposition}

\begin{proof}
This is \cite[Theorem~4.2]{BuninaMakadiyaTriality}.
We retain the short coefficient argument explaining the role of
the trace. Its Lemma~4.1 recognizes the root images over arbitrary
targets. Additivity and the fixed Weyl words give additive maps
$\mu:S\to T$ on short roots and
$\mu_0:S^{\theta_S}\to T^{\theta_T}$ on long roots, both fixing $1$.
Comparing the $X_c$-coordinate of $[x_a(v),x_b(t)]$ and the
$X_d$-coordinate of $[x_c(t),x_b(u)]$ gives
$$
\begin{aligned}
        \mu(vt)&=\mu_0(v)\mu(t)
                     &&(v\in S^{\theta_S},\ t\in S),\\
        \mu(t\times u)&=\mu(t)\times\mu(u)
                     &&(t,u\in S).
\end{aligned}
$$
Taking $t=1$ in the first identity gives
$\mu_0=\mu|_{S^{\theta_S}}$.
Taking $u=1$ in the second and using
$t\times1=\operatorname{Tr}_{\theta_S}(t)-t$ gives trace preservation.
Arbitrary multiplication is then recovered by
$$
 2tu=t\times u+\operatorname{Tr}_{\theta_S}(t)u
       +\operatorname{Tr}_{\theta_S}(u)t
       -\operatorname{Tr}_{\theta_S}(t\times u).
$$
The traced factors are fixed, so applying $\mu$ gives
$2\mu(tu)=2\mu(t)\mu(u)$. Cancelling $2$ proves multiplicativity.
Root generation gives the asserted action, and one short-root
parameterization gives uniqueness. Applying recovery to $\psi^{-1}$
when it exists produces the inverse ring map. No locality or residual
hypothesis is needed.
\end{proof}

\smallskip
\noindent\emph{Orbit localization.}
Trace preservation must also recover the maximal-ideal orbits,
without using an induced group action. Ideals are transported by
direct image, $\mu(\mathfrak a)=\{\mu(a):a\in\mathfrak a\}$.

\begin{lemma}[Triality automorphisms and orbit neighborhoods]
\label{lem:triality-orbit-localization}
Every $\mu\in\operatorname{Aut}_{\mathrm{tr}}(R,\theta)$ induces
a permutation $\mu^\sharp$ of the orbit index set and trace-preserving
ring isomorphisms
$$
\begin{gathered}
 \mu(\mathcal O_i)=\mathcal O_{\mu^\sharp(i)},\qquad
 \mu_i^B:B_i\xrightarrow{\sim}B_{\mu^\sharp(i)},\qquad
 \mu_i^A:A_i\xrightarrow{\sim}A_{\mu^\sharp(i)},\\
 \mu_i^A\iota_i=\iota_{\mu^\sharp(i)}\mu.
\end{gathered}
$$
\end{lemma}

\begin{proof}
Put $\kappa=\mu\theta\mu^{-1}$. Then
$1+\kappa+\kappa^2=1+\theta+\theta^2$ as additive operators.

\smallskip
\noindent\emph{Recovering the orbits.}
The $\kappa$- and $\theta$-orbits of a maximal ideal
$\mathfrak m$ coincide. Otherwise the Chinese remainder theorem
gives an element with residue $1$ at an ideal belonging to just
one orbit and $0$ at all other ideals in their union.
The chosen ideal differs from $\mathfrak m$ and occurs exactly
once in its three-element orbit, so evaluation of the trace identity
modulo $\mathfrak m$ gives $1=0$.
Since $\kappa^j(\mu(\mathfrak m))=\mu(\theta^j(\mathfrak m))$,
this proves the orbit permutation. In particular,
$\mu(\Sigma_i)=\Sigma_{\mu^\sharp(i)}$, giving the asserted
localization maps and their compatibility with $\iota_i$.

\smallskip
\noindent\emph{Preserving the localized traces.}
In an invariant localization,
$$
        \operatorname{Tr}_\theta(r/s)
          =\frac{\operatorname{Tr}_\theta(rs^\#)}{N_\theta(s)}.
$$
Preservation of $\#$ and $N$ proves trace compatibility on $B_i$
and on $A_i=B_i$ for fixed orbits.
For a three-element target orbit, choose $\epsilon\in\{1,-1\}$
so that $\kappa$ and $\theta^\epsilon$ induce the same permutation.
To compare their coefficient actions, put
$$
 D(r)=\kappa(r)-\theta^\epsilon(r),\qquad
 F(r)=\kappa(r)-\theta^{-\epsilon}(r).
$$
The trace identities at $r$ and $r^2$ give
$\kappa(r)\kappa^2(r)=\theta(r)\theta^2(r)$, hence
$$
        D(r)F(r)=0,\qquad D(r)F(t)+D(t)F(r)=0.
$$
Fix $\mathfrak m$ in this orbit and choose $t$ with residues
$0$ at $\theta^{-\epsilon}(\mathfrak m)$ and $1$ at
$\theta^\epsilon(\mathfrak m)$.
Then $F(t)$ is a unit in $R_{\mathfrak m}$; the first identity
gives $D(t)=0$ there, and the second gives $D(r)=0$ for every $r$.
Doing this at all three ideals and localizing proves
$$
        \mu_i^A\theta_{A_i}(\mu_i^A)^{-1}
             =\theta_{A_{\mu^\sharp(i)}}^\epsilon.
$$
This gives trace compatibility on the split model as well.
\end{proof}

\subsection{Maximal normal subgroups}
\label{subsec:triality-maximal-normal-subgroups}
\label{subsubsec:triality-maximal-normal-subgroups}

We identify the maximal congruence quotients required in
Subsection~\ref{subsec:congruence-spectrum}.
The argument does not assign an invariant ideal to every normal
subgroup; such a classification fails even for a local ring.

\smallskip
\noindent\emph{A square-zero example.}
Let $R=\mathbb Q\oplus V$, $V=\mathbb Q^2$, with
$$
 (q,v)(q',v')=(qq',qv'+q'v),\qquad
 \theta(q,v)=(q,Av),\qquad
 A=\begin{pmatrix}0&-1\\1&-1\end{pmatrix}.
$$
By \cite[Example~7.2 and Subsection~7.1]{BuninaMakadiyaTriality},
$$
        E(R)=G(R)\simeq
        (W_7\otimes_{\mathbb Q}V)\rtimes G_2(\mathbb Q),
$$
where $W_7$ is the seven-dimensional irreducible module and the
constant group acts only on that factor. In these multiplicity
coordinates, radical short-root maps are signed copies of
$v\mapsto w_\lambda\otimes v$.
For $L=\mathbb Q(1,0)$, the subgroup $H=W_7\otimes L$ is normal,
with short-root parameter groups $L$, although $\theta(L)\ne L$.
The only invariant ideals are $0,V,R$, because $X^2+X+1$ is
irreducible over $\mathbb Q$. None gives
$E(R,I)\leq H\leq N_I$: for $I=0$ the upper group is trivial;
for $I=V$ the lower group contains all short parameters in $V$;
for $I=R$ it is $E(R)$.
Thus the triality assertion of
\cite[Theorem~1.4]{GargeMakadiyaNormal}, in arXiv:2502.04766v2,
cannot be used in this generality.
Also $\operatorname{Tr}_\theta(q,v)=(3q,0)$, so every linear
automorphism of $V$ preserves the trace, whereas not all normalize
$\langle A\rangle$.

\smallskip
\noindent\emph{The argument for maximal normal subgroups.}
For a proper $H\lhd E(R)$, its long-root parameters generate a
proper invariant ideal contained in some $I_i$.
We construct a regular function $F$ with
$x_f(F(h))\in H$ for every $h\in H$ and $F(\dot w)=-1$ for
a fixed Weyl element. If $H$ had nontrivial image in the simple
quotient $E(K_i)$, a lift of $\dot w$ in $H$ would contradict
these two properties.

\begin{lemma}[The long-root level]
\label{lem:triality-long-root-level}
For $H\lhd E(R)$, put
$$
        \mathfrak a_0(H)=\{v\in R_0\mid x_a(v)\in H\}.
$$
It is an ideal of $R_0$ and the parameter level of $H$ on every
long root. If $H$ is proper, then $R\mathfrak a_0(H)$ is a proper
invariant ideal of $R$.
\end{lemma}

\begin{proof}
Additivity and Weyl transport identify the long-root levels, and
$[x_a(v),x_e(w)]=x_f(vw)$ proves closure under multiplication
by $w\in R_0$.
If $1\in\mathfrak a_0(H)$, all long-root groups lie in $H$.
The relation
$$
 [x_a(1),x_b(t)]
   =x_c(t)x_d(t^\#)x_e(N_\theta(t))x_f(-N_\theta(t))
$$
then gives $Q(t)=x_c(t)x_d(t^\#)\in H$.
Since $(-t)^\#=t^\#$, cancellation gives
$$
        Q(t/2)Q(-t/2)^{-1}=x_c(t).
$$
Weyl transport supplies all short-root groups, so $H=E(R)$.
Finally, applying $\operatorname{Tr}_\theta/3$ to a hypothetical
$1=\sum_k r_kv_k$, $v_k\in\mathfrak a_0(H)$, would give
$$
        1=\sum_k\frac{\operatorname{Tr}_\theta(r_k)}3v_k
            \in\mathfrak a_0(H).
$$
Thus extension to $R$ preserves properness; its generators are fixed,
so the extended ideal is invariant.
\end{proof}

\smallskip
\noindent\emph{A universal element.}
Work over the coordinate ring
$$
        \mathbf G=G_{\mathrm{ad}}(D_4)_{\mathbb Z[1/6]},\qquad
        C=\mathbb Z[1/6][\mathbf G],\qquad
        \theta_C=(\rho^*)^{-1}.
$$
The universal point $\mathbf g\in\mathbf G(C)$ is fixed by
$\rho\theta_C$, and evaluation at $h\in G(R)$ is equivariant.
Fixed regular functions therefore evaluate in $R_0$.
Put
$$
 n_\gamma=x_\gamma(1)x_{-\gamma}(-1)x_\gamma(1)
       \quad(\gamma=a,b),\qquad \dot w=(n_bn_a)^3.
$$
This $\rho$-fixed representative of the longest absolute Weyl
element satisfies $\dot w^{-1}x_a(t)\dot w=x_{-a}(-t)$.

\begin{lemma}[A universal long-root parameter]
\label{lem:triality-universal-long-root-parameter}
There is $F\in C^{\theta_C}$ with $F(\dot w)=-1$ such that
$$
        x_f(F(h))\in H\qquad(h\in H\lhd E(R))
$$
for every coefficient pair $(R,\theta)$.
\end{lemma}

\begin{proof}
We use the universal-element construction of
\cite[Section~5]{StepanovUnipotents}, but extract the parameter
inside a two-root subgroup rather than the whole unipotent radical.

\smallskip
\noindent\emph{The two long-root directions.}
Let $\mathbf P$ be the standard absolute parabolic with Levi roots
$\{\pm a\}=\{\pm\alpha_2\}$. The long roots
$$
        e=\alpha_1+\alpha_2+\alpha_3+\alpha_4,\qquad
        f=e+\alpha_2
$$
are fixed by $\rho$. Their subgroup
$\mathbf V=\mathbf X_e\mathbf X_f\simeq\mathbb G_a^2$
is normalized by $\mathbf P$.
Indeed, the radical roots are the positive absolute roots other than
$a$; none has a root sum with $e$ or $f$. Thus the radical centralizes
$\mathbf V$, while the torus normalizes it and the rank-one Levi acts by
$$
\begin{aligned}
 x_a(r)x_e(p)x_f(q)x_a(-r)&=x_e(p)x_f(q+rp),\\
 x_{-a}(r)x_e(p)x_f(q)x_{-a}(-r)&=x_e(p+rq)x_f(q).
\end{aligned}
$$
These are integral root identities, so normalization holds over
every base ring. In particular,
$$
        [x_{-a}(1),x_e(p)x_f(q)]=x_e(q).
$$

\smallskip
\noindent\emph{Clearing denominators.}
The ring $C$ is a domain, since it is flat over $\mathbb Z[1/6]$
with integral generic fiber.
In the invariant cell
$\Omega=\mathbf U^+\mathbf T\mathbf U^-\dot w$, choose
$D(s_0)\subseteq\Omega$ with $s_0(\dot w)=1$.
Such an $s_0$ exists because the ideal of the complement restricts
to the unit ideal along the section $\dot w$.
Set $s=s_0\theta_C(s_0)\theta_C^2(s_0)$.
Then $s$ is fixed, $s(\dot w)=1$, and over $C_s$ the universal
point has a unique factorization
$$
        \mathbf g=utv\dot w,\qquad
        u\in\mathbf U^+(C_s),\quad t\in\mathbf T(C_s),\quad
        v\in\mathbf U^-(C_s).
$$
All three factors are twisting-fixed by uniqueness.
For sufficiently large $n$,
$$
        \xi=u x_a(s^n)u^{-1}\in\mathbf U^+(C).
$$
Indeed, the finitely many absolute coordinates of
$u x_a(X)u^{-1}$ are polynomials in $X$ over $C_s$ with zero
constant term, so substituting $X=s^n$ clears their denominators.
Twisting invariance descends along $C\hookrightarrow C_s$;
collecting coordinates by root orbits gives $\xi\in E(C)$.

\smallskip
\noindent\emph{Extracting the regular function.}
Put $\mathbf z=\mathbf g^{-1}\xi\mathbf g\xi^{-1}$.
Since $x_a(s^n)$ belongs to the Levi of $\mathbf P^-$ and
$\dot w^{-1}\mathbf P^-\dot w=\mathbf P$, over $C_s$ we have
$$
 \mathbf g^{-1}\xi\mathbf g
   =\dot w^{-1}v^{-1}t^{-1}x_a(s^n)tv\dot w\in\mathbf P(C_s).
$$
Also $\xi^{-1}\in\mathbf P(C)$. Closedness of $\mathbf P$ and
$C\hookrightarrow C_s$ therefore give $\mathbf z\in\mathbf P(C)$.
Normalization of $\mathbf V$ now gives the stronger two-factor identity
$$
        [\mathbf z,x_e(1)]=x_e(P)x_f(F),\qquad
        P,F\in C^{\theta_C}.
$$
The coordinates are regular over $C$, and fixed because the two
root maps and the commutator are twisting-fixed.
For $h\in H$, normality and $\xi(h)\in E(R)$ give
$\mathbf z(h)\in H$. A second commutator yields
$$
        [x_{-a}(1),[\mathbf z(h),x_e(1)]]=x_e(F(h))\in H.
$$
Long-root Weyl transport gives $x_f(F(h))\in H$.
The universal point itself need not be elementary: normality is
used only after evaluation, with conjugators in $E(R)$.

At $\dot w$, the cell factors and $s$ equal $1$, so
$\xi(\dot w)=x_a(1)$ and
$\mathbf z(\dot w)=x_{-a}(-1)x_a(-1)$.
The displayed rank-one action gives
$$
        [x_{-a}(-1)x_a(-1),x_e(1)]=x_e(1)x_f(-1),
$$
and hence $F(\dot w)=-1$.
\end{proof}

\begin{lemma}[Residual simplicity]
\label{lem:triality-residual-simplicity}
For every orbit $\mathcal O_i$, reduction $E(R)\to E(K_i)$ is
surjective and $E(K_i)$ is nontrivial and simple.
\end{lemma}

\begin{proof}
Surjectivity is Lemma~\ref{lem:surjectivity-reduction}.
For a fixed residue field, use
\cite[Lemma~6.1]{BuninaMakadiyaTriality} for split $G_2$ or
genuine triality. For a split residue algebra, projection identifies
the group with $E_{\mathrm{ad}}(D_4,k)$; apply
\cite[Corollary~4.5]{StepanovUnipotents}.
A root element of parameter $1$ proves nontriviality.
\end{proof}

\begin{theorem}[Maximal normal subgroups in triality]
\label{thm:triality-maximal-normal-subgroups}
Every proper normal subgroup of $E(R)$ is contained in some $N_i$.
The assignment $i\mapsto N_i$ is a bijection from $\mathcal I$ to
the maximal proper normal subgroups of $E(R)$.
\end{theorem}

\begin{proof}
For proper $H\lhd E(R)$, the ideal $J=R\mathfrak a_0(H)$ is
proper and invariant. Choose a maximal ideal containing it;
its orbit intersection $I_i$ also contains $J$.
If $(q_i)_*(H)\ne1$, residual simplicity and surjectivity give
$(q_i)_*(H)=E(K_i)$. Take $h\in H$ reducing to $\dot w$.
The universal lemma and long-root Weyl transport give
$$
        F(h)\in\mathfrak a_0(H)\subseteq I_i,\qquad
        q_i(F(h))=F(\dot w)=-1,
$$
a contradiction. The second equality holds because $F$ is regular.
Thus $H\subseteq N_i$.
Each $N_i$ is maximal normal, since its quotient is nontrivial
simple, and the containment just proved exhausts the maximal
normal subgroups. Finally, $N_i\cap X_b(R)=x_b(I_i)$ and
Lemma~\ref{lem:maximal-invariant-ideals} give distinctness.
\end{proof}

This proves the triality case of
Proposition~\ref{prop:maximal-normal-subgroups} and supplies
$\varphi^\sharp$ and the residual isomorphisms of Section~2.
The argument uses neither residual standardness nor an automorphism
classification.

\subsection{Residual standardness}
\label{subsec:triality-residual-standardness}

The field isomorphism theorems give geometric conjugacy of the
constant frames. Frame rigidity descends the conjugator, and
coefficient recovery then determines the coefficient isomorphism.
No cubic orientation needs to be chosen beforehand.

\begin{proposition}[Residual triality standardness]
\label{prop:residual-triality-standardness}
For residue algebras $(K_1,\theta_1),(K_2,\theta_2)$, every
isomorphism $\alpha:E(K_1)\xrightarrow{\sim}E(K_2)$ has the form
$$
        \alpha=i_g\circ\Phi_\nu,\qquad
        g\in G(K_2),\qquad \nu:K_1\xrightarrow{\sim}K_2,
$$
where $\nu$ preserves the trace. The two residue algebras are either
both fixed or both split.
\end{proposition}

\begin{proof}
Put $F_i=K_i^{\theta_i}$. The groups are
$E(K_i)=\mathbf H_i(F_i)^+$, generated by rational unipotent
radicals, with the following possibilities:
$$
\begin{array}{c|c|c}
 (K_i,\theta_i)&\mathbf H_i&\operatorname{rank}_{F_i}\mathbf H_i\\\hline
 K_i=F_i,\ \theta_i=1&G_2\text{ split}&2\\
 {}[K_i:F_i]=3&{}^3D_4\text{ quasi-split}&2\\
 K_i=F_i^3&D_4\text{ split}&4.
\end{array}
$$
The middle extension is cyclic Galois and defines the pinned form;
the last identification is
Lemma~\ref{lem:split-product-identification}.
The two groups, and hence the fields $F_i$, are either both finite
or both infinite.

\smallskip
\noindent\emph{Geometric conjugacy.}
For finite fields, \cite[Theorem~37]{SteinbergLectures} separates
the three cases and determines the field size; no exceptional
isomorphism occurs in the allowed characteristics.
After the field identification, and a pinned transposition if cubic
orientations differ, Theorems~30 and~36 of that reference leave
an inner--diagonal factor, hence geometric conjugation.

For infinite fields, use the Borel--Tits theorem and its isomorphism
complement as in Section~\ref{sec:residual-standardness}; see
\cite[Theorem~1.3 and Section~3.9]{SteinbergAbstractHomomorphisms}.
They give an isomorphism $F_1\to F_2$ and a special isogeny of the
corresponding forms. The target is adjoint, the groups are absolutely
simple and isotropic, and the image is Zariski dense.
The central factor is trivial, and special isogenies in types
$D_4,G_2$ are central when $2,3$ are invertible
\cite[Section~2.5]{SteinbergAbstractHomomorphisms}.
Thus the adjoint forms are isomorphic; absolute type and relative
rank distinguish the rows.
Over an algebraic closure the algebraic isomorphism is inner
followed by a pinned diagram map
\cite[Proposition~7.1.6]{ConradReductive}.
For genuine triality, extend the recovered field map to an embedding
of the cubic splitting field into that closure.

Every pinned diagram permutation fixes the signed simple frame
$$
 u_{\pm b}=x_{\pm\alpha_4}(1)x_{\pm\alpha_1}(1)x_{\pm\alpha_3}(1),
 \qquad u_{\pm a}=x_{\pm\alpha_2}(1),
$$
since the outer factors commute. It therefore fixes the entire
frame, obtained by the same Weyl words. Field transfer fixes it
as well. For split $G_2$, the inner conjugator lies in the pinned
$G_2\subseteq G_{\mathrm{ad}}(D_4)$.
Consequently, all cases give a geometric transporter of the same
indexed frame; unlike in ${}^2D_4$, no distinct graph frames remain.

\smallskip
\noindent\emph{Descent and recovery.}
Let $k=K_2$ in the field case or $K_2=k^3$ in the split case,
and use inclusion or first projection
$p:G(K_2)\to G_{\mathrm{ad}}(D_4,k)$.
The transporter from the constant frame to
$(p(\alpha(x_\lambda(1))))_\lambda$ has a point over $\overline k$
and at most one point over every $k$-algebra by frame rigidity.
Its pullbacks to $\overline k\otimes_k\overline k$ agree, so
affine faithfully flat descent gives
$c\in G_{\mathrm{ad}}(D_4,k)$, without a perfectness assumption on $k$.
For a field target, twisting gives another transporter, hence
$c\in G(K_2)$; set $g=c$.
For a split target, set $g=\Delta_\rho(c)$.
Now $i_{g^{-1}}\alpha$ fixes the frame, and
Proposition~\ref{prop:global-triality-coefficient-recovery} gives
its trace-preserving coefficient isomorphism $\nu$.
Here $\Phi_\nu$ denotes the existing normalized isomorphism;
no converse construction has been assumed.
\end{proof}

\subsection{Relative local normalization and coefficient isomorphisms}
\label{subsec:triality-relative-local-results}

We combine the fixed-target theorem with the common fixed-to-split
lemma. The converse is stated for the isomorphisms between local
models needed to construct global coefficient automorphisms.

\begin{theorem}[Relative local triality theorem]
\label{thm:relative-local-triality-normal-form}
Let $S$ be an arbitrary coefficient pair, $A$ a triality local model,
and $\Psi:E(S)\to E(A)$ a homomorphism. Suppose that
$g_{\mathrm{res}}\in G(A_{\mathrm{res}})$ and a trace-preserving
unital $\nu_{\mathrm{res}}:S\to A_{\mathrm{res}}$ satisfy
$$
 (q_A)_*(\Psi(x_\lambda(t)))
   =g_{\mathrm{res}}x_\lambda(\nu_{\mathrm{res}}(t))g_{\mathrm{res}}^{-1}
       \qquad(\lambda\in\Phi_\rho,\ t\in S_\lambda).
$$
There are a unique frame transporter $g_A\in G(A)$ lifting
$g_{\mathrm{res}}$ and a unique trace-preserving unital
$\nu:S\to A$ lifting $\nu_{\mathrm{res}}$ such that
$$
        i_{g_A^{-1}}\Psi(x_\lambda(t))=x_\lambda(\nu(t)).
$$
We denote this normalized homomorphism by $\Phi_\nu$;
neither injectivity nor surjectivity of $\Psi$ is assumed.

Conversely, every trace-preserving ring isomorphism
$\nu:S\xrightarrow{\sim}A$ between local models induces an
elementary group isomorphism with these root rules in both the
adjoint and simply connected forms. These isomorphisms commute
with the central isogeny and respect composition.
\end{theorem}

\begin{proof}
\smallskip
\noindent\emph{Normalization.}
For a fixed target, the residual conjugator lifts by
\cite[Proposition~2.2]{BuninaMakadiyaTriality}.
Its removal prepares the residual frame, so Theorem~4.3 of that
paper gives the transporter and coefficient map for the arbitrary
source $S$.

For a split target $A=B^3$, put $k=B/\operatorname{Jac}(B)$ and
$S_0=S^{\theta_S}$. Trace preservation sends $S_0$ into the fixed
diagonal of $k^3$, so
$$
        \nu_{\mathrm{res}}(s)=(f(s),f(s),f(s))\qquad(s\in S_0)
$$
for a unital $f:S_0\to k$.
Restricting $\Psi$ along $E_0=E_\rho(D_4,S_0)\to E(S)$
therefore gives exactly the residual hypothesis of
Lemma~\ref{lem:fixed-to-split-frame-normalization} with $m=3$.
The fixed source $E_0$ is perfect: the long-root $A_2$-relations
put its long-root groups in $[E_0,E_0]$, and
Lemma~\ref{lem:triality-long-root-level} then gives $[E_0,E_0]=E_0$.
The fixed-target input was just established. The common lemma
supplies $g_A$, and
Proposition~\ref{prop:global-triality-coefficient-recovery} applied
to $i_{g_A^{-1}}\Psi$ supplies $\nu$.
In both cases, comparison on a short root gives
$q_A\nu=\nu_{\mathrm{res}}$.
Frame rigidity makes $g_A$ unique, and one short-root
parameterization makes $\nu$ unique.

\smallskip
\noindent\emph{The converse.}
An isomorphism of local models joins two fixed models or two split
models. For fixed models, apply
\cite[Theorem~5.12]{BuninaMakadiyaTriality} to $\nu$ and
$\nu^{-1}$. The resulting elementary homomorphisms have inverse
root rules. For the simply connected form, transfer them through
the root-preserving elementary isomorphisms of the central isogeny
in \cite[Proposition~6.7]{BuninaMakadiyaTriality}.

For split models $S=B_1^3$, $A=B_2^3$, the isomorphism permutes
primitive idempotents and carries the fixed diagonal onto itself.
It therefore has the form
$$
        \nu(t_0,t_1,t_2)
          =\bigl(f(t_{\pi(0)}),f(t_{\pi(1)}),f(t_{\pi(2)})\bigr)
$$
for one isomorphism $f:B_1\xrightarrow{\sim}B_2$ and $\pi\in S_3$.
In the coherent simple short-root chart,
$$
        \operatorname{pr}_0x_{\pm b}(t_0,t_1,t_2)
          =x_{\pm\alpha_4}(t_0)x_{\pm\alpha_3}(t_1)
                                      x_{\pm\alpha_1}(t_2).
$$
Put $(\beta_0,\beta_1,\beta_2)=(\alpha_4,\alpha_3,\alpha_1)$
and choose the pinned diagram map $d_\pi$ with
$d_\pi(\beta_{\pi(i)})=\beta_i$ and $d_\pi(\alpha_2)=\alpha_2$.
Then
$$
        (\operatorname{pr}_{0,A})^{-1}(d_\pi)_*f_*
                                      \operatorname{pr}_{0,S}
$$
has the prescribed signed simple root rules. It fixes the constant
Weyl words, so coherence gives every root rule.
The split identification works in both isogeny forms, and pinned
diagram maps preserve both the root and weight lattices.
Thus these constructions commute with the central isogeny;
composition is checked on root generators.
\end{proof}

\begin{corollary}[Base-change compatibility]
\label{cor:relative-local-triality-base-change}
For an equivariant $f:A\to T$ in the preceding theorem,
$f_*\Psi$ has the unique frame transporter $f_*(g_A)$ and
normalized root rule
$$
        f_*\Phi_\nu=\Phi_{f\nu}.
$$
No locality of $T$ is needed. The source remains $E(S)$ and the
right-hand side denotes the existing composite. This applies in
particular to equivariant localizations and faithfully flat pullbacks.
\end{corollary}

\begin{proof}
Apply $f_*$ to the transport identities and root rules.
The composite $f\nu$ preserves the trace; frame rigidity and
injectivity of a short-root parameterization give uniqueness.
\end{proof}

\subsection{Global coefficient actions and classification}
\label{subsec:global-triality-classification}

Local coefficient isomorphisms detect all relations in a proposed
global coefficient action. For an arbitrary group automorphism,
local normalization and unique affine gluing supply its frame
transporter, leaving the coefficient action just constructed.

\begin{proposition}[Global triality coefficient automorphisms]
\label{prop:global-triality-coefficient-automorphisms}
For every $\mu\in\operatorname{Aut}_{\mathrm{tr}}(R,\theta)$, the
root prescription defines an automorphism $\Phi_\mu$ of $E(R)$.
It fixes the frame and satisfies
$$
        \Phi_\mu\circ\Phi_\nu=\Phi_{\mu\circ\nu},\qquad
        \Phi_\mu^{-1}=\Phi_{\mu^{-1}}.
$$
\end{proposition}

\begin{proof}
For a target orbit $i$, put $j=(\mu^\sharp)^{-1}(i)$.
Lemma~\ref{lem:triality-orbit-localization} gives a trace-preserving
isomorphism $\mu_j^A:A_j\xrightarrow{\sim}A_i$ with
$\mu_j^A\iota_j=\iota_i\mu$.
The converse of Theorem~\ref{thm:relative-local-triality-normal-form}
gives an actual elementary isomorphism $\Phi_{\mu_j^A}$.
Thus every relation in the global root generators, transformed by
$\mu$, holds over every $A_i$, and hence over $R$ by affine local
detection. The root prescription defines an endomorphism with
$$
        (\iota_i)_*\Phi_\mu=\Phi_{\mu_j^A}(\iota_j)_*.
$$
The same construction for $\mu^{-1}$ gives its inverse.
The composition law and fixation of the frame hold on generators.
\end{proof}

\begin{theorem}[Global frame normalization in triality]
\label{thm:global-triality-frame-normalization}
For every $\varphi\in\operatorname{Aut}E(R)$, there is a unique
$g\in G(R)$ with
$$
        \varphi(x_\lambda(1))=gx_\lambda(1)g^{-1}
        \qquad(\lambda\in\Phi_\rho).
$$
\end{theorem}

\begin{proof}
Theorem~\ref{thm:triality-maximal-normal-subgroups} gives the
permutation $\varphi^\sharp$ of maximal congruence kernels.
For a target orbit $i$, put $j=(\varphi^\sharp)^{-1}(i)$.
Residual standardness gives
$$
        \varphi_i^{\mathrm{res}}=i_{h_i}\Phi_{\nu_i}:
        E(K_j)\xrightarrow{\sim}E(K_i),
$$
where $\nu_i:K_j\to K_i$ preserves the trace.
With $A_{i,\mathrm{res}}=K_i$, the localized homomorphism
$(\iota_i)_*\varphi$ has residual conjugator $h_i$ and coefficient
map $\nu_iq_j$. The relative local theorem supplies a point of
the ordinary affine frame transporter $\mathcal T_\varphi(A_i)$.
Only this scheme is used from Section~2, not its ordinary local
hypotheses requiring equivariant coefficients.
Frame rigidity and Proposition~\ref{prop:unique-affine-gluing}
give a unique global transporter $g\in G_{\mathrm{ad}}(D_4,R)$.
Twisting gives another point of the same transporter scheme,
so uniqueness forces $\rho\theta(g)=g$.
\end{proof}

\begin{theorem}[Automorphisms of triality type ${}^3D_4$]
\label{thm:automorphisms-triality}
Every automorphism of $E(R)$ has a unique decomposition
$$
        \varphi=i_g\circ\Phi_\mu,\qquad
        g\in G(R),\quad
        \mu\in\operatorname{Aut}_{\mathrm{tr}}(R,\theta).
$$
Conversely, every such composition is an automorphism.
\end{theorem}

\begin{proof}
Remove the global frame transporter and apply
Proposition~\ref{prop:global-triality-coefficient-recovery}.
Conversely, every trace-preserving coefficient automorphism now
has a group action, and $G(R)$ normalizes $E(R)$.
In any decomposition, $\Phi_\mu$ fixes the frame, so $g$ is its
unique transporter. A short-root parameterization then determines
$\mu$ uniquely.
\end{proof}

For equivariant $\mu$, the action is $\Phi_\mu=\mu_*$;
when $\mu\theta=\theta^{-1}\mu$, it is the orientation-reversing
graph--ring action in the coherent coordinates.
The square-zero example shows that trace preservation allows more.
There is no separate graph factor. Full-group extensions are
constructed in Proposition~\ref{prop:triality-full-coefficient-isomorphisms}.
\medskip


\section{Arbitrary isogeny types and full twisted groups}
\label{sec:arbitrary-isogeny-types-full-groups}

We complete Theorems~\ref{thm:main-elementary}
and~\ref{thm:main-full}. There are two different passages:
lifting an adjoint elementary automorphism through a central isogeny,
and extending an elementary automorphism to the full group.
Perfectness makes the first lift unique; the second may be obstructed
or may have a central ambiguity. We construct the required actions,
prove characteristicity of the elementary subgroup, and determine
the image and kernel of restriction.

Throughout, $\Lambda_r\subseteq\Lambda\subseteq\Lambda_w$ is
$\rho$-invariant and the standing assumptions hold. Write
$$
 E_\Lambda(U):=E_{\Lambda,\sigma_U}(\Phi,U),\qquad
 G_\Lambda(U):=G_{\Lambda,\sigma_U}(\Phi,U),\qquad
 \sigma_U=\rho\theta_U.
$$
As before, $E(U),G(U)$ denote the adjoint twisted groups, whereas
$G_{\mathrm{ad}}(\Phi,U)$ is the ambient split group.
The central isogeny
$$
        q_\Lambda:G_\Lambda(\Phi)\longrightarrow G_{\mathrm{ad}}(\Phi)
$$
commutes with twisting and base change. Its maps on twisted groups
are denoted by $q_{\Lambda,*}$, and we put
$$
 Z_\Lambda(R):=\ker\bigl(q_{\Lambda,*}:G_\Lambda(R)\to G(R)\bigr).
$$
This kernel is central. Surjectivity will be used on elementary
groups, not assumed on full-group points.

\subsection{Central quotients and the adjoint action}
\label{subsec:central-isogenies-elementary-groups}

The elementary central quotient and perfectness control both
uniqueness of lifts and preservation of the elementary subgroup.

\begin{proposition}[The adjoint central quotient]
\label{prop:adjoint-central-quotient}
The central isogeny gives an exact sequence
$$
 1\longrightarrow Z(E_\Lambda(R))\longrightarrow E_\Lambda(R)
 \xrightarrow{\ q_{\Lambda,*}\ }E(R)\longrightarrow1.
$$
\end{proposition}

\begin{proof}
Every relative root element lifts with the same parameter, so the
map is surjective and has central kernel. Conversely, the image
of a central element centralizes the adjoint frame and is one by
Section~\ref{sec:rigidity-standard-frame}.
\end{proof}

\begin{lemma}
\label{lem:perfectness-arbitrary-isogeny}
The group $E_\Lambda(R)$ is perfect.
\end{lemma}

\begin{proof}
Use the rank-two commutator identities of
\cite[Corollary~6.6]{GargeMakadiyaNormal}. They hold in every
isogeny form under the standing assumptions, specializing to the
folded split identities when the coefficient operator is trivial.
This argument uses no ideal-level normal structure and applies to
triality as well.
\end{proof}

\begin{corollary}[Uniqueness over the adjoint quotient]
\label{cor:uniqueness-over-adjoint-quotient}
An adjoint elementary automorphism has at most one lift to
$E_\Lambda(R)$ compatible with $q_{\Lambda,*}$.
More generally, two homomorphisms from a perfect group to
$G_\Lambda(R)$ with the same composite with $q_{\Lambda,*}$ are equal.
\end{corollary}

\begin{proof}
For such homomorphisms $f_1,f_2$, the map
$x\mapsto f_1(x)f_2(x)^{-1}$ takes values in the central kernel
and is a homomorphism. Perfectness makes it trivial.
\end{proof}

\smallskip
\noindent\emph{The adjoint action.}
Conjugation on the split group scheme descends through its central
quotient to an action
$$
 a_\Lambda:G_{\mathrm{ad}}(\Phi)\times G_\Lambda(\Phi)
       \longrightarrow G_\Lambda(\Phi),\qquad
 a_\Lambda(q_\Lambda(h),x)=hxh^{-1};
$$
see \cite[Theorem~7.1.9]{ConradReductive}.
Write $a_g(x)=a_\Lambda(g,x)$. Naturality gives
$$
 a_ga_h=a_{gh},\qquad
 q_{\Lambda,*}(a_g(x))=gq_{\Lambda,*}(x)g^{-1},\qquad
 \sigma(a_g(x))=a_{\sigma(g)}(\sigma(x)).
$$
Thus $g\in G(R)$ acts on $G_\Lambda(R)$ without a lift of $g$
over $R$ or an extension of the coefficient operator.

To see that it preserves the elementary subgroup, observe that
$$
 q_{\Lambda,*}^{-1}(E(R))=E_\Lambda(R)Z_\Lambda(R),\qquad
 [q_{\Lambda,*}^{-1}(E(R)),q_{\Lambda,*}^{-1}(E(R))]=E_\Lambda(R).
$$
The first equality follows from elementary surjectivity, the second
from perfectness and centrality. Since $E(R)$ is normal in $G(R)$,
$a_g$ preserves this preimage and hence its commutator subgroup.
Denote the action on both groups by $i_g^\Lambda$.
It is faithful on $E_\Lambda(R)$ by adjoint frame rigidity.
In particular, $E_\Lambda(R)$ is normal in $G_\Lambda(R)$.

\smallskip
\noindent\emph{Strictly inner and diagonal factors.}
If $g=q_{\Lambda,*}(h)$ for $h\in G_\Lambda(R)$, the action
$i_g^\Lambda$ is conjugation by $h$, called \emph{strictly inner}.
The adjoint-induced and strictly inner automorphism groups are
$\operatorname{Inn}_{\mathrm{ad},\Lambda}(R,\theta)$ and
$\operatorname{Inn}_{\mathrm{str},\Lambda}(R,\theta)$.
The latter is normal in the former: conjugation by $i_g^\Lambda$
sends the action of $h$ to that of $a_g(h)$.
Consequently,
$$
 \operatorname{Diag}_\Lambda(R,\theta):=
 \operatorname{Inn}_{\mathrm{ad},\Lambda}(R,\theta)/
 \operatorname{Inn}_{\mathrm{str},\Lambda}(R,\theta)
 \simeq G(R)/q_{\Lambda,*}(G_\Lambda(R)).
$$
Choosing coset representatives separates $i_g^\Lambda$ into strictly
inner and diagonal factors, noncanonically. Keeping it as one factor
preserves uniqueness; in the adjoint form the diagonal quotient
is trivial. Ordinary equivariant coefficient automorphisms act on
all forms and commute with $q_{\Lambda,*}$.

\subsection{Lifting graph and coefficient factors}
\label{subsec:lifting-isogeny-factors}

Only the graph factor has a lattice obstruction. We first show that
an abstract elementary lift must satisfy the same condition as a
pinned group-scheme action, then construct the coefficient lifts.

\smallskip
\noindent\emph{Graph factors and the lattice.}
Put
$$
 \mathcal D_\Lambda:=
 \{\gamma\in\operatorname{Aut}(\operatorname{Dyn}(\Phi))
                                   \mid\gamma(\Lambda)=\Lambda\}.
$$

\begin{lemma}[Lifting graph automorphisms through a central isogeny]
\label{lem:lifting-graph-automorphisms-central-isogeny}
Let $p_\Lambda:G_{\mathrm{sc}}(\Phi)\to G_\Lambda(\Phi)$ and
$C_\Lambda=\ker p_\Lambda$.
\begin{enumerate}
\item
For a diagram automorphism $\gamma$,
$\gamma_{\mathrm{sc}}(C_\Lambda)=C_{\gamma(\Lambda)}$.
Its pinned action descends uniquely to $\gamma_\Lambda$ precisely
when $\gamma(\Lambda)=\Lambda$, and then
$q_\Lambda\gamma_\Lambda=\gamma_{\mathrm{ad}}q_\Lambda$.
\item
Let $\Phi=D_4$, let $S\ne0$ with $2\in S^\times$, and suppose
$\psi\in\operatorname{Aut}E_\Lambda(D_4,S)$ satisfies
$$
        q_{\Lambda,*}\psi=\gamma_{\mathrm{ad},*}q_{\Lambda,*}.
$$
Then $\gamma(\Lambda)=\Lambda$ and $\psi=(\gamma_\Lambda)_*$.
\end{enumerate}
\end{lemma}

\begin{proof}
The characters trivial on $C_\Lambda$ form $\Lambda$ inside
$\Lambda_w$. The pinned action therefore gives (1), by the
root-datum description of central quotients; see
\cite[Exp.~XXIII]{SGA3}.

For (2), the two homomorphisms
$\psi p_{\Lambda,*}$ and $p_{\Lambda,*}(\gamma_{\mathrm{sc}})_*$
from the untwisted $E_{\mathrm{sc}}(D_4,S)$ have equal composites
with $q_{\Lambda,*}$. This source is perfect by the ordinary rank-two identities,
so the central-discrepancy argument of
Corollary~\ref{cor:uniqueness-over-adjoint-quotient} makes them equal.
In particular,
\begin{equation}
\label{eq:lifted-graph-action-on-root-elements}
 \psi(x_{\alpha,\Lambda}(t))
     =p_{\Lambda,*}\bigl((\gamma_{\mathrm{sc}})_*
                            (x_{\alpha,\mathrm{sc}}(t))\bigr).
\end{equation}
This comparison uses the pinned action on the simply connected
form and does not assume that it descends to $G_\Lambda$.

Only intermediate lattices require a check. Their kernels
$C_\Lambda\simeq\boldsymbol\mu_2$ are distinguished by the three
central involutions $z_{13},z_{14},z_{34}$ of
\eqref{eq:D4-frame-central-word}. These are elementary rank-one
words and remain distinct over every residue field of $S$.
The equality of the two homomorphisms sends the canonical
nonidentity element of $C_\Lambda$ back into $C_\Lambda$.
Since the pinned action permutes the three involutions, it must
preserve this order-two subgroup scheme. Part (1) gives lattice
preservation, and \eqref{eq:lifted-graph-action-on-root-elements}
then identifies $\psi$ with the descended action.
\end{proof}

For ${}^2D_4$, use the triples of
Section~\ref{sec:graph-automorphisms-2D4} and put
$$
 \mathcal C_{\rho,\Lambda}(R,\theta):=
 \{\mathbf e\in\mathcal C_\rho(R,\theta)\mid
                 e_j=0\text{ if }\tau^j(\Lambda)\ne\Lambda\}.
$$
These \emph{$\Lambda$-admissible} data form a subgroup under the
same cyclic convolution. Their componentwise pinned lifts define
$\Gamma_{\mathbf e,\Lambda}$ on the full twisted group, with
$$
 q_{\Lambda,*}\Gamma_{\mathbf e,\Lambda}
                  =\Gamma_{\mathbf e}q_{\Lambda,*}.
$$
They preserve $q_{\Lambda,*}^{-1}(E(R))$ and hence its commutator
subgroup $E_\Lambda(R)$. The individual factors $e_1R,e_2R$ are
exchanged by $\theta$; neither needs its own involution.

\begin{corollary}[Detection of admissible graph data]
\label{cor:detection-admissible-D4-graph-data}
In type ${}^2D_4$, let $\mathbf e\in\mathcal C_\rho(R,\theta)$.
If $\Psi\in\operatorname{Aut}E_\Lambda(R)$ satisfies
$q_{\Lambda,*}\Psi=\Gamma_{\mathbf e}q_{\Lambda,*}$, then
$\mathbf e\in\mathcal C_{\rho,\Lambda}(R,\theta)$.
\end{corollary}

\begin{proof}
Only the exchanged factors need checking. If $e_1\ne0$, put
$d=e_1+e_2$, an invariant idempotent.
For $x\in E_\Lambda(dR)$, the component of $\Psi(x)$ on
$(1-d)R$ has trivial adjoint image. It is therefore a homomorphism
from a perfect group to a central subgroup and is trivial.
Applying the same argument to $\Psi^{-1}$ shows that $\Psi$
restricts to an automorphism of $E_\Lambda(dR)$.
First projection for $dR\simeq e_1R\times e_1R$ identifies this
restriction with an untwisted elementary automorphism lifting $\tau_*$;
the root-coordinate proof of
Lemma~\ref{lem:split-product-identification} works in every form.
The preceding lemma gives $\tau(\Lambda)=\Lambda$, and hence
$\tau^2(\Lambda)=\Lambda$ as well. If $e_1=0$, compatibility
gives $e_2=0$ and there is nothing to prove.
\end{proof}

The diagram group permutes the three proper intermediate $D_4$
lattices as it permutes the three nonzero classes in
$\Lambda_w/\Lambda_r$. The involution $\rho$ fixes exactly one,
with stabilizer $\langle\rho\rangle$.
For this form of ${}^2D_4$, therefore,
$\mathcal C_{\rho,\Lambda}(R,\theta)=\{(1,0,0)\}$.
In triality, the three intermediate lattices are cyclically permuted,
so only $\Lambda_r,\Lambda_w$ are invariant.

\smallskip
\noindent\emph{Triality coefficient factors.}
The local coefficient isomorphisms of
Theorem~\ref{thm:relative-local-triality-normal-form} exist in both
forms. They detect relations in the global root prescription.

\begin{proposition}[Triality coefficients in either isogeny form]
\label{prop:triality-coefficient-arbitrary-isogeny}
\label{prop:triality-elementary-coefficient-action}
For $\Lambda\in\{\Lambda_r,\Lambda_w\}$, every trace-preserving
ring isomorphism $\nu:S\xrightarrow{\sim}T$ induces a unique
isomorphism
$$
 \Phi_{\nu,\Lambda}:E_\Lambda(S)\xrightarrow{\sim}E_\Lambda(T),
 \qquad x_\lambda(t)\longmapsto x_\lambda(\nu(t)).
$$
These maps respect composition and the central isogeny, and have
inverse $\Phi_{\nu^{-1},\Lambda}$. They are natural in coefficient
diagrams with equivariant vertical maps. For equivariant $\nu$,
they are coefficientwise; for
$\mu\in\operatorname{Aut}_{\mathrm{tr}}(R,\theta)$, they give
$\Phi_{\mu,\Lambda}\in\operatorname{Aut}E_\Lambda(R)$.
\end{proposition}

\begin{proof}
For an automorphism $\mu$, repeat the relation-checking argument of
Proposition~\ref{prop:global-triality-coefficient-automorphisms}
in either form. For a target orbit $i$, put
$j=(\mu^\sharp)^{-1}(i)$.
The trace-preserving isomorphism $\mu_j^A:A_j\to A_i$ satisfies
$\mu_j^A\iota_j=\iota_i\mu$ and induces an actual elementary
isomorphism by Theorem~\ref{thm:relative-local-triality-normal-form}.
Thus every transformed relation holds over every $A_i$, hence over
$R$ by affine local detection. Applying the construction to $\mu^{-1}$
gives the inverse.

For $\nu:S\xrightarrow{\sim}T$, apply this result to the
trace-preserving automorphism
$$
        \eta(s,t)=(\nu^{-1}(t),\nu(s))
$$
of $S\times T$ with the product operator. Its elementary action
exchanges the factors and restricts to $\Phi_{\nu,\Lambda}$.
All remaining identities are checked on root generators.
\end{proof}

\smallskip
\noindent\emph{Exceptional unitary factors.}
For $n=3,4$ and $\delta\in\mathcal E_{\mathrm{glob}}(R,\theta)$,
interpret the elementary words of
Proposition~\ref{prop:global-exceptional-automorphisms} in form
$\Lambda$, using the signed coordinates of
Subsection~\ref{subsec:relative-root-subgroups}.
Denote their values by $X_{\lambda,\Lambda}^{(n),\delta}(t)$.
These words commute with central isogenies and equivariant
localization; at $\delta=0$ they equal $x_\lambda(t)$ in every form.

\begin{proposition}[Exceptional automorphisms in arbitrary isogeny type]
\label{prop:exceptional-arbitrary-isogeny}
The rule
$$
 \Theta_{\delta,\Lambda}^{(n)}(x_\lambda(t))
                         =X_{\lambda,\Lambda}^{(n),\delta}(t)
$$
defines a unique automorphism of $E_\Lambda(R)$ in type ${}^2A_n$.
It lifts $\Theta_\delta^{(n)}$, commutes with equivariant
localization, and satisfies
$$
 \Theta_{\varepsilon,\Lambda}^{(n)}\Theta_{\delta,\Lambda}^{(n)}
       =\Theta_{\delta+\varepsilon,\Lambda}^{(n)},\qquad
 (\Theta_{\delta,\Lambda}^{(n)})^{-1}
       =\Theta_{-\delta,\Lambda}^{(n)}.
$$
\end{proposition}

\begin{proof}
By Proposition~\ref{prop:global-exceptional-automorphisms},
$\delta_i$ is admissible on fixed local models and zero on split
ones. If $\delta_i\ne0$, the local involution is nontrivial and
\cite[Subsections~3.2--3.3]{BuninaMakadiya2ASmall} gives the
arbitrary-isogeny action. Otherwise the words give the identity.
Every transformed relation therefore holds over all $A_i$ and,
by affine local detection, over $R$. The same detection gives the
local addition and inverse laws globally. Compatibility follows
from the words, and root generation gives uniqueness.
\end{proof}

\subsection{The elementary classification}
\label{subsec:complete-adjoint-elementary-theorem}
\label{subsec:automorphisms-arbitrary-elementary-groups}

The preceding constructions lift the adjoint normal forms. The given
automorphism itself forces admissibility of its graph data.

\begin{proof}[Proof of Theorem~\ref{thm:main-elementary}]
An automorphism $\varphi$ induces $\overline\varphi$ on
$E_\Lambda(R)/Z(E_\Lambda(R))\simeq E(R)$.
Its unique decomposition is given by
Theorems~\ref{thm:local-global-automorphism},
\ref{thm:exceptional-unitary-automorphisms},
\ref{thm:automorphisms-type-2D4}, and~\ref{thm:automorphisms-triality},
with the ordinary local inputs verified in
Sections~\ref{sec:residual-standardness}
and~\ref{sec:relative-local-rigidity}.
All factors lift as above. In ${}^2D_4$, first write
$\overline\varphi=i_g\Gamma_{\mathbf e}\mu_*$.
The automorphism $(i_g^\Lambda)^{-1}\varphi\mu_*^{-1}$ lifts
$\Gamma_{\mathbf e}$, so
Corollary~\ref{cor:detection-admissible-D4-graph-data}
makes $\mathbf e$ admissible.

The product of the lifted factors has the same adjoint action
as $\varphi$, hence equals it by
Corollary~\ref{cor:uniqueness-over-adjoint-quotient}.
Conversely, all listed factors are automorphisms.
Uniqueness of the adjoint data and of their lifts gives uniqueness
in every isogeny form.
\end{proof}

\begin{corollary}[The classical elementary decomposition]
\label{cor:classical-elementary-decomposition}
Separating the adjoint-induced factor into strictly inner and diagonal
factors, and adjoining the stated coefficient, exceptional, or graph
factors, gives all elementary automorphisms.
No nontrivial central factor occurs. The diagonal factor is absent
in the adjoint form, and the nonneutral graph factor is absent for
the proper intermediate invariant lattice of ${}^2D_4$.
\end{corollary}

\begin{proof}
Use the diagonal quotient, perfectness, and the lattice stabilizer
above. Only the separation of the adjoint-induced factor is
noncanonical; the normalized data remain unique.
\end{proof}

In type ${}^2D_4$, the composition law of
Corollary~\ref{cor:automorphism-group-2D4} gives, in every form,
$$
 \operatorname{Aut}E_\Lambda(R)\simeq
 G(R)\rtimes\bigl(\mathcal C_{\rho,\Lambda}(R,\theta)
                         \rtimes\operatorname{Aut}(R,\theta)\bigr).
$$
The actions are $\mu\cdot\mathbf e=(\mu(e_j))_j$ and
$(\mathbf e,\mu)\cdot g=\Gamma_{\mathbf e}(\mu_*(g))$.
The group $G(R)$ here is adjoint, independently of $\Lambda$.

\subsection{Restriction to the elementary subgroup}
\label{subsec:passage-to-full-twisted-groups}

We prove characteristicity uniformly in all types, using the
maximal-normal-subgroup result already established. The key chain is
perfectness, an idempotent support ideal, and full support forced by
the centralizer. Restriction then leaves exactly a central ambiguity.

\begin{proposition}
\label{prop:characteristic-elementary-and-centralizer}
The subgroup $E_\Lambda(R)$ is characteristic in $G_\Lambda(R)$, and
$$
 C_{G_\Lambda(R)}(E_\Lambda(R))=Z(G_\Lambda(R))=Z_\Lambda(R).
$$
\end{proposition}

\begin{proof}
The adjoint image of an element centralizing $E_\Lambda(R)$
centralizes the frame and is one. Conversely, $Z_\Lambda(R)$ is
central. This proves both centralizer equalities.
Assume $R\ne0$ and abbreviate
$\mathcal G=G_\Lambda(R)$, $\mathcal E=E_\Lambda(R)$, and
$q=q_{\Lambda,*}$.

\smallskip
\noindent\emph{Finite normal generation.}
Let $M$ be the normal closure in $\mathcal E$ of its constant frame
$(u_\lambda)_\lambda$. Its adjoint image contains the adjoint frame
and cannot be proper: Proposition~\ref{prop:maximal-normal-subgroups}
would put it in some $N_i$, whereas the constant root elements
have nontrivial reduction. Thus
$\mathcal E=M\ker(q|_{\mathcal E})$.
The kernel is central, so perfectness gives
$\mathcal E=[M,M]\subseteq M$.
For $\alpha\in\operatorname{Aut}\mathcal G$, put
$$
        K=\alpha(\mathcal E),\qquad k_\lambda=\alpha(u_\lambda).
$$
Then $K$ is normal, perfect, and normally generated by the finitely
many $k_\lambda$, with $C_{\mathcal G}(K)=Z(\mathcal G)$.

\smallskip
\noindent\emph{Full support.}
Choose a faithful closed representation $\jmath$ of the ambient
split adjoint group over $\mathbb Z$, and put
$M_x=\jmath_R(q(x))\in\mathrm{GL}_N(R)$.
Let $\mathfrak a$ be the invariant ideal generated by the entries
of all $M_{k_\lambda}-I_N$ and their coefficient translates.
Normal generation gives $M_x\equiv I_N\pmod{\mathfrak a}$
for every $x\in K$.
Modulo $\mathfrak a^2$, matrices congruent to one modulo
$\mathfrak a$ form an abelian group. Perfectness therefore gives
$M_x\equiv I_N\pmod{\mathfrak a^2}$ for every $x\in K$.
Applying this to the generators and their translates yields
$\mathfrak a=\mathfrak a^2$.

This finitely generated ideal is $\epsilon R$ for an idempotent
$\epsilon$, fixed by $\theta$ because the ideal is invariant.
On $(1-\epsilon)R$, the image of $K$ has trivial adjoint image
and is central. Hence $E_\Lambda((1-\epsilon)R)$ centralizes $K$.
If this ring were nonzero, its elementary group would be nontrivial
and perfect, hence nonabelian, contradicting
$C_{\mathcal G}(K)=Z(\mathcal G)$.
Thus $\epsilon=1$ and $\mathfrak a=R$.

\smallskip
\noindent\emph{Containment and characteristicity.}
Put $L=q(K\cap\mathcal E)\lhd E(R)$.
If $L$ were proper, Proposition~\ref{prop:maximal-normal-subgroups}
would give $L\subseteq N_i$ for some $i$.
For $x\in\mathcal E$, the commutator $[k_\lambda,x]$ belongs
to $K\cap\mathcal E$. Surjectivity of the elementary quotient
and reduction therefore makes the residue of each $q(k_\lambda)$
centralize $E(K_i)$. Frame rigidity makes these residues one.
Every generator of $\mathfrak a$ then lies in the invariant ideal
$I_i$, contradicting $\mathfrak a=R$.
Consequently,
$$
 \mathcal E=(K\cap\mathcal E)\ker(q|_{\mathcal E}),\qquad
 \mathcal E=[K\cap\mathcal E,K\cap\mathcal E]\subseteq K.
$$
Applying the same argument to $\alpha^{-1}$ gives the reverse
inclusion, so $\alpha(\mathcal E)=\mathcal E$.
\end{proof}

\begin{definition}
\label{def:central-automorphism-full-group}
A \emph{central automorphism} of $G_\Lambda(R)$ is a bijective map
$$
 c_\zeta(x)=\zeta(x)x,\qquad
 \zeta:G_\Lambda(R)\longrightarrow Z(G_\Lambda(R))
$$
with $\zeta$ a group homomorphism.
\end{definition}

Centrality makes $c_\zeta$ a homomorphism; bijectivity is an additional
condition.

\begin{proposition}[The kernel of restriction]
\label{prop:kernel-restriction-full-group}
The restriction homomorphism
$$
 \operatorname{res}_\Lambda:\operatorname{Aut}G_\Lambda(R)
                  \longrightarrow\operatorname{Aut}E_\Lambda(R)
$$
has exactly the central automorphisms as its kernel.
Thus two extensions of the same elementary automorphism differ
by a central automorphism.
\end{proposition}

\begin{proof}
If $\psi$ fixes $E_\Lambda(R)$ pointwise, normality gives
$$
 \psi(x)e\psi(x)^{-1}=\psi(xex^{-1})=xex^{-1}
       \quad(x\in G_\Lambda(R),\ e\in E_\Lambda(R)).
$$
Hence $x^{-1}\psi(x)$ centralizes $E_\Lambda(R)$ and is central.
The map $\zeta(x)=\psi(x)x^{-1}$ is therefore a homomorphism to
the center, giving $\psi=c_\zeta$.
Conversely, every homomorphism to the center kills the perfect
$E_\Lambda(R)$, so every central automorphism lies in the kernel.
In particular, $\zeta$ factors through $G_\Lambda(R)/E_\Lambda(R)$.
\end{proof}

This identifies the ambiguity of extension, not its existence.
The adjoint, ordinary coefficient, and admissible graph actions
already exist on full groups. The two remaining constructions
are trace-preserving triality actions and exceptional unitary actions.

\subsection{Extension of triality coefficient actions}
\label{subsec:triality-full-extension}

Let $\Phi=D_4$, $\operatorname{ord}(\rho)=3$, and
$\Lambda\in\{\Lambda_r,\Lambda_w\}$. Put
$$
 H_\Lambda(U):=\mathbf T_\Lambda(U)^{\rho\theta_U},\qquad
 L_\Lambda(U):=E_\Lambda(U)H_\Lambda(U).
$$
The fixed torus normalizes the relative root groups, so $L_\Lambda(U)$
is a subgroup. We extend the action to this subgroup first, then
use finite invariant covers to glue values of arbitrary group elements.
No global equality $G_\Lambda(U)=L_\Lambda(U)$ is assumed.

\begin{lemma}[Factorization over an orbit neighborhood]
\label{lem:triality-EH-orbit-neighborhood}
If $U$ is semilocal and $\theta_U$ acts transitively on its maximal
ideals, then $G_\Lambda(U)=L_\Lambda(U)$.
In the simply connected form, $H_{\Lambda_w}(V)\subseteq E_{\Lambda_w}(V)$
over every coefficient pair $V$; in particular,
$G_{\Lambda_w}(U)=E_{\Lambda_w}(U)$.
\end{lemma}

\begin{proof}
\smallskip
\noindent\emph{The simply connected torus.}
Simple coroots form a cocharacter basis, so a fixed torus point
is $h_b(r)h_a(s)$, where $r\in V^\times$,
$s\in(V^{\theta_V})^\times$, and
$$
\begin{aligned}
 h_b(r)&=h_{\alpha_4}(r)h_{\alpha_1}(\theta_V(r))
                                h_{\alpha_3}(\theta_V^2(r)),\\
 h_a(s)&=h_{\alpha_2}(s).
\end{aligned}
$$
Both are elementary rank-one words:
$$
 w_\gamma(v)=x_\gamma(v)x_{-\gamma}(-v^{-1})x_\gamma(v),\qquad
 h_\gamma(v)=w_\gamma(v)w_\gamma(1)^{-1}.
$$
For a short root, this is the product of three commuting absolute
rank-one identities.

\smallskip
\noindent\emph{Semilocal factorization.}
For $J=\operatorname{Jac}(U)$, the residue algebra is a field with
an operator or a cyclic product of three fields.
The fixed-field factorization of
\cite[Proposition~2.2]{BuninaMakadiyaTriality} and the split
identification give
$$
        G_\Lambda(U/J)=E_\Lambda(U/J)H_\Lambda(U/J).
$$
The latter identification works in both forms by the same root formulas.
Root parameters lift, averaging fixed parameters as in
Lemma~\ref{lem:surjectivity-reduction}.
Torus parameters lift using the simple-root character basis in the
adjoint form and the simple-coroot cocharacter basis in the simply
connected form. Units lift modulo $J$; averaging a fixed unit's lift
preserves its invertible residue.

For $x\in G_\Lambda(U)$, choose $e_0h_0\in L_\Lambda(U)$
with the same residue. Then $z=(e_0h_0)^{-1}x$ reduces to one,
so its inverse image of the absolute big cell
$\mathbf U^+\mathbf T_\Lambda\mathbf U^-$ contains all maximal
points, hence all of $\operatorname{Spec}U$.
The big cell is an open immersion
\cite[Theorem~5.1.13]{ConradReductive}; its factorization of $z$
is unique and twisting-fixed. Collecting the unipotent coordinates
by triality orbits puts them in $E_\Lambda(U)$, so
$z\in L_\Lambda(U)$. The simply connected torus assertion then
gives the last equality.
\end{proof}

\begin{proposition}[Full triality coefficient isomorphisms]
\label{prop:triality-full-coefficient-isomorphisms}
\label{prop:global-extension-triality-coefficient-full-group}
Every trace-preserving isomorphism $\nu:S\xrightarrow{\sim}T$ has
a canonical extension of its elementary action
$$
 \widehat\Phi_{\nu,\Lambda}:G_\Lambda(S)\xrightarrow{\sim}G_\Lambda(T).
$$
These extensions commute with the central isogeny and satisfy
$$
 \widehat\Phi_{\kappa,\Lambda}\widehat\Phi_{\nu,\Lambda}
       =\widehat\Phi_{\kappa\nu,\Lambda},\qquad
 \widehat\Phi_{\nu,\Lambda}^{-1}
       =\widehat\Phi_{\nu^{-1},\Lambda}.
$$
If $f:S\to S'$ and $k:T\to T'$ are equivariant and
$\nu':S'\xrightarrow{\sim}T'$ preserves the traces with $k\nu=\nu'f$,
then
$$
        k_*\widehat\Phi_{\nu,\Lambda}
                     =\widehat\Phi_{\nu',\Lambda}f_*.
$$
For equivariant $\nu$, the extension is $\nu_*$.
\end{proposition}

\begin{proof}
We extend the elementary action of
Proposition~\ref{prop:triality-elementary-coefficient-action}
by the root-and-torus formulas of
\cite[Proposition~6.8]{BuninaMakadiyaTriality}, then glue its
values on finite local factorizations.

\smallskip
\noindent\emph{The torus rule.}
In the adjoint form, a fixed torus point is uniquely $t_U(u,v)$,
with $u\in U^\times$, $v\in(U^{\theta_U})^\times$, and
$$
\begin{aligned}
 \alpha_4(t_U(u,v))&=u,&\alpha_1(t_U(u,v))&=\theta_U(u),\\
 \alpha_3(t_U(u,v))&=\theta_U^2(u),&\alpha_2(t_U(u,v))&=v.
\end{aligned}
$$
In the coherent coordinates of Section~\ref{sec:triality}, its
positive-root conjugation multipliers are
$$
\begin{array}{c|cccccc}
 \gamma&b&a&c&d&e&f\\ \hline
 m_\gamma(t_U(u,v))&u&v&uv&u^\#v&N_{\theta_U}(u)v&N_{\theta_U}(u)v^2.
\end{array}
$$
For negative roots take inverses. These character formulas hold over
arbitrary coefficient rings; see
\cite[Subsection~6.3]{BuninaMakadiyaTriality}.
Define $\nu_H(t_S(u,v))=t_T(\nu(u),\nu(v))$.
Since $\nu$ preserves fixed coefficients, $\#$, and the norm, it
preserves all the multipliers. Hence
$$
 \Phi_{\nu,\Lambda_r}(hxh^{-1})
   =\nu_H(h)\Phi_{\nu,\Lambda_r}(x)\nu_H(h)^{-1}
       \quad(h\in H_{\Lambda_r}(S),\ x\in E(S)).
$$
Thus trace preservation, rather than commutation with $\theta$,
is exactly what these formulas require.

\smallskip
\noindent\emph{Agreement on $E\cap H$.}
For $h\in E(S)\cap H_{\Lambda_r}(S)$, the two values
$\Phi_{\nu,\Lambda_r}(h)$ and $\nu_H(h)$ induce the same
conjugation on $E(T)$ by surjectivity of $\Phi_{\nu,\Lambda_r}$.
Their quotient is therefore one by frame rigidity.
It follows that
$$
 \Xi_{\nu,\Lambda_r}(eh):=\Phi_{\nu,\Lambda_r}(e)\nu_H(h)
$$
is independent of factorization: if $eh=e'h'$, the discrepancy
$e'^{-1}e=h'h^{-1}$ belongs to the intersection, where the rules
agree. Conjugation compatibility proves multiplicativity, and
$\nu^{-1}$ gives the inverse.
For the simply connected form, $L_{\Lambda_w}=E_{\Lambda_w}$,
so set $\Xi_{\nu,\Lambda_w}=\Phi_{\nu,\Lambda_w}$.
Both constructions respect composition, central isogenies, and the
coefficient diagrams in the statement, directly on their generators:
$$
        k_*\Xi_{\nu,\Lambda}=\Xi_{\nu',\Lambda}f_*
                       \quad\text{on }L_\Lambda(S).
$$
In particular, their localization compatibility is established
before defining any full-group map.

\smallskip
\noindent\emph{Finite local factorizations and gluing.}
Fix $x\in G_\Lambda(S)$.
Over each orbit neighborhood $B_i=\Sigma_i^{-1}S$, the preceding
lemma factors $x$ into a finite root word and a fixed torus element.
Fixed denominators are cofinal: inverting
$N_{\theta_S}(t)\in\Sigma_i\cap S^{\theta_S}$ also inverts
$t\in\Sigma_i$.
Clear the finitely many denominators in the parameters and inverse
units, enlarging the denominator to retain the fixedness relations
and the factorization equality. Finite presentation gives
$$
        x_{s_i}\in L_\Lambda(S_{s_i}),\qquad
        s_i\in\Sigma_i\cap S^{\theta_S}.
$$
These opens cover every maximal ideal and hence the spectrum;
choose a finite subcover.

The images $\nu(s_i)$ are fixed, and the localized isomorphisms
$\nu_{s_i}:S_{s_i}\to T_{\nu(s_i)}$ preserve the trace.
Define
$$
        y_i=\Xi_{\nu_{s_i},\Lambda}(x_{s_i})
                            \in G_\Lambda(T_{\nu(s_i)}).
$$
On every intersection both values equal
$\Xi_{\nu_{s_is_j},\Lambda}(x_{s_is_j})$ by localization
compatibility. They glue to a unique absolute group point, whose
local twisting invariance places it in $G_\Lambda(T)$.
This is $\widehat\Phi_{\nu,\Lambda}(x)$.
Common refinements make it independent of the cover and factorizations.
Only finite root-and-torus data are spread out; no abstract local
full-group homomorphism is descended.

\smallskip
\noindent\emph{Compatibility of the glued maps.}
For two group elements, use a common cover on which both are in
$L_\Lambda$. Multiplicativity there gives global multiplicativity.
For an elementary element the one-set cover gives the original action.
Composition is checked on the image cover under the first coefficient
isomorphism, and $\nu^{-1}$ supplies the inverse.
Central-isogeny compatibility follows on the same factorizations.
In a coefficient diagram, the images $f(s_i)$ still give a cover,
so naturality of $\Xi$ proves naturality after gluing.
For equivariant $\nu$, the root and torus rules are coefficientwise,
hence so is the extension.
\end{proof}

Here \emph{canonical} specifies the root-and-torus construction.
An adjoint extension is unique over every ring by frame rigidity.
In the simply connected form, we only assert uniqueness over
local models.

\begin{corollary}[Local extension to the full triality group]
\label{cor:local-triality-full-extension}
For local models $A,B$, every trace-preserving isomorphism
$\nu:A\xrightarrow{\sim}B$ has a unique full-group extension
$\widehat\Phi_{\nu,\Lambda}$ of its elementary action in either
form, with the compatibilities above.
\end{corollary}

\begin{proof}
For the adjoint form, two extensions induce the same conjugation
on the target elementary group, whose centralizer is trivial.
For the simply connected form, full and elementary groups coincide
over both models by Lemma~\ref{lem:triality-EH-orbit-neighborhood}.
\end{proof}

\begin{proposition}[Extension of the non-exceptional factors]
\label{prop:extension-standard-factors-full-group}
The adjoint-induced inner, ordinary equivariant coefficient,
compatible $\Lambda$-admissible graph, and triality coefficient
factors extend canonically to the full twisted groups, preserving
the elementary subgroups and commuting with central quotients.
\end{proposition}

\begin{proof}
The first three actions were constructed above; the fourth is
Proposition~\ref{prop:triality-full-coefficient-isomorphisms}.
\end{proof}

\subsection{Exceptional unitary full groups}
\label{subsec:exceptional-unitary-full-groups}

It remains to extend the exceptional factor in ${}^2A_n$, $n=3,4$.
For simply connected ${}^2A_3$, it acts trivially away from its
parameter's support, and full-group elements are locally elementary
near that support. In the other forms, a sign-changing torus element
violates the conjugation compatibility required of an extension.

Write $G_{\mathrm{sc}}^{(3)},E_{\mathrm{sc}}^{(3)}$ for the
simply connected groups of type ${}^2A_3$. Their full group is
$$
 G_{\mathrm{sc}}^{(3)}(R)=\operatorname{SU}_4(R,\theta;Q_4)
   =\{x\in\operatorname{SL}_4(R)\mid xQ_4\overline x^{\,t}=Q_4\}.
$$

\begin{lemma}[Local exceptional images in $\operatorname{SU}_4$]
\label{lem:zariski-gluing-exceptional-SU4}
For $\delta\in\mathcal E_{\mathrm{glob}}(R,\theta)$ and
$x\in G_{\mathrm{sc}}^{(3)}(R)$, there is a finite principal cover
$\operatorname{Spec}R=\bigcup_iD(s_i)$, with $s_i\in R^\theta$,
such that on each member
$$
 \delta_{s_i}=0\quad\text{or}\quad
 x_{s_i}\in E_{\mathrm{sc}}^{(3)}(R_{s_i}).
$$
The local values
$$
 y_i=\begin{cases}
 \Theta_{\delta_{s_i},\Lambda_w}^{(3)}(x_{s_i}),
                    &x_{s_i}\in E_{\mathrm{sc}}^{(3)}(R_{s_i}),\\
 x_{s_i},&\delta_{s_i}=0
 \end{cases}
$$
are consistent and glue uniquely in $G_{\mathrm{sc}}^{(3)}(R)$,
independently of the cover.
\end{lemma}

\begin{proof}
Put $\mathfrak a=\operatorname{Ann}_R(\delta)$ and
$Z=\operatorname{Supp}_R(R\delta)=V(\mathfrak a)$.
The ideal $\mathfrak a$ is invariant.
For $\mathfrak p\notin Z$, finite prime avoidance gives
$a\in\mathfrak a\setminus(\mathfrak p\cup\theta(\mathfrak p))$.
Then $s=a\overline a$ is fixed, avoids $\mathfrak p$, and kills
$\delta$ after localization.

For $\mathfrak p\in Z$, the inclusion
$\mathfrak b_\theta(R)\subseteq\mathfrak a\subseteq\mathfrak p$
and Subsection~\ref{subsec:global-exceptional-parameters} give
$R/\mathfrak p\simeq\mathbb F_3$ and
$\theta(\mathfrak p)=\mathfrak p$.
Since $\delta_{\mathfrak p}\ne0$, the localized involution is
nontrivial, and the factorization in the proof of
\cite[Theorem~5.4]{BuninaMakadiya2ASmall} gives
$$
 G_{\mathrm{sc}}^{(3)}(R_{\mathfrak p})
                    =E_{\mathrm{sc}}^{(3)}(R_{\mathfrak p}).
$$
Spread the finite root factorization of $x_{\mathfrak p}$,
including its parameter relations and its equality with $x$, to a
principal neighborhood. Replacing the denominator by its product
with its conjugate gives fixed $s\notin\mathfrak p$ with $x_s$
elementary. These neighborhoods cover the spectrum; choose a finite
subcover.

For fixed denominators $s,t$, the elementary root formulas give
\begin{equation}
\label{eq:exceptional-localization-naturality}
 (\iota_{s,st})_*\Theta_{\delta_s,\Lambda_w}^{(3)}
     =\Theta_{\delta_{st},\Lambda_w}^{(3)}(\iota_{s,st})_*.
\end{equation}
Thus two elementary descriptions agree on intersections; if one
parameter vanishes, both descriptions there are the identity rule.
Affine gluing gives one absolute group point, locally twisting-fixed
and hence in $G_{\mathrm{sc}}^{(3)}(R)$.
Common refinements prove independence of the cover.
\end{proof}

\begin{proposition}[The exceptional extension to $\operatorname{SU}_4$]
\label{prop:exceptional-extension-simply-connected-A3}
The preceding rule defines a canonical extension
$\widehat\Theta_{\delta,\mathrm{sc}}^{(3)}$ of
$\Theta_{\delta,\Lambda_w}^{(3)}$ to $G_{\mathrm{sc}}^{(3)}(R)$.
It is an automorphism and satisfies
$$
 \widehat\Theta_{\varepsilon,\mathrm{sc}}^{(3)}
       \widehat\Theta_{\delta,\mathrm{sc}}^{(3)}
       =\widehat\Theta_{\delta+\varepsilon,\mathrm{sc}}^{(3)},\qquad
 (\widehat\Theta_{\delta,\mathrm{sc}}^{(3)})^{-1}
       =\widehat\Theta_{-\delta,\mathrm{sc}}^{(3)}.
$$
For $\mu\in\operatorname{Aut}(R,\theta)$,
$$
 \mu_*\widehat\Theta_{\delta,\mathrm{sc}}^{(3)}\mu_*^{-1}
                  =\widehat\Theta_{\mu(\delta),\mathrm{sc}}^{(3)}.
$$
\end{proposition}

\begin{proof}
For two elements, refine their covers for the same $\delta$.
On each member the parameter vanishes or both elements are
elementary, so local multiplicativity gives a group homomorphism.
The one-set cover for a global elementary element gives the
prescribed restriction.
For two parameters and one element, refine the covers for both:
either both parameters vanish or the element and its exceptional
images are elementary. The elementary addition law therefore gives
the global addition law, and $-\delta$ supplies the inverse.
An equivariant ring automorphism transports the covers and formulas,
proving the conjugation identity. The construction specifies the
extension; any other one differs by a central automorphism.
\end{proof}

\smallskip
\noindent\emph{The obstruction in the other forms.}
An extendible elementary automorphism must normalize the subgroup
of $\operatorname{Aut}E_\Lambda(R)$ formed by conjugations from
$G_\Lambda(R)$: an extension sends conjugation by $t$ to conjugation
by its image. The following torus elements test this necessary
compatibility.

For $A_3$, put $\Lambda_{\mathrm{mid}}=\Lambda_r+\mathbb Z\omega_2$
and $T_0=\operatorname{diag}(1,1,-1,-1)\in\operatorname{SL}_4(R)$.
Although $\sigma(T_0)=-T_0$, its image $t_\Lambda^{(3)}$ is fixed
for $\Lambda=\Lambda_{\mathrm{mid}},\Lambda_r$, because $-I_4$
is killed by those central quotients.
For $A_4$, $T=\operatorname{diag}(1,-1,1,-1,1)$ is itself a
simply connected twisted element; denote its image in either form
by $t_\Lambda^{(4)}$.

\begin{lemma}[Change of sign of the exceptional parameter]
\label{lem:sign-change-exceptional-full-groups}
For
$$
 (n,\Lambda)\in
 \{(3,\Lambda_{\mathrm{mid}}),(3,\Lambda_r),
                         (4,\Lambda_w),(4,\Lambda_r)\},
$$
conjugation by $t=t_\Lambda^{(n)}$ satisfies
$$
 i_t\Theta_{\delta,\Lambda}^{(n)}i_t^{-1}
                         =\Theta_{-\delta,\Lambda}^{(n)}
                         \quad\text{on }E_\Lambda(R).
$$
\end{lemma}

\begin{proof}
Conjugate the root formulas by $T_0$ or $T$, as in
\cite[Lemma~5.2]{BuninaMakadiya2ASmall}.
The polynomial identities use the annihilator relations of $\delta$
and hold over $R$. They prove the adjoint identity; uniqueness of
an elementary lift gives it in the other forms.
\end{proof}

\begin{proposition}[The exceptional extension obstruction]
\label{prop:obstruction-exceptional-extension-full-groups}
For any of these four pairs, $\Theta_{\delta,\Lambda}^{(n)}$
extends to an automorphism of $G_\Lambda(R)$ only if $\delta=0$.
\end{proposition}

\begin{proof}
Suppose an extension exists, and let $s$ be its value at
$t=t_\Lambda^{(n)}$. On the elementary subgroup, conjugation
compatibility and the sign-change lemma give
$$
 i_{st^{-1}}
   =\Theta_{\delta,\Lambda}^{(n)}i_t
                \Theta_{-\delta,\Lambda}^{(n)}i_t^{-1}
   =\Theta_{2\delta,\Lambda}^{(n)}.
$$
The left side is strictly inner and the right side a pure exceptional
factor. Uniqueness in Theorem~\ref{thm:main-elementary} forces
$2\delta=0$, hence $\delta=0$.
\end{proof}

\subsection{The full-group classification}
\label{subsec:final-automorphism-theorem}

The constructions above determine which elementary normal forms
extend. The kernel of restriction supplies the remaining factor.

\begin{proof}[Proof of Theorem~\ref{thm:main-full}]
For $\varphi\in\operatorname{Aut}G_\Lambda(R)$, characteristicity
gives $\alpha=\varphi|_{E_\Lambda(R)}\in\operatorname{Aut}E_\Lambda(R)$.
Apply Theorem~\ref{thm:main-elementary} to $\alpha$.
Outside the exceptional unitary types all its factors extend by
Proposition~\ref{prop:extension-standard-factors-full-group}.
In type ${}^2A_n$, write
$\alpha=i_g^\Lambda\Theta_{\delta,\Lambda}^{(n)}\mu_*$.
Then $(i_g^\Lambda)^{-1}\varphi\mu_*^{-1}$ extends
$\Theta_{\delta,\Lambda}^{(n)}$.
Such an extension exists for every $\delta$ in simply connected
${}^2A_3$ by
Proposition~\ref{prop:exceptional-extension-simply-connected-A3},
and requires $\delta=0$ in the other forms by
Proposition~\ref{prop:obstruction-exceptional-extension-full-groups}.
Conversely, composing the specified extensions of the factors gives
an extension whenever this criterion holds.

For an extendible $\alpha$, fix this product $\widehat\alpha$ in
the order of Theorem~\ref{thm:main-elementary}.
The automorphism $\widehat\alpha^{-1}\varphi$ fixes the elementary
subgroup pointwise, so
Proposition~\ref{prop:kernel-restriction-full-group} gives
$$
        \varphi=\widehat\alpha\circ c_\zeta.
$$
Every such product is an automorphism. Restriction determines
$\alpha$ and all its normalized data; once its specified extension
is fixed, $c_\zeta$ is unique. For the adjoint form, $Z(G(R))=1$,
so this factor is the identity.
\end{proof}

\begin{corollary}
\label{cor:full-adjoint-groups-nonexceptional}
For $\Lambda=\Lambda_r$, restriction is injective.
Outside ${}^2A_3,{}^2A_4$, it is an isomorphism
$\operatorname{Aut}G(R)\xrightarrow{\sim}\operatorname{Aut}E(R)$.
Thus the full adjoint group has the same normalized decomposition
as its elementary subgroup.
\end{corollary}

\begin{proof}
The kernel consists of central automorphisms and is trivial.
Outside the two exceptional types all elementary factors extend.
\end{proof}

In particular, every automorphism of
$\operatorname{SU}_4(R,\theta;Q_4)$ is uniquely
$$
        \varphi=i_g^{\Lambda_w}
          \circ\widehat\Theta_{\delta,\mathrm{sc}}^{(3)}
          \circ\mu_*\circ c_\zeta
$$
with the normalized data and specified extension of the two main
theorems. Its exceptional family is nontrivial precisely when
$\mathcal E_{\mathrm{glob}}(R,\theta)\ne0$.
The central factor records the ambiguity of full-group extension;
only a further separation of $i_g^{\Lambda_w}$ into strictly inner
and diagonal factors requires an additional choice.

\medskip


\section*{Use of AI-assisted tools}

The authors used ChatGPT (OpenAI) during the research and preparation
of this paper to assist with the development and review of mathematical
arguments, selected computational checks, the preparation of code for
symbolic calculations, and the organization and editing of the manuscript.
The authors retain full responsibility for the mathematical results and
proofs, the computations, and the final text.

\end{document}